\documentclass{article}

\usepackage[utf8]{inputenc}
\usepackage{hyperref}
\usepackage{enumitem}

\usepackage{array}

\usepackage[a4paper, total={6in, 9in}]{geometry}

\usepackage{tikz}
\usepackage{tikz-cd}
\usetikzlibrary{arrows}
\tikzcdset{arrow style=tikz,
diagrams={>={Stealth[round,length=4pt,width=6pt,inset=3pt]}}}

\usepackage{amsthm}

\usepackage{amsmath} 
\usepackage{amssymb}
\usepackage{amsfonts}
\usepackage{faktor}
\usepackage[T1]{fontenc}
\usepackage{xfrac}
\usepackage{dsfont}
\usepackage{color}
 \usepackage{todonotes}
\usepackage[english]{babel}
\usepackage{mathtools}

\usepackage{mathdots}

\newtheorem{theorem}{Theorem}
\newtheorem{definition}[theorem]{Definition}

\newtheorem{lemma}[theorem]{Lemma}
\newtheorem{notation}[theorem]{Notation}

\newtheorem{proposition}[theorem]{Proposition}

\newtheorem{example}[theorem]{Example}
\newtheorem{remark}[theorem]{Remark}
\newtheorem{corollary}[theorem]{Corollary}

   \newtheorem*{theoremA}{Theorem A}
   \newtheorem*{theoremB}{Theorem B}

\DeclareMathOperator{\ord}{ord}

\DeclareMathOperator{\length}{length}

\DeclareMathOperator{\chara}{char}

\DeclareMathOperator{\Ann}{Ann}

\DeclareMathOperator{\edim}{edim}
\DeclareMathOperator{\NZD}{NZD}
\DeclareMathOperator{\Frac}{Frac}
\DeclareMathOperator{\mult}{mult}
\DeclareMathOperator{\Hom}{Hom}

\DeclareMathOperator{\Tor}{Tor}
\DeclareMathOperator{\Ext}{Ext}

\numberwithin{theorem}{section}

\renewcommand{\mod}{\,\operatorname{mod}\,}

\title{The Milnor Number of One Dimensional Local Rings}
\author{Yotam Svoray}
\date{}

\AtEndDocument{\bigskip{\footnotesize%
  \textsc{ Department of Mathematics, University of Utah, UT 84112} \par  
  \textit{E-mail address}: \texttt{svoray@math.utah.edu} 
}}

\begin{document}

\maketitle

\begin{abstract}
    We present an analogue of the Milnor number for one dimensional local ring inspired by Milnor's formula, and we show that it satisfies analogous properties to those of the classical Milnor number of curves over a field. In addition, we study two analogues of the semigroup of values of a one dimensional ring and show how they relate to our Milnor number. Finally, we use these tools to classify one dimensional rings of finite Cohen-Macaulay type based upon their semigroups and relate these semigroups to those of the classical ADE singularities. 
\end{abstract}

  \tableofcontents

\section{Introduction}\label{sec:Intro3}
 
We study one dimensional Noetherian strictly Henselian reduced Nagata local rings in a manner similar to the study of algebraic curves over algebraically closed fields, and show that they have similar properties, behaviors, and classifications. We focus on two invariants and their analogues, the Milnor number of a curve and the semigroup of values of a curve.\\
 
Given a power series in two variables $f \in k[[x,y]]$ over an algebraically closed field $k$, the Milnor  number of $f$, denoted $\mu(f)$, is defined to be the $k-$dimension of the quotient ring $\frac{k[[x,y]]}{\langle \partial_x(f), \partial_y(f) \rangle}$. The Milnor number is an important invariant in the study of isolated singularities, defined by Milnor in~\cite{milnor1968singular}, which allows to understand their topology and algebraic behavior. For more information, see~\cite{greuel2007introduction, varchenko1985singularities}. Milnor also proved, using a topological argument, that in the case $k=\mathbb{C}$, we have an equality $\mu(f)=2\delta(f)-r(f)+1$, where $\delta(f)$ is the delta invariant of $f$ (defined to be the $k-$dimension of the quotient $\frac{\overline{\mathcal{O}_f}}{\mathcal{O}_f}$ with $\mathcal{O}_f=\frac{k[[x,y]]}{\langle f \rangle}$ and $\overline{\mathcal{O}_f}$ being its normalization) and $r(f)$ is the number of irreducible components, with an algebraic proof over arbitrary algebraically closed fields of characteristic zero given by Risler in~\cite{risler1971ideal}. For reduced complex curve singularities of arbitrary embedding dimension, Buchweitz and Greuel~\cite{buchweitz1980milnor} defined the Milnor number as $\dim_{\mathbb{C}}\left(\omega/d\mathcal{O}\right)$, where $\omega$ is the dualizing module, and proved that it is again equal to $2\delta-r+1$. Deligne in~\cite{deligne2006formule} proved that in the positive characteristic case, we have an inequality $\mu(f)\geq 2\delta(f)-r(f)+1$, and there has been much work on the study of the condition for which we have an equality in the positive characteristic case, see for example~\cite{melle2001pencils, nguyen2016invariants, boubakri2012invariants, greuel2012some, barroso2018milnor, barroso2016milnor}. In addition, in~\cite{barroso2022note}, the invariant $\overline{\mu}(f) :=2\delta(f)-r(f)+1$ was defined based upon Milnor's formula and studied. \\
 
Based upon the invariant $\overline{\mu}$, in Section~\ref{sec:milnor} we define an analogous invariant for one dimensional local rings and show that it behaves in a similar fashion to the behavior of the Milnor number over an algebraically closed field of characteristic zero. Over $\mathbb{C}$ our invariant is the Milnor number of Buchweitz and Greuel; the point is that the expression $2\delta-r+1$ makes sense for every ring in our class, in any (including mixed) characteristic, and does not depend on a choice of embedding. Unlike the traditional Milnor number, whose behavior in positive characteristic is subtle (see, for example, the discussion in the Introduction of~\cite{boubakri2009hypersurface}), our invariant is defined in the same way in every characteristic. \\

Our first main result shows that $\mu(R)$ detects the simplest singularities, satisfies an analogue of Hironaka's formula for the delta invariant of a union of curves, and is bounded from below in terms of the number of branches and the embedding dimension of $R$. Beyond the plane curve case some classical statements change: for instance, the sharp lower bound is $\mu(R)\geq r(R)-1$, with equality exactly for the $r(R)$ coordinate axes, while the classical bound $\mu \geq (r-1)^2$ holds for rings of embedding dimension at most two. More generally, we give a sharp lower bound for $\mu(R)$ in terms of $r(R)$ and $\edim(R)$, which interpolates between these two bounds, and we characterize its equality case by the position of the tangent directions of the branches.
 
\begin{theoremA}
    Let $r=r(R)$ and $e=\edim(R)$.
    \begin{enumerate}
        \item (Corollary~\ref{cor:reg}, Theorem~\ref{thm:Morse}, Proposition~\ref{cor:delta1}) $\mu(R)=0$ if and only if $R$ is a DVR, and $\mu(R)=1$ if and only if $R$ is a double point, i.e. $\hat{R}\cong\frac{S}{\langle xy\rangle}$ for a $2-$dimensional complete regular local ring $S$ with maximal ideal $\langle x,y\rangle$. If $\chara(\kappa)\neq 2$, then $\mu(R)=2$ if and only if either $\hat{R}\cong\frac{S}{\langle x^2+y^3\rangle}$ for such an $S$ (an $A_2$ singularity), or $\hat{R}\cong\frac{S'}{\langle xy,xz,yz\rangle}$ for a $3-$dimensional complete regular local ring $S'$ with maximal ideal $\langle x,y,z\rangle$.
        \item (Proposition~\ref{thm:branch}) We have
        \begin{equation*}
            \mu(R)-1=\sum_{i=1}^{r}\left(\mu\left(\frac{R}{\mathfrak{p}_i}\right)-1\right)+2\sum_{j=2}^{r} i(\mathfrak{p}_j,\mathfrak{p}_1\cap\cdots\cap\mathfrak{p}_{j-1}),
        \end{equation*}
        and if $e\leq 2$, the last sum equals $\sum_{i<j}i(\mathfrak{p}_i,\mathfrak{p}_j)$.
        \item (Corollary~\ref{cor:lower_bound}, Proposition~\ref{prop:axes}) $\mu(R)\geq r-1$, with equality if and only if $\hat{R}\cong\frac{S}{\langle x_ix_j\colon i\neq j\rangle}$ for an $r-$dimensional complete regular local ring $S$ with maximal ideal $\langle x_1,\dots,x_r\rangle$, i.e. $R$ is the $r$ coordinate axes.
        \item (Theorem~\ref{thm:edim_bound}, Proposition~\ref{prop:omp_edim2}) If $R$ is not regular, then 
        \begin{equation*}
            \mu(R)\geq 2\Phi(r,e)-r+1,
        \end{equation*}
        where $\Phi(r,e)=\sum_{d\geq 0}\max\left(0,\ r-\binom{d+e-1}{e-1}\right)$, with equality if and only if every branch $\frac{R}{\mathfrak{p}_i}$ is regular and the $r$ tangent directions of the branches are in generic position in $\mathbb{P}^{e-1}(\kappa)$ (Definition~\ref{def:tangent_directions}).  For $e=2$ it reads $\mu(R)\geq (r-1)^2$, and if $r\geq 2$, equality holds if and only if $\hat{R}\cong\frac{S}{\langle f\rangle}$ with $S$ as in the first item and $f=\prod_{i=1}^{r}(x+u_iy)$ for units $u_1,\dots,u_r \in S$ such that $u_i-u_j$ is a unit for every $i\neq j$.
        \item (Corollary~\ref{cor:gor_bound}) If $R$ is Gorenstein and $r\geq 3$, then $\mu(R)\geq r+1$.
        \item (Lemma~\ref{lem:rho_delta}) If $r=1$ and $R$ is residually rational, then $\mu(R)-\mu(R^{(1)})=2\length_R\left(\frac{R^{(1)}}{R}\right)$, where $R^{(1)}$ is the blow up of $R$ at $\mathfrak{m}$ (Definition~\ref{def:blowup}).
    \end{enumerate}
\end{theoremA}
 
In Section~\ref{sec:semigroup} we define two semigroup invariants, that are analogues of the classical semigroup of values of a curve and of a one dimensional local ring. The first, $\Gamma(R)$, which is a sub-semigroup of $\mathbb{N}$, is a generalization of the semigroup of values of a one dimensional domain, yet does not detect whether $R$ is a domain or not. The second, $\nu(R)$, is a sub-semigroup of $\mathbb{N}^{r(R)}$, where $r(R)$ is the number of minimal primes that $R$ has (which is the equivalent of ``the number of irreducible components of a curve'' for local rings), is more complicated to compute but encodes more information regarding $R$. We show how these semigroups are related to the Milnor number we defined and how we can use them to deduce information about the original ring; in particular, $\nu(R)$ determines $\delta(R)$ and whether $R$ is Gorenstein, and for Gorenstein rings we compute the conductor branch by branch in terms of Milnor numbers. There has been much research on the relationship between one dimensional rings and algebraic curves with their semigroups of values, for example, see~\cite{assi2020numerical, barroso2012approach, barroso2022note, castellanos2005semigroup, campillo1994gorenstein, d2025value, de1987semigroup, kunz1970value, zariski2006moduli}. \\
 
Finally, in Sections~\ref{sec:FCMT1}, \ref{sec:gorenstein}, and~\ref{sec:nonGorenstein} we use the Milnor number and the semigroups we define to relate between one dimensional Cohen-Macaulay local rings of finite Cohen-Macaulay type and the classical ADE singularities using a notion of ``equisingularity'', which is a generalized analogue of the ``equisingularity of a curve'' notion defined by Zariski in~\cite{zariski2006moduli}. There has been much research on one dimensional Cohen-Macaulay local rings of finite Cohen-Macaulay type. In particular, in this case, having finite Cohen-Macaulay type is equivalent to satisfying the Drozd--Ro\u{\i}ter conditions, as first presented in~\cite{drozd1967commutative} and later expanded upon by \c{C}imen in~\cite{ccimen1998one}, based upon the work of Green and Reiner in~\cite{green1978integral}. Later, Greuel and Kn\"orrer in~\cite{greuel1984einfache} classified the one dimensional reduced local rings of finite Cohen-Macaulay type containing an algebraically closed field of characteristic zero as those birationally dominating an ADE singularity, a result that was later expanded upon by Wiegand in~\cite{wiegand1991curve, wiegand1994one} to one dimensional Cohen-Macaulay local rings that contain a perfect field of arbitrary characteristic. Baeth in~\cite{baeth2007krull} studied when such local rings satisfy the Krull--Schmidt theorem, based upon the work of Wiegand in~\cite{wiegand2023failure}. \\
 
In Section~\ref{sec:FCMT1} we show that for residually rational domains finite Cohen-Macaulay type can be read off the semigroup $\Gamma(R)$ alone, in every characteristic and without assuming that $R$ is Gorenstein. In Section~\ref{sec:gorenstein} we show that a residually rational Gorenstein ring has finite Cohen-Macaulay type if and only if $\nu(R)$ is the semigroup of an ADE singularity. When $\chara(\kappa) \neq 2$, the ``only if'' part can also be deduced from Herzog's theorem that Gorenstein rings of finite Cohen-Macaulay type are hypersurfaces~\cite{herzog1978ringe} together with the classification of~\cite{greuel1984einfache, svoray2025ade} (which in characteristics $3$ and $5$ contains additional normal forms of type $E$). Our proof is semigroup-theoretic: it uses only the Drozd--Ro\u{\i}ter conditions and no normal forms, and so it also applies in characteristic $2$ (in equal and in mixed characteristic), where some of the ADE normal forms of characteristic zero are not even reduced. Without the Gorenstein assumption, a ring of finite Cohen-Macaulay type need not be equisingular to an ADE singularity (Example~\ref{ex:nonGor}). In Section~\ref{sec:nonGorenstein} we show that the non-Gorenstein rings of finite Cohen-Macaulay type have minimal multiplicity $3$, and use their first neighborhoods to determine all of their semigroups. As a consequence, finite Cohen-Macaulay type and the Gorenstein property only depend on $\nu(R)$, and $R$ has finite Cohen-Macaulay type if and only if $\nu(R)$ contains the semigroup of an ADE singularity, which is a semigroup analogue, valid in every characteristic, of the theorem of Greuel and Kn\"orrer. The semigroups of the ADE singularities are drawn in Appendix~\ref{app:figures}. We summarize these results in the following two theorems.

\begin{theoremB}
    Assume that $R$ is residually rational, let $\nu(R)\subset\mathbb{N}^{r(R)}$ be its semigroup of values (Definition~\ref{def:nu}), and let $\mathfrak{q}_i=\bigcap_{j\neq i}\mathfrak{p}_j$. 
    \begin{enumerate}
        \item (Theorems~\ref{thm:ADE_curve} and~\ref{thm:nonGor}, Corollaries~\ref{cor:full_classification} and~\ref{cor:containment}) $R$ has finite Cohen--Macaulay type (Definition~\ref{def:cm_type}) if and only if $\nu(R)$ is, up to a permutation of the coordinates, one of the semigroups of Table~\ref{tab:FCMT}, and in this case $R$ is Gorenstein if and only if $\nu(R)$ is in the Gorenstein part of the table. Equivalently, $R$ has finite Cohen--Macaulay type if and only if $\nu(R)$ contains, up to a permutation of the coordinates, the semigroup $\Sigma(X)$ (Definition~\ref{def:ADE_semigroups}) of an ADE singularity $X$. In particular, whether $R$ has finite Cohen--Macaulay type and whether $R$ is Gorenstein depend only on $\nu(R)$.
        \item (Theorems~\ref{thm:ADE_curve} and~\ref{thm:nonGor}) Assume that $R$ has finite Cohen--Macaulay type and is not regular. If $R$ is Gorenstein, then $\edim(R)=2$, and if moreover $\chara(\kappa)\neq 2$, then $R$ is equisingular (Definition~\ref{def:equidominated}) to the ADE type singularity $X$ of Definition~\ref{def:ADE_curves} with $\nu(R)=\Sigma(X)$. If $R$ is not Gorenstein, then $\mult(R)=\edim(R)=3$.
        \item (Theorem~\ref{thm:FCMT_DOm}) More generally, let $R$ be a one dimensional complete local domain which is residually rational and not regular, where $\kappa$ is an arbitrary field (not necessarily separably closed), and let $\Gamma(R)\subset\mathbb{N}$ be its semigroup of values. Then $R$ has finite Cohen--Macaulay type if and only if $\Gamma(R)$ is one of the semigroups of Table~\ref{tab:FCMT} with $r=1$ other than $\mathbb{N}$. 
    \end{enumerate}
\end{theoremB}
 
\begin{table}[!htbp]
\centering
\footnotesize
\setlength{\tabcolsep}{3.5pt}
\renewcommand{\arraystretch}{1.35}
\begin{tabular}{|l|p{4.55cm}|c|c|c|c|c|}
    \hline
    $X$ & $\nu(R)$ & $r$ & $\mult$ & $\delta$ & $\mu$ & $\vec{c}(R)$ \\ \hline\hline
    \multicolumn{7}{|c|}{\textit{Gorenstein}} \\ \hline
    $A_0$ & $\mathbb{N}$ ($R$ is a DVR) & $1$ & $1$ & $0$ & $0$ & $0$ \\ \hline
    $A_{2n}$ & $\lceil 2,2n+1\rfloor$ & $1$ & $2$ & $n$ & $2n$ & $2n$ \\ \hline
    $A_{2n+1}$ & $\lceil (1,1)\rfloor\cup((n+1,n+1)+\mathbb{N}^2)$ & $2$ & $2$ & $n+1$ & $2n+1$ & $(n+1,n+1)$ \\ \hline
    $D_{2n}$ & $\{0\}\cup((2,n,n)+\mathbb{N}^3)$\newline $\cup\,\{(a,k,k)\colon a\geq 1,\ 1\leq k\leq n-1\}$\newline $\cup\,\{(1,n-1,l),(1,l,n-1)\colon l\geq n\}$ & $3$ & $3$ & $n+1$ & $2n$ & $(2,n,n)$ \\ \hline
    $D_{2n+1}$ & $\{0\}\cup\{(a,2s)\colon a\geq 1,\ 1\leq s\leq n-1\}$\newline $\cup\,\{(1,2n-1)\}\cup((2,2n)+\mathbb{N}^2)$ & $2$ & $3$ & $n+1$ & $2n+1$ & $(2,2n)$ \\ \hline
    $E_6$ & $\lceil 3,4\rfloor$ & $1$ & $3$ & $3$ & $6$ & $6$ \\ \hline
    $E_7$ & $\lceil (1,2)\rfloor\cup\{(k,3)\colon k\geq 2\}$\newline $\cup\,((3,5)+\mathbb{N}^2)$ & $2$ & $3$ & $4$ & $7$ & $(3,5)$ \\ \hline
    $E_8$ & $\lceil 3,5\rfloor$ & $1$ & $3$ & $4$ & $8$ & $8$ \\ \hline\hline
    \multicolumn{7}{|c|}{\textit{Not Gorenstein}} \\ \hline
    & $\lceil 3,4,5\rfloor$ & $1$ & $3$ & $2$ & $4$ & $3$ \\ \hline
    & $\lceil 3,5,7\rfloor$ & $1$ & $3$ & $3$ & $6$ & $5$ \\ \hline
    $L\vee A_{2n}$ & $\{0\}\cup\big((\mathbb{N}\setminus\{0\})\times(\lceil 2,2n+1\rfloor\setminus\{0\})\big)$ & $2$ & $3$ & $n+1$ & $2n+1$ & $(1,2n)$ \\ \hline
    $L\vee A_{2n+1}$ & $\{0\}\cup\{(a,j,j)\colon a,j\geq 1\}$\newline $\cup\,\{(a,b,c)\colon a\geq 1,\ b,c\geq n+1\}$ & $3$ & $3$ & $n+2$ & $2n+2$ & $(1,n+1,n+1)$ \\ \hline
    $\Theta$ & $\{(0,0),(1,2)\}\cup((2,3)+\mathbb{N}^2)$ & $2$ & $3$ & $3$ & $5$ & $(2,3)$ \\ \hline
\end{tabular}
\caption{The semigroups $\nu(R)$, up to a permutation of the coordinates, of the residually rational rings of finite Cohen--Macaulay type (Corollary~\ref{cor:full_classification}), together with the number of branches $r=r(R)$, the multiplicity, $\delta(R)$, $\mu(R)$, and the conductor vector $\vec{c}(R)$ (Notation~\ref{not:cond_vec}). Here $n\geq 1$ for $A_{2n}$ and $L\vee A_{2n}$, $n\geq 0$ for $A_{2n+1}$ and $L\vee A_{2n+1}$, and $n\geq 2$ for $D_{2n}$ and $D_{2n+1}$.}\label{tab:FCMT} 
\end{table}
\noindent\textbf{Notations}: We denote $\mathbb{N}=\{0,1,2, \dots\}$ and the set of $r-$tuples over $\mathbb{N}$ by $\mathbb{N}^r$. We order $\mathbb{N}^r$ with the product order defined component-wise, that is, given two tuples $(\alpha_1, \dots, \alpha_r),(\beta_1, \dots, \beta_r) \in \mathbb{N}^r$ we say that $(\alpha_1, \dots, \alpha_r) \leq (\beta_1, \dots, \beta_r)$ if $\alpha_i \leq \beta_i$ for every $i$. Given some $\vec{a}_1, \dots, \vec{a}_n \in \mathbb{N}^r$, we denote by $\lceil \vec{a}_1, \dots, \vec{a}_n \rfloor$ the smallest (additive) semigroup of $\mathbb{N}^r$ that contains $\vec{a}_1, \dots, \vec{a}_n$. Given some $\vec{a} \in \mathbb{N}^r$, we denote $\vec{a} + \mathbb{N}^{r}$ the set of all $\vec{b} \in \mathbb{N}^r$ such that $\vec{b} \geq \vec{a}$. We write $\mathds{1}=(1,\dots,1)$. Similarly, we denote $\mathbb{N}_\infty = \mathbb{N} \cup \{\infty\}$ with its set of tuples $\mathbb{N}_\infty^r$ (together with the corresponding product order) and with sets of the form $\vec{a} + \mathbb{N}^r_\infty$.   \\

\noindent\textbf{Acknowledgments.} This work was done as part of the Author's PhD thesis under the guidance of Karl Schwede, and we wish to thank him for his guidance, help, and support. We wish to thank Benjamin Baily, Neil Epstein, Gert-Martin Greuel, and Tim Tribone for productive mathematical discussions and their inputs on the different ideas presented in this paper. The author was partially supported by NSF grant DMS-2101800. 

\section{The Milnor Number $\mu(R)$}\label{sec:milnor}

In this section we define the Milnor number $\mu(R)$ of a one dimensional local ring and show that it behaves similarly to its analogue for algebraic curves. Before defining $\mu(R)$, we turn to some notions,  technical results, and remarks that we use throughout this paper. Even though some of these results are known to experts,  we place them here for the sake of completeness and ease of reference.\\

Throughout this section, we assume that $(R, \mathfrak{m}, \kappa)$ is a one-dimensional Noetherian strictly Henselian reduced Nagata local ring with normalization $\overline{R}$. 

\begin{notation}\label{notation:ringstuff}
Let $(R, \mathfrak{m}, \kappa)$ be a one dimensional Noetherian local ring. Then: 
    \begin{enumerate}
        \item The delta invariant of $R$ is defined to be $\delta(R) = \length_R \left( \frac{\overline{R}}{R}\right)$, i.e. the length over $R$ of the quotient module $\frac{\overline{R}}{R}$. 
        \item Given a module $M$ over $R$, the multiplicity of $M$ over $R$ is defined to be $\mult_R(M)=\lim_{n \to \infty} \frac{\length_R\left(\frac{M}{\mathfrak{m}^nM}\right)}{n}$, and the multiplicity of $R$ is defined to be $\mult(R)=\mult_R(R)$. 
        \item The embedding dimension of $R$ is defined to be $\edim(R) = \dim_{\kappa} \left( \frac{\mathfrak{m}}{\mathfrak{m}^2}\right)$.
        \item Given a ring $S$, we denote by $\dim(S)$ its Krull dimension. 
        \item We denote by $\hat{R}$  the completion of $R$ (with respect to the maximal ideal $\mathfrak{m}$).
        \item We denote the number of minimal primes in $R$ (which is finite since $R$ is Noetherian) by $r(R)$.
        \item We denote by $\NZD(R)$ the set of non zero divisors of $R$.
        \item Given ideals $I$ and $J$ of $R$, the intersection multiplicity of $I$ and $J$, denoted by $i(I, J)$, is the $R-$length of $\frac{R}{I+J}$. 
        \item The conductor of $R$ is the $R-$ideal $\mathfrak{c}_R=\Ann_R\left(\frac{\overline{R}}{R}\right)$ (which can in fact also be viewed as an $\overline{R}-$ideal).
        \item If $\mathfrak{p}_1, \dots, \mathfrak{p}_{r(R)}$ are the minimal primes of $R$, we denote by $V_i$ the normalization of $\frac{R}{\mathfrak{p}_i}$, by $\pi_i$ a uniformizer of $V_i$, by $\kappa_i$ the residue field of $V_i$, and by $\mathfrak{q}_i=\bigcap_{j\neq i}\mathfrak{p}_j$.  
    \end{enumerate}
\end{notation}

\begin{remark}\label{rem:techincal_ring}
\begin{enumerate}
    \item \textup{Note that since $R$ is one dimensional and Cohen-Macaulay (being reduced), by~\cite{abhyankar1967local} we have that $\mult(R)\geq \edim(R)$.}
    \item \textup{Since $R$ is one-dimensional and reduced, every prime $\mathfrak p\ne\mathfrak m$ is minimal and $R_{\mathfrak p}$ is a field. Thus $R$ is regular on the punctured spectrum.}
    \item \textup{Since $R$ is reduced then we have that $\overline{R}$ is the product of the integral closure of $\frac{R}{\mathfrak{p}_i}$ over every minimal prime $\mathfrak{p}_1, \dots, \mathfrak{p}_{r(R)}$ of $R$. Therefore we identify $\overline{R}=\prod_{i=1}^{r(R)} V_i$, view $R \subset \overline{R}$, and write $g=(g_1,\dots,g_{r(R)})$ for $g \in \overline{R}$.}
\end{enumerate}
\end{remark}

\begin{lemma}\label{lem:mult_sum_prim}
    If $\mathfrak{p}_1, \dots, \mathfrak{p}_{r(R)}$ are the minimal primes of $R$, then we have that $\mult(R)=\sum_{i=1}^{r(R)} \mult\left(\frac{R}{\mathfrak{p}_i}\right)$. 
\end{lemma}

\begin{proof}
    Denote $\mathfrak{q}=\bigcap_{i=1}^{r(R)-1} \mathfrak{p}_i$. Then we have that $\mathfrak{q}\cap \mathfrak{p}_{r(R)} = \{0\}$, and so we get a short exact sequence:
\begin{equation*}
    0 \to R \to \frac{R}{\mathfrak{q}} \oplus \frac{R}{\mathfrak{p}_{r(R)}} \to \frac{R}{\mathfrak{q}+\mathfrak{p}_{r(R)}} \to 0. 
\end{equation*}
Therefore, by applying multiplicity to the exact sequence (see~\cite[Tag 0AZW]{stacks-project}), we can conclude that $\mult(R) + \mult(\frac{R}{\mathfrak{p}_{r(R)}+\mathfrak{q}}) = \mult(\frac{R}{\mathfrak{q}} \oplus \frac{R}{\mathfrak{p}_{r(R)}})$. Yet, since $\frac{R}{\mathfrak{p}_{r(R)}+\mathfrak{q}}$ is zero dimensional, we get that  $\mult(\frac{R}{\mathfrak{p}_{r(R)}+\mathfrak{q}})=0$, and so the result follows by induction on $r(R)$ (with respect to the ring $\frac{R}{\mathfrak{q}}$). 
\end{proof}

\begin{corollary}\label{cor:mult_r}
    $\mult(R) \geq r(R)$.
\end{corollary}

\begin{proof}
    This follows directly from Lemma~\ref{lem:mult_sum_prim}, noting that $\mult \left(\frac{R}{\mathfrak{p}_i}\right) \geq 1$ for every $i$. 
\end{proof}

\begin{lemma}\label{cor:edim_bound}
    $\delta(R) \geq \mult(R)-1$. 
\end{lemma}

\begin{proof}
    By applying length over $R$ to the short exact sequence of $R-$modules 
    \begin{equation*}
        0 \to \frac{R}{\mathfrak{m}} \to \frac{\overline{R}}{\mathfrak{m}} \to \frac{\overline{R}}{R} \to 0,
    \end{equation*}
    \noindent we get that $\delta(R)=\length_R\left(\frac{\overline{R}}{\mathfrak{m}}\right)-1$. Yet, we have that $\length_R\left(\frac{\overline{R}}{\mathfrak{m}}\right) \geq \length_R\left(\frac{\overline{R}}{\mathfrak{m}\overline{R}}\right)$ which equals to $\mult(R)$ (see Section 11.2. in~\cite{huneke2006integral}). 
\end{proof}

The following is a stronger version of the Cohen structure theorem, that we use extensively throughout when classifying $R$ based upon its invariants: 

\begin{lemma}\label{lem:edim_CST}
    For every $n \geq \edim(R)$, there exists some complete regular local ring $(S, \mathfrak{n})$ with a surjective map $S \to \hat{R}$ such that $\dim(S)=n$.
\end{lemma}

\begin{proof}
    By looking at the ring of power series over a complete regular local ring, it is enough to prove the lemma in the case where $n=\edim(R)$. From the Cohen structure theorem (see, e.g.~\cite[Tag 032A]{stacks-project} or Corollary 28.3. in~\cite{matsumura1970commutative}), there exists some complete regular local ring $(S, \mathfrak{n})$ of some dimension $N\geq \edim(R)$ together with a surjective map $S \to \hat{R}$. Therefore, we have a surjective map of vector spaces $\frac{\mathfrak{n}}{\mathfrak{n}^2} \to \frac{\mathfrak{m}}{\mathfrak{m}^2}$ whose kernel is of dimension $N-\edim(R)$. If $I$ denotes the kernel of $S \to \hat{R}$, then the kernel of this map of vector spaces is $\frac{I+\mathfrak{n}^2}{\mathfrak{n}^2}$, so we can choose $f_1, \dots, f_{N-\edim(R)} \in I$ whose images mod $\mathfrak{n}^2$ form a basis of it, and set $S_1=\frac{S}{\langle f_1, \dots, f_{N-\edim(R)} \rangle}$. We have that $S_1$ is a complete ring of dimension $\edim(R)$ (by the Krull principal ideal theorem) and the map $S \to \hat{R}$ factors through $S_1$ giving us a surjective map $S_1 \to \hat{R}$. In addition, since the maximal ideal of $S_1$ is $\frac{\mathfrak{n}}{\langle f_1, \dots, f_{N-\edim(R)} \rangle}$, we have that $S_1$ is regular as its Zariski cotangent space is $\frac{\mathfrak{n}}{\langle f_1, \dots, f_{N-\edim(R)} \rangle + \mathfrak{n}^2}$, which is a vector space of dimension $\edim(R)$. 
\end{proof}

\begin{remark}
    \textup{Lemma~\ref{lem:edim_CST} gives us that $\hat{R}$ is isomorphic to some quotient $\frac{S}{I}$ for some complete regular local ring $S$. This observation is important when classifying rings $R$ (up to completion). }
\end{remark}

    We are now ready to define the Milnor number of $R$, inspired by the definition of $\overline{\mu}$ in~\cite{barroso2022note}:

\begin{definition}\label{def:Milnor}
 We define the Milnor number of $R$ to be $\mu(R)= 2 \delta(R)-r(R) + 1$.
\end{definition}

By the theorem of Buchweitz and Greuel~\cite{buchweitz1980milnor} recalled in the introduction, if $R$ is the local ring of a reduced complex curve singularity, of any embedding dimension, then $\mu(R)$ coincides with the Milnor number $\dim_{\mathbb{C}}\left(\omega_R/dR\right)$ defined there.

\begin{remark}\label{rem:conductor}
\begin{enumerate}
    \item \textup{By Corollary 12.2.4 in~\cite{huneke2006integral} and by Proposition 4.4.7 in~\cite{campillo2006algebroid} we have that $R$ is Gorenstein if and only if   $2\length_R(\frac{R}{\mathfrak{c}_R}) = \length_R(\frac{\overline{R}}{\mathfrak{c}_R}) $. In addition, since we have a short exact sequence
    \begin{equation*}
        0 \to \frac{R}{\mathfrak{c}_R} \to \frac{\overline{R}}{\mathfrak{c}_R} \to  \frac{\overline{R}}{R} \to 0, 
    \end{equation*}
    \noindent we can conclude that this is also equivalent to having $2 \delta(R)=  \length_R(\frac{\overline{R}}{\mathfrak{c}_R})$. }
    \item \textup{Note that $\mu(R)$ is finite since $\delta(R)$ is finite. This is true since $R$ is a one-dimensional reduced Nagata ring, and so $\overline{R}$ is a normal Nagata ring such that $R \to \overline{R}$ is a finite map. In particular, the support of $\frac{\overline{R}}{R}$ as an $R-$module is zero dimensional, and so it must have a finite length. }
    \item \textup{Remark~\ref{rem:techincal_ring} (together with the previous item) can be thought of as an analogue of Lemma 2.3. in Chapter I of~\cite{greuel2007introduction}, i.e. that the Milnor number is finite if and only if our hypersurface has an isolated singularity.  }
    \item \textup{Since $R$ is a reduced Nagata local ring of dimension one, its completion $\hat{R}$ is again a reduced Nagata local ring of dimension one and $\overline{R}\otimes_R \hat{R}$ is the normalization of $\hat{R}$ (see~\cite[Tag 0C3V]{stacks-project}. This is a theorem of Nagata, compare~\cite[Tags 0331, 032Y]{stacks-project}). In particular $R$ is analytically unramified. Since $R$ is moreover strictly Henselian, $\overline{R}$ is a product of $r(R)$ local rings, so $\overline{R}\otimes_R\hat{R}$ is a product of $r(R)$ complete DVRs, and we get that $r(R)=r(\hat{R})$. Moreover $\delta(R)=\delta(\hat{R})$ by~\cite[Tag 0C3W]{stacks-project}. Therefore we can conclude that $\mu(R)=\mu(\hat{R})$. In particular, the hypothesis ``analytically unramified'' is automatic for the rings considered, and we do not repeat it. Moreover, if $\mathfrak{p}$ is a minimal prime of $R$ then $\mathfrak{p}\hat{R}$ is a minimal prime of $\hat{R}$: indeed $\mathfrak{p}\hat{R}$ is radical (as $\frac{R}{\mathfrak{p}}$ is analytically unramified), every minimal prime of $\hat{R}$ contains exactly one $\mathfrak{p}_i\hat{R}$ (it lies over a minimal prime of $R$ by flatness, and it cannot contain the $\mathfrak{m}\hat{R}-$primary ideal $(\mathfrak{p}_i+\mathfrak{p}_j)\hat{R}$ for $i \neq j$), every minimal prime over $\mathfrak{p}_i\hat{R}$ is a minimal prime of $\hat{R}$, and $r(\hat{R})=r(R)$. Therefore, each $\mathfrak{p}_i\hat{R}$ has exactly one minimal prime over it, and being radical, it is prime.  }
\end{enumerate}
\end{remark}

The following Lemma is an analogue of~\cite[Tag 0C42]{stacks-project}:

\begin{lemma}\label{lem:milnor_nonneg}
    Let $J$ be the Jacobson radical of $\overline{R}=\prod_i V_i$. Then 
    \begin{equation*}
        \delta(R) = \sum_{i=1}^{r(R)} [\kappa_i : \kappa] - 1 + \length_R\left(\frac{R+J}{R}\right) \geq r(R)-1,
    \end{equation*}
    and so $\mu(R) \geq \delta(R) \geq 0$. 
\end{lemma}

\begin{proof}
     Since $\overline{R}$ is integral over $R$, every maximal ideal of $\overline{R}$ contracts to $\mathfrak{m}$, and so $\mathfrak{m} \subset J$. Conversely, elements of $J \cap R$ are not units in $\overline{R}$, hence not in $R$. Thus $J \cap R =\mathfrak{m}$, and the image of $R$ in $\frac{\overline{R}}{J}=\prod_{i=1}^{r(R)} \kappa_i$ is the diagonal copy of $\kappa$. Therefore $\length_R\left(\frac{\overline{R}}{R+J}\right) = \dim_\kappa\left(\frac{\prod_{i=1}^{r(R)}  \kappa_i}{\kappa}\right)=\sum_i [\kappa_i : \kappa]-1 \geq r(R)-1$, and the displayed formula follows from the short exact sequence $0 \to \frac{R+J}{R} \to \frac{\overline{R}}{R} \to \frac{\overline{R}}{R+J} \to 0$. In addition, $\mu(R) = \delta(R) + (\delta(R)-r(R)+1) \geq \delta(R)$. 
\end{proof}

\begin{corollary}
    Using the notations of Lemma~\ref{lem:milnor_nonneg}, the following are equivalent:
    \begin{enumerate}
        \item $\delta(R) = \sum_{i} [\kappa_i : \kappa] - 1$.
        \item $J \subset R$. 
        \item $\mathfrak{m}=J$.
    \end{enumerate}
\end{corollary}
\begin{proof}
    The equivalence of (1) and (2) follows from Lemma~\ref{lem:milnor_nonneg}. If $J\subset R$, then $J\subset J\cap R=\mathfrak m\subset J$, so (2) implies (3). The implication (3)$\Rightarrow$(2) is immediate. 
\end{proof}

We now turn to presenting an analogue of Hironaka's lemma (first proven by Hironaka for projective curves in~\cite{hironaka1957arithmetic}), which we use to classify $R$ for specific values of $\mu(R)$. For more information on Hironaka's Lemma, see Lemma 2.1 in~\cite{campillo1983hamburger}, Lemma 1.2.2 in~\cite{buchweitz1980milnor}, or Lemma 3.32. of Chapter I in~\cite{greuel2007introduction}. As the example following the proposition shows, the classical formula in terms of the pairwise intersection numbers $i(\mathfrak{p}_i, \mathfrak{p}_j)$ only holds for rings of embedding dimension at most two. In general, one has to intersect each branch with the union of the previous ones (compare Lemma 1.2.2 in~\cite{buchweitz1980milnor}).

\begin{proposition}[Hironaka's Lemma] \label{thm:branch}
    Let $\mathfrak{p}_1, \dots, \mathfrak{p}_{r(R)}$ be the minimal primes of $R$. 
    \begin{enumerate}
        \item If $\{1,\dots,r(R)\}=I \sqcup I'$ with $I, I'$ non-empty, $\mathfrak{q}=\bigcap_{i\in I}\mathfrak{p}_i$, and $\mathfrak{q}'=\bigcap_{i\in I'}\mathfrak{p}_i$, then $\delta(R)=\delta\left(\frac{R}{\mathfrak{q}}\right)+\delta\left(\frac{R}{\mathfrak{q}'}\right)+i(\mathfrak{q},\mathfrak{q}')$. Consequently 
    \begin{equation*}
        \delta(R) = \sum_{i=1}^{r(R)}  \delta\left(\frac{R}{\mathfrak{p}_i}\right) + \sum_{j=2}^{r(R)} i(\mathfrak{p}_j, \mathfrak{p}_1\cap \cdots \cap \mathfrak{p}_{j-1}),
    \end{equation*}
    and in particular $\delta(R) \geq \sum_{i}\delta\left(\frac{R}{\mathfrak{p}_i}\right)+r(R)-1$.
        \item If $\edim(R) \leq 2$ then $i(\mathfrak{p}_j, \mathfrak{p}_1\cap \cdots \cap \mathfrak{p}_{j-1})=\sum_{i<j} i(\mathfrak{p}_i, \mathfrak{p}_j)$ for every $j$, and so
    \begin{equation*}
        \delta(R) = \sum_{i=1}^{r(R)}  \delta\left(\frac{R}{\mathfrak{p}_i}\right) + \sum_{i <  j} i(\mathfrak{p}_i, \mathfrak{p}_j) \geq \frac{r(R)(r(R)-1)}{2}.
    \end{equation*}
    \end{enumerate}
\end{proposition}

\begin{proof}
For the first item, since $R$ is reduced we have that $\mathfrak{q}\cap \mathfrak{q}' = \{0\}$, and so we get a short exact sequence
\begin{equation*}
    0 \to R \to \frac{R}{\mathfrak{q}} \oplus \frac{R}{\mathfrak{q}'} \to \frac{R}{\mathfrak{q}+\mathfrak{q}'} \to 0. 
\end{equation*}
\noindent Therefore, we can conclude that 
\begin{equation*}
    \delta(R)=\length_R\left(\frac{\overline{R}}{R}\right) = \length_R\left(\frac{\overline{R}}{\frac{R}{\mathfrak{q}} \oplus \frac{R}{\mathfrak{q}'} }\right) + \length_R \left(\frac{\frac{R}{\mathfrak{q}} \oplus \frac{R}{\mathfrak{q}'} }{R} \right).
\end{equation*}

\noindent Yet, the normalization of $\frac{R}{\mathfrak{q}} \oplus \frac{R}{\mathfrak{q}'}$ is exactly $\overline{R}$, that is, $\overline{R} = \overline{\left(\frac{R}{\mathfrak{q}}\right)} \times \overline{\left(\frac{R}{\mathfrak{q}'}\right)}$. In addition, we have that 
\begin{equation*}
    \frac{\frac{R}{\mathfrak{q}} \oplus \frac{R}{\mathfrak{q}'} }{R}  \cong \frac{R}{\mathfrak{q} + \mathfrak{q}'}. 
\end{equation*}
\noindent Therefore, we can conclude that $\delta(R) = \delta\left(\frac{R}{\mathfrak{q}}\right) + \delta\left(\frac{R}{\mathfrak{q}'} \right) + i(\mathfrak{q}, \mathfrak{q}')$. Applying this with $\mathfrak{q}=\mathfrak{p}_1 \cap \cdots \cap \mathfrak{p}_{r(R)-1}$ and $\mathfrak{q}'=\mathfrak{p}_{r(R)}$, and noting that the ring $\frac{R}{\mathfrak{q}}$ again satisfies our standing assumptions and its minimal primes correspond to $\mathfrak{p}_1, \dots, \mathfrak{p}_{r(R)-1}$, the displayed formula follows by induction on $r(R)$. Finally, $i(\mathfrak{p}_j, \mathfrak{p}_1\cap \cdots \cap \mathfrak{p}_{j-1}) \geq 1$ for every $j \geq 2$, since $\mathfrak{p}_j + (\mathfrak{p}_1\cap \cdots \cap \mathfrak{p}_{j-1}) \subset \mathfrak{m}$.\\

For the second item, we may assume that $R$ is not a DVR, so $\edim(R)=2$. All the intersection multiplicities are lengths of $R-$modules of finite length, and so they do not change when passing to $\hat{R}$; in addition, by Remark~\ref{rem:conductor} the ideals $\mathfrak{p}_i\hat{R}$ are exactly the minimal primes of $\hat{R}$, and $(\mathfrak{p}_1 \cap \cdots \cap \mathfrak{p}_{j-1})\hat{R}=\mathfrak{p}_1\hat{R} \cap \cdots \cap \mathfrak{p}_{j-1}\hat{R}$ by flatness. By Lemma~\ref{lem:edim_CST} we can write $\hat{R}=\frac{S}{I}$ with $(S, \mathfrak{n})$ a two dimensional complete regular local ring. Since $S$ is a UFD and $\hat{R}$ is reduced of dimension one, $I=\langle f \rangle$ with $f=f_1\cdots f_{r(R)}$, where $f_1, \dots, f_{r(R)}$ are pairwise non-associate primes of $S$ with $\mathfrak{p}_i\hat{R}=\frac{\langle f_i \rangle}{\langle f \rangle}$, and so $(\mathfrak{p}_1 \cap \cdots \cap \mathfrak{p}_{j-1})\hat{R}=\frac{\langle f_1\cdots f_{j-1} \rangle}{\langle f \rangle}$. Therefore $i(\mathfrak{p}_j, \mathfrak{p}_1\cap \cdots \cap \mathfrak{p}_{j-1}) = \length_S\left(\frac{S}{\langle f_j, f_1\cdots f_{j-1}\rangle}\right)$, and the claim follows from the additivity $\length_S\left(\frac{S}{\langle f_j, gh\rangle}\right)=\length_S\left(\frac{S}{\langle f_j, g\rangle}\right)+\length_S\left(\frac{S}{\langle f_j, h\rangle}\right)$ for $g,h \notin \langle f_j \rangle$, which follows from the short exact sequence 
\begin{equation*}
    0 \to \frac{S}{\langle f_j, h\rangle} \xrightarrow{\cdot g} \frac{S}{\langle f_j, gh\rangle} \to \frac{S}{\langle f_j, g\rangle} \to 0
\end{equation*}
(the first map is injective since $\frac{S}{\langle f_j \rangle}$ is a domain in which $g\neq 0$). The inequality follows since $i(\mathfrak{p}_i, \mathfrak{p}_j)\geq 1$ for every $i \neq j$.
\end{proof}

\begin{example}\label{ex:hironaka_fails}
\textup{The assumption $\edim(R)\leq 2$ in the second item of Proposition~\ref{thm:branch} cannot be dropped, as we can see in the following two examples (let $k$ be an algebraically closed field):
\begin{enumerate}
    \item Let $R=\frac{k[[x,y,z]]}{\langle xy,yz,xz\rangle}$ be the three coordinate axes. Then $\overline{R}=k[[t]]^3$ and $R=\{(a,b,c) \colon a(0)=b(0)=c(0)\}$, so $\delta(R)=2$ and $r(R)=3$, while $i(\mathfrak{p}_i,\mathfrak{p}_j)=1$ for all $i\neq j$; thus $\sum_{i<j}i(\mathfrak{p}_i,\mathfrak{p}_j)=3>\delta(R)$, and $\mu(R)=2<(r(R)-1)^2$.
    \item Let $R=\frac{k[[x,y,z]]}{\langle xy, z^2-x^2-y^2\rangle}$ (with $\chara(k)\neq 2$), a complete intersection consisting of the four lines parametrized by $(0,t,t),(0,t,-t),(t,0,t),(t,0,-t)$. The image of $R$ in $k[[t]]^4$ contains all elements of order at least $2$, a $3-$dimensional space of linear terms, and the constants, so $\delta(R)=4$, while all six pairwise intersection numbers are equal to $1$.
\end{enumerate}}
\end{example}

From Proposition~\ref{thm:branch} we can conclude an analogous result for $\mu(R)$, which is an analogue of Property 2.3. in~\cite{ploski1995milnor}, of Corollary 1.2.3 in~\cite{buchweitz1980milnor}, and of item 2 of Proposition 1.1 in~\cite{barroso2018milnor}:

\begin{corollary}\label{cor:lower_bound}
    Let $\mathfrak{p}_1, \dots, \mathfrak{p}_{r(R)}$ be the minimal primes of $R$.
    \begin{enumerate}
        \item With $\mathfrak{q}, \mathfrak{q}'$ as in the first item of Proposition~\ref{thm:branch}, we have $\mu(R)=\mu\left(\frac{R}{\mathfrak{q}}\right)+\mu\left(\frac{R}{\mathfrak{q}'}\right)+2i(\mathfrak{q},\mathfrak{q}')-1$. Consequently
    \begin{equation*}
        \mu(R) -1 = \sum_{i=1}^{r(R)} \left(\mu\left(\frac{R}{\mathfrak{p}_i}\right)-1\right) + 2\sum_{j=2}^{r(R)} i(\mathfrak{p}_j, \mathfrak{p}_1\cap \cdots \cap \mathfrak{p}_{j-1}),
    \end{equation*}
    and in particular $\mu(R) \geq r(R)-1$. 
        \item If $\edim(R)\leq 2$ then $\mu(R) -1 = \sum_{i=1}^{r(R)} \left(\mu\left(\frac{R}{\mathfrak{p}_i}\right)-1\right) + 2\sum_{i <  j} i(\mathfrak{p}_i, \mathfrak{p}_j)$, and in particular $\mu(R) \geq (r(R)-1)^2$.
    \end{enumerate}
\end{corollary}

\begin{proof}
    The formulas follow directly from Proposition~\ref{thm:branch} together with the definition $\mu(R)=2\delta(R)-r(R)+1$ (note that $r(R)=r\left(\frac{R}{\mathfrak{q}}\right)+r\left(\frac{R}{\mathfrak{q}'}\right)$ and $\mu\left(\frac{R}{\mathfrak{p}_i}\right)=2\delta\left(\frac{R}{\mathfrak{p}_i}\right)$). For the inequalities, since $\mu\left(\frac{R}{\mathfrak{p}_i}\right)\geq 0$ and every intersection multiplicity is at least $1$, we get $\mu(R)-1 \geq -r(R)+2(r(R)-1)$ in the first case and $\mu(R) -1\geq -r(R) + r(R)(r(R)-1)$ in the second, i.e. $\mu(R)\geq r(R)-1$ and $\mu(R) \geq (r(R)-1)^2$, respectively.
\end{proof}

\begin{remark}
    \textup{Note that Corollary~\ref{cor:lower_bound} provides us with an alternative proof to Lemma~\ref{lem:milnor_nonneg} (i.e. that $\mu(R)$ is non-negative, as $r(R)\geq 1$). }
\end{remark}

We can use Corollary~\ref{cor:lower_bound} to provide a classification of DVR as rings that satisfy $\mu(R)=0$, which is an analogue of Lemma 2.44. in Chapter I of~\cite{greuel2007introduction} and of item 5 of Proposition 1.1 in~\cite{barroso2018milnor}. 

\begin{corollary}\label{cor:reg}
    $\mu(R)=0$ if and only if $R$ is a DVR. 
\end{corollary}

\begin{proof}
    By Corollary~\ref{cor:lower_bound}, if $\mu(R)=0$, we must get that $r(R)=1$, and so we can conclude that $0=\mu(R)=2\delta(R)$. Therefore $R$ must be normal, and a normal one dimensional local ring must be a DVR. Conversely, a DVR is a normal domain, so $\delta(R)=0$, $r(R)=1$ and $\mu(R)=0$. 
\end{proof}

The next proposition describes the rings for which the inequality $\delta(R)\geq r(R)-1$ of Lemma~\ref{lem:milnor_nonneg} (equivalently, $\mu(R)\geq r(R)-1$) is an equality. They are the analogues of "$r$ lines in $r-$space in general position", i.e. the seminormal curve singularities; compare~\cite[Tag 0C42]{stacks-project}, where it is shown that if equality holds then $R$ is a wedge of discrete valuation rings. 

\begin{proposition}\label{prop:axes}
    The following are equivalent:
    \begin{enumerate}
        \item $\mu(R)=\delta(R)$,
        \item $\delta(R)=r(R)-1$,
        \item $\kappa_i=\kappa$ for every $i$ and $\mathfrak{m}$ is the Jacobson radical of $\overline{R}$,
        \item there exists some $r(R)-$dimensional complete regular local ring $(S, \mathfrak{n})$ with $\mathfrak{n} =\langle x_1, \dots, x_{r(R)} \rangle$ such that $\hat{R}$ is isomorphic to $\frac{S}{\langle x_ix_j \colon i \neq j \rangle}$.
    \end{enumerate}
    In this case $\frac{R}{\mathfrak{p}_i}$ is a DVR for every $i$, $\mult(R)=\edim(R)=r(R)$, and $\mu(R)=r(R)-1$. Moreover, if in addition $\edim(R)\leq 2$ or $R$ is Gorenstein, then these conditions hold if and only if $R$ is either a DVR or a double point. 
\end{proposition}

\begin{proof}
    The equivalence of the first two items follows from the definition of $\mu(R)$, and the equivalence of the second and third items follows from Lemma~\ref{lem:milnor_nonneg}, since $\sum_i[\kappa_i : \kappa]-1 \geq r(R)-1$ with equality if and only if $\kappa_i=\kappa$ for every $i$. Let $J=\prod_i \pi_iV_i$ be the Jacobson radical of $\overline{R}$.\\

    Assume the third item holds. Then the image $\frac{R}{\mathfrak{p}_i}$ of $R$ in $V_i$ contains $\pi_iV_i$ (as $J \subset R$) and surjects onto $\kappa_i=\kappa$, and so $\frac{R}{\mathfrak{p}_i}=V_i$ is a DVR. Let $x_i \in J=\mathfrak{m}$ be the element whose $i-$th coordinate is $\pi_i$ and whose other coordinates are $0$. Since $R \to V_i$ is surjective we have $Rx_i=\pi_iV_i$ (viewed in the $i-$th factor), so $\mathfrak{m}=J=\sum_iRx_i$, $x_ix_j=0$ for $i\neq j$, and $\mathfrak{m}^2=\prod_i \pi_i^2V_i$. In particular $\edim(R)=r(R)$, and by Lemma~\ref{lem:mult_sum_prim} $\mult(R)=r(R)$. By Lemma~\ref{lem:edim_CST} there is a surjection $S \to \hat{R}$ from an $r(R)-$dimensional complete regular local ring $(S,\mathfrak{n})$; as it induces an isomorphism $\frac{\mathfrak{n}}{\mathfrak{n}^2}\cong\frac{\mathfrak{m}}{\mathfrak{m}^2}$, we may lift the $x_i$ to generators $X_1, \dots, X_{r(R)}$ of $\mathfrak{n}$, and the surjection factors through $T=\frac{S}{\langle X_iX_j \colon i\neq j\rangle}$. The ring $T$ is reduced, and its primes $P$ with $\dim\left(\frac{T}{P}\right)=1$ are exactly its $r(R)$ minimal primes $P_i=\langle X_j \colon j \neq i\rangle$. By Remark~\ref{rem:conductor}, $\hat{R}$ has $r(\hat{R})=r(R)$ minimal primes $Q$, each with $\dim\left(\frac{\hat{R}}{Q}\right)=1$. Each of them pulls back to a prime of $T$ of dimension one, i.e. to one of the $P_i$, and distinct primes pull back to distinct primes; hence every $P_i$ is such a pullback, and the kernel of $T \to \hat{R}$ is contained in $\bigcap_i P_i=0$. Thus $\hat{R}\cong T$.\\

    Assume the fourth item holds, and let $T=\frac{S}{\langle x_ix_j \colon i\neq j\rangle} \cong \hat{R}$. Its minimal primes are $P_i=\langle x_j \colon j\neq i \rangle$, and $W_i=\frac{T}{P_i}$ is a DVR with uniformizer the image of $x_i$ and residue field $\kappa$. The normalization of $T$ is $\prod_iW_i$, and since $T \to W_i$ is surjective, the image of $x_iT$ in $\prod_i W_i$ is the maximal ideal of $W_i$ placed in the $i-$th factor. Thus $T$ contains the Jacobson radical of its normalization, and by Lemma~\ref{lem:milnor_nonneg} $\delta(T)=r(T)-1$. Since $\delta$ and $r$ are preserved under completion (Remark~\ref{rem:conductor}), the second item follows. In this case $\mu(R)=2(r(R)-1)-r(R)+1=r(R)-1$.\\

    Finally, if $\edim(R)\leq 2$ then $r(R)=\edim(R)\leq 2$. If $R$ is Gorenstein and $r(R)\geq 2$, then $J \subset R$ is an $\overline{R}-$ideal, so $J\subset \mathfrak{c}_R$; as $R\neq \overline{R}$ we also have $\mathfrak{c}_R \subset \mathfrak{m}=J$. Thus $\mathfrak{c}_R=J$ and $\length_R\left(\frac{\overline{R}}{\mathfrak{c}_R}\right)=r(R)$, and by the first item of Remark~\ref{rem:conductor}, $R$ is Gorenstein if and only if $r(R)=2\delta(R)=2r(R)-2$, i.e. $r(R)=2$. In both cases $r(R)\leq 2$, and the fourth item says that $R$ is a DVR (if $r(R)=1$) or a double point (if $r(R)=2$). Conversely, DVRs and double points satisfy the fourth item. 
\end{proof}

\begin{remark}\label{rem:omp_lines_edim}
    \textup{Proposition~\ref{prop:axes} tells us that intuitively, the rings with $\delta(R)=r(R)-1$ can be geometrically thought of as "a collection of $r(R)$ lines in $r(R)-$space in general position, all intersecting at a single point". For $r(R)\geq 3$ they have embedding dimension $r(R)$ and are not Gorenstein. }
\end{remark}

\begin{definition}\label{def:doublept}
    We say that $(R , \mathfrak{m})$ is a double point if there exists some $2-$dimensional complete regular local ring $(S, \mathfrak{n})$ such that $\mathfrak{n} = \langle x,y \rangle$ and $\hat{R}$ is isomorphic to $\frac{S}{\langle xy \rangle}$. 
\end{definition}

The following theorem is an analogue version of Morse's lemma, as presented in Theorem 2.46 in~\cite{greuel2007introduction} (which is based upon~\cite{morse1934calculus, varchenko1985singularities}). Note that it is also a characteristic-free one dimensional analogue of Proposition 3.3. in~\cite{carvajal2019covers}, together with its generalization in Proposition 3.14 of~\cite{svoray2025ade}. 

\begin{theorem}\label{thm:Morse}
    $\mu(R)=1$ if and only if $R$ is a double point. 
\end{theorem}

\begin{proof}
    If $\mu(R)=1$ then by Corollary~\ref{cor:lower_bound} we have that $1 \geq r(R)-1$, and so $r(R) \in \{1,2\}$. Yet, $r(R)$ can not be $1$ since as $\mu(R)=2\delta(R)-r(R)+1$ then we would conclude that $1 = \mu(R) = 2\delta(R)$, which is impossible as $\delta(R)$ is an integer. Therefore we must have that $r(R)=2$ and $\delta(R)=1=r(R)-1$, and so by Proposition~\ref{prop:axes} the ring $\hat{R}$ is isomorphic to $\frac{S}{\langle x_1x_2\rangle}$ for a $2-$dimensional complete regular local ring $S$ with maximal ideal $\langle x_1,x_2 \rangle$, i.e. $R$ is a double point. Conversely, if $R$ is a double point, then by Proposition~\ref{prop:axes} we have $\delta(R)=r(R)-1=1$, and so $\mu(R)=2-2+1=1$.
\end{proof}

\begin{remark}\label{rem:nodal_interpretation}
    \textup{Lemma~\ref{lem:milnor_nonneg} and Theorem~\ref{thm:Morse} suggest reading the decomposition $\mu(R)=\delta(R)+\left(\delta(R)-r(R)+1\right)$ in terms of double points. For complex plane curves this reading is classical. Let $R$ be the local ring of a reduced plane curve singularity $C \subset \mathbb{C}^2$ with $r=r(R)$ branches. Then $C$ admits a $\delta-$constant deformation whose nearby fibers $C_s$ have, inside a small ball around the singular point, exactly $\delta(R)$ singular points, all of which are nodes, i.e. double points in the sense of Definition~\ref{def:doublept}; moreover, the normalization of $C_s$ is a disjoint union of $r$ discs, one for each branch of $C$ (see, e.g.,~\cite{greuel2007introduction}). Let $\Gamma$ be the graph whose vertices are these discs and whose edges are the nodes of $C_s$, where the edge of a node joins the discs containing its two preimages (and is a loop if both preimages lie on the same disc). Smoothing the nodes of $C_s$ gives the Milnor fiber of $C$, which is connected, and smoothing a node does not change the number of connected components; hence $C_s$ and $\Gamma$ are connected, and $b_1(\Gamma)=\delta(R)-r+1$. By Theorem~\ref{thm:Morse}, every node has Milnor number $1$, and so
    \begin{equation*}
        \mu(R)=\sum_{q \in \operatorname{Sing}(C_s)}\mu\left(\mathcal{O}_{C_s,q}\right)+b_1(\Gamma).
    \end{equation*}
    That is, $\mu(R)$ is the number of double points of a $\delta-$constant nodal deformation plus the number of independent cycles formed by gluing the branches at these double points. This is the Euler characteristic count behind the topological proof of Milnor's formula (compare Chapter 10 of~\cite{milnor1968singular}). For example, for the $A_2$ singularity $y^2=x^3$ one can take $C_s=\{y^2=x^2(x+s^2)\}$: it has a single node, at the origin, and both preimages of the node lie on the unique disc, so $\Gamma$ is one vertex with one loop and $\mu(R)=1+1=2$. In particular, $\Gamma$ is a tree exactly when $\delta(R)=r(R)-1$, which for plane curves only happens for DVRs and double points (Proposition~\ref{prop:axes}).}
\end{remark}

The following lemma shows that, among Gorenstein rings, minimal multiplicity only occurs in multiplicity at most two.

\begin{lemma}\label{lem:gor_minmult}
    Assume that $R$ is Gorenstein. Then $\mult(R)=\edim(R)$ if and only if $\mult(R)\leq 2$. In particular, if $R$ is Gorenstein and $\mult(R)\leq 3$, then $\edim(R)\leq 2$.
\end{lemma}

\begin{proof}
    If $\mult(R)=1$ then $R$ is regular and $\edim(R)=1$. If $\mult(R)=2$ then $R$ is not regular, so $2 \leq \edim(R) \leq \mult(R)=2$ by Remark~\ref{rem:techincal_ring}. Conversely, assume that $e=\mult(R)=\edim(R)\geq 3$. Let $x \in \mathfrak{m} \setminus \mathfrak{m}^2$ be a superficial element (see Remark~\ref{rem:reduction}); it is a non zero divisor (otherwise $\mathfrak{m}^{n+1}=x\mathfrak{m}^n$ would be contained in a minimal prime), and $\length_R\left(\frac{R}{\langle x \rangle}\right)=e$ by Theorem 8.6.6 in~\cite{huneke2006integral} (see also the first item of Lemma~\ref{lem:semigroup}). Since $x \notin \mathfrak{m}^2$ we have $\length_R\left(\frac{\mathfrak{m}}{\mathfrak{m}^2+\langle x \rangle}\right)=\edim(R)-1$, and so
    \begin{equation*}
        e=\length_R\left(\frac{R}{\langle x \rangle}\right) = 1 + (\edim(R)-1) + \length_R\left(\frac{\mathfrak{m}^2+\langle x\rangle}{\langle x \rangle}\right),
    \end{equation*}
    which gives $\mathfrak{m}^2 \subset \langle x \rangle$. Therefore $\mathfrak{m}$ annihilates $\frac{\mathfrak{m}}{\langle x \rangle}$, so the socle of the Artinian ring $\frac{R}{\langle x \rangle}$ has length at least $e-1\geq 2$. Since $x$ is a non zero divisor, $\frac{R}{\langle x \rangle}$ is Gorenstein if $R$ is, and an Artinian Gorenstein local ring has a one dimensional socle. Thus $R$ is not Gorenstein. For the last statement, recall that $\edim(R)\leq \mult(R)$, so if $\mult(R)\leq 3$ and $\edim(R)\geq 3$ then $\mult(R)=\edim(R)=3$.
\end{proof}

\begin{corollary}\label{cor:gor_bound}
    Assume that $R$ is Gorenstein with $r(R)\geq 3$. Then $\delta(R)\geq r(R)$ and $\mu(R)\geq r(R)+1$.
\end{corollary}

\begin{proof}
    By Lemma~\ref{lem:milnor_nonneg} we have $\delta(R) \geq r(R)-1$, and by Proposition~\ref{prop:axes} equality would imply $r(R)\leq 2$. Thus $\delta(R)\geq r(R)$ and $\mu(R)=2\delta(R)-r(R)+1 \geq r(R)+1$. 
\end{proof}

\begin{remark}
    \textup{The bound of Corollary~\ref{cor:gor_bound} is sharp: the $D_4$ singularity $\frac{k[[x,y]]}{\langle xy(x-y)\rangle}$ has $r=\delta=3$, and the complete intersection of Example~\ref{ex:hironaka_fails} has $r=\delta=4$. For rings of embedding dimension two the bound $\mu(R)\geq (r(R)-1)^2$ of Corollary~\ref{cor:lower_bound} is stronger.}
\end{remark}

Inspired by Theorem~\ref{thm:Morse} we now turn to looking at two further generalizations of the double point, namely the ordinary multiple point and the $A_k$ singularities. We show that they have similar classification properties based on the values of $\mu(R)$ and $\delta(R)$.  

\begin{definition}\label{def:omp}
    We say that $R$ is an ordinary multiple point if $\edim(R)\leq 2$ and $\mu(R)=(r(R)-1)^2$.  
\end{definition}

\begin{remark}\label{rem:omp}
\begin{enumerate}
    \item \textup{Note that by the definition of $\mu(R)$, a ring $R$ with $\edim(R)\leq 2$ is an ordinary multiple point if and only if $\delta(R)=\frac{r(R)(r(R)-1)}{2}$. In addition, if we denote the minimal primes of $R$ by $\mathfrak{p}_1, \dots, \mathfrak{p}_{r(R)}$, then by the second items of Proposition~\ref{thm:branch} and Corollary~\ref{cor:lower_bound} together with Corollary~\ref{cor:reg} we can conclude that $R$ (with $\edim(R)\leq 2$) is an ordinary multiple point if and only if $\frac{R}{\mathfrak{p}_i}$ is a DVR for every $i$ and $i(\mathfrak{p}_i, \mathfrak{p}_j)=1$ for every $i \neq j$. }
    \item \textup{Notice that an ordinary multiple point can be thought of as a ring of embedding dimension at most two with the minimal value for $\mu(R)$, i.e. the equality cases in the second items of Proposition~\ref{thm:branch} and Corollary~\ref{cor:lower_bound}. The assumption $\edim(R)\leq 2$ is essential: for $r(R)\geq 3$ the $r(R)$ coordinate axes of Proposition~\ref{prop:axes} satisfy $\mu(R)=r(R)-1<(r(R)-1)^2$. By Corollary~\ref{cor:reg} and Theorem~\ref{thm:Morse}, an ordinary multiple point with $r(R) \leq 2$ is a DVR or a double point.}
\end{enumerate}
\end{remark}

\begin{proposition}\label{prop:omp_edim2}
    Assume that $r(R)\geq 2$. Then $R$ is an ordinary multiple point if and only if $\hat{R}$ is isomorphic to the quotient ring $\frac{S}{\langle f\rangle}$ for some $2$-dimensional complete regular local ring $(S, \mathfrak{n})$ with $\mathfrak{n} = \langle x,y \rangle$ and where $f=\prod_{i=1}^{r(R)} (x+u_iy)$ for units $u_1, \dots, u_{r(R)} \in S$ such that $u_i-u_j$ is a unit for every $i \neq j$.  
\end{proposition}

\begin{proof}
   Assume that $R$ is an ordinary multiple point. Since $r(R)\geq 2$, $R$ is not regular, and so $\edim(R)=2$. By Lemma~\ref{lem:edim_CST} and since $\hat{R}$ is reduced of dimension one, $\hat{R}$ is isomorphic to some quotient $\frac{S}{\langle f \rangle}$ where $(S,\mathfrak{n})$ is a two dimensional complete regular local ring and $f=f_1\cdots f_{r(R)}$ with $f_i \in \mathfrak{n}$ pairwise non-associate primes (as in the proof of Proposition~\ref{thm:branch}). By Remark~\ref{rem:omp} we have that $\frac{R}{\mathfrak{p}_i}$ is a DVR for every $i$, and so $f_i \in \mathfrak{n} \setminus \mathfrak{n}^2$, and that $i(\mathfrak{p}_i, \mathfrak{p}_j)=1$, i.e. $\langle f_i, f_j \rangle = \mathfrak{n}$, for $i \neq j$. Since $\kappa$ is infinite, we can choose generators $x,y$ of $\mathfrak{n}$ such that none of the (finitely many) lines spanned by the images of the $f_i$ in $\frac{\mathfrak{n}}{\mathfrak{n}^2}$ is spanned by the image of $x$ or of $y$. Writing $f_i=a_ix+b_iy$ with $a_i,b_i \in S$, both $a_i$ and $b_i$ are then units, and so $\langle f_i \rangle = \langle x+u_iy \rangle$ with $u_i=\frac{b_i}{a_i}$. Finally, $\langle x+u_iy, x+u_jy\rangle = \langle x+u_iy, (u_j-u_i)y\rangle$ equals $\mathfrak{n}$ if and only if $u_i-u_j$ is a unit, and so $\langle f \rangle = \langle \prod_{i=1}^{r(R)} (x+u_i y) \rangle$ as desired.\\

   Conversely, if $\hat{R}\cong \frac{S}{\langle f \rangle}$ with $f$ as above, then $\edim(R)=\edim(\hat{R}) \leq 2$, the branches $\frac{S}{\langle x+u_iy \rangle}$ are regular, and $\length_S\left(\frac{S}{\langle x+u_iy, x+u_jy\rangle}\right)=1$ for $i\neq j$. Thus by the second item of Proposition~\ref{thm:branch} (applied to $\hat{R}$, together with Remark~\ref{rem:conductor}) we get $\delta(R)=\frac{r(R)(r(R)-1)}{2}$, and $R$ is an ordinary multiple point. 
\end{proof}

\begin{remark}
\begin{enumerate}
    \item \textup{From Proposition~\ref{prop:omp_edim2} and Theorem~\ref{thm:Morse} we can conclude that $R$ is a double point if and only if $R$ is an ordinary multiple point with $r(R)=2$. This is true since if $\mathfrak{n}=\langle x,y \rangle$ then for every units $u, v \in S$ such that $u-v \in S$ is a unit as well we have that $x+uy, x+vy$ also generate $\mathfrak{n}$.  }
    \item \textup{Proposition~\ref{prop:omp_edim2} tells us that we can intuitively think of an ordinary multiple point as "having $r(R)$ distinct tangent lines in the plane". This can be viewed as an analogue of the classical definition of an ordinary multiple point of an algebraic curve, as in Section 3.1 of~\cite{fulton2008algebraic}. For a version of Proposition~\ref{prop:omp_edim2} in arbitrary embedding dimension, see Theorem~\ref{thm:edim_bound} and Remark~\ref{rem:edim_bound}.}
\end{enumerate}
\end{remark}

\begin{definition}\label{def:A_k}
    We say that $R$ is an $A_{k}$ singularity (for $k \geq 1$) if there exists some $2-$dimensional complete regular local ring $(S, \mathfrak{n})$ such that $\mathfrak{n} = \langle x,y \rangle$ and $\hat{R} \cong \frac{S}{\langle x^2 +y^{k+1} \rangle}$. 
\end{definition}

\begin{remark}\label{rem:A_1_doublept}
    \textup{Note that if $\chara(\kappa) \neq 2$, then $R$ is a double point if and only if $R$ is an $A_1$ singularity. This is true since $R$ is strictly Henselian, and so we can find $i \in R$ such that $i^2=-1$. Therefore, $x^2+y^2=(x+iy)(x-iy)$ and $x,y $ generate $\mathfrak{n}$ if and only if $x+iy, x-iy$ do. }
\end{remark}

\begin{proposition}\label{cor:delta1}
    Assume that $\chara(\kappa) \neq 2$. Then:
    \begin{enumerate}
        \item If $\mult(R)=2$ then $R$ is an $A_k$ singularity for some $k \geq 1$.
        \item If $R$ is an $A_k$ singularity then $\mu(R)=k$.
        \item $\delta(R)=1$ if and only if $R$ is either an $A_1$ or $A_2$ singularity.
        \item $\mu(R) = 2$ if and only if $R$ is either an $A_2$ singularity or $\hat{R}$ is isomorphic to $\frac{S}{\langle xy, yz, xz \rangle}$ for some $3-$dimensional complete regular local ring $(S, \mathfrak{n})$ with $\mathfrak{n}=\langle x,y,z\rangle$. In particular, if $\edim(R)\leq 2$ or $R$ is Gorenstein, then $\mu(R)=2$ if and only if $R$ is an $A_2$ singularity.
    \end{enumerate} 
\end{proposition}

\begin{proof}
    For the first item, since $\mult(R)=2$ then by the first item of Remark~\ref{rem:techincal_ring} we have that $\edim(R) \leq 2$, yet since $\mult(R)=2$ then $R$ is not regular and so $\edim(R)>1$. Therefore we can conclude that $\edim(R)=2$ and so by Lemma~\ref{lem:edim_CST}, there exists some $2-$dimensional complete regular $(S, \mathfrak{n})$ and some $f \in S$ such that $\hat{R}=\frac{S}{\langle f \rangle}$, and since $\mult(R)=2$ we can conclude that $f \in \mathfrak{n}^2 \setminus \mathfrak{n}^3$. In addition, $\hat{R}$ is reduced (Remark~\ref{rem:conductor}) and therefore so is $f \in S$. If the quadratic initial form of $f$ has rank $2$, then Proposition 3.14 in \cite{svoray2025ade} gives the $A_1$ case. If it has rank $1$, then Proposition 3.15 in \cite{svoray2025ade} gives an $A_k$ form or the nonreduced $A_\infty$ form, and the latter is excluded because $\hat R$ is reduced. Multiplying the defining equation by a unitand rescaling the square variable, using strict Henselianity and $2\in S^\times$, we obtain a generator $x^2+y^{k+1}$.\\

     For the second item, note that since $\mu$ is preserved under completion, it is enough to compute that $\mu\left(R\right)=k$ where $R=\frac{S}{\langle x^2 + y^{k+1}\rangle}$ for some complete regular $(S, \mathfrak{n})$ with $\mathfrak{n}= \langle x,y \rangle$. Since $S$ is Henselian and $\kappa$ is separably closed of characteristic different from $2$, there exists some $i\in S$ such that $i^2=-1$. If $k=2n-1$ is odd then $x^2+ y^{2n} = (x+iy^n)(x-iy^n)$, so $r(R)=2$, the minimal primes of $R$ are $\langle x+iy^{n}\rangle$ and $\langle x-iy^{n}\rangle$, and both branches are regular. Thus, by Proposition~\ref{thm:branch}, $\delta(R)=\length_S\left(\frac{S}{\langle x+iy^n, x-iy^n\rangle}\right)=\length_S\left(\frac{S}{\langle x, y^n\rangle}\right)=n$ (as $2i$ is a unit), and so $\mu(R)=2n-2+1=k$. If $k=2n$ is even then $R$ is a domain (so $r(R)=1$): otherwise $f=x^2+y^{2n+1}$ would have a factor $g \in \mathfrak{n}\setminus\mathfrak{n}^2$ whose initial form divides $x^2$ (as the associated graded ring of $S$ is a polynomial ring), so in the DVR $\frac{S}{\langle g \rangle}$, whose maximal ideal is generated by the image of $y$, the image of $x$ would have order at least $2$; then the image of $f$ would have order $\min(2\ord(x), 2n+1)<\infty$ (one order is even and the other is odd), contradicting $f \in \langle g \rangle$. Let $t=\frac{x}{y^n} \in \Frac(R)$. Then $t^2=-y$, so $R[t]=R+Rt$ is a finite extension of $R$, which is local (as $R$ is Henselian and $R[t]$ is a domain), and whose maximal ideal is generated by $t$ (as $y=-t^2$ and $x=ty^n$). Thus $R[t]$ is a DVR and $\overline{R}=R[t]$. Therefore $\delta(R)=\length_R\left(\frac{R+Rt}{R}\right)=\length_R\left(\frac{R}{(R:_R t)}\right)$. Note that $\langle x, y^n \rangle \subset (R :_R t)$, as $y^n t = x$ and $xt=-y^{n+1}$. Since the ideals of $R$ containing $\langle x,y^n \rangle$ are of the form $\langle x, y^l \rangle$, and $y^{n-1}t=\frac{x}{y} \notin R$ (as $x \notin \langle y, x^2 \rangle$ in $S$), we get $(R:_R t)=\langle x,y^n \rangle$, and so $\delta(R)=\length_S\left(\frac{S}{\langle x, y^n\rangle}\right)=n$ and $\mu(R)=2n=k$.\\
    
    For the third item, if $\delta(R)=1$ then we have that $R$ is not regular, that is $\edim(R)>1$. In addition, we have that $\mu(R)=3-r(R)$. Therefore, by Corollary~\ref{cor:lower_bound}, either $\mu(R)=1$ (and by Theorem~\ref{thm:Morse} and Remark~\ref{rem:A_1_doublept} we can conclude that $R$ is an $A_1$ singularity) or $\mu(R)=2$ and $r(R)=1$.  Yet, by Lemma~\ref{cor:edim_bound} and the first item of Remark~\ref{rem:techincal_ring} we have that $\edim(R)-1 \leq \mult(R)-1 \leq \delta(R)=1$, and so we can conclude that $\mult(R)=\edim(R)=2$. Thus by the first two items $R$ is an $A_k$ singularity with $k=\mu(R)=2$. The converse follows from the second item.\\

    For the fourth item, if $\mu(R)=2$ then $2\delta(R)=r(R)+1$, so $r(R)$ is odd, and by Lemma~\ref{lem:milnor_nonneg} $r(R)+1 = 2\delta(R) \geq 2r(R)-2$, so $r(R)\in\{1,3\}$. If $r(R)=1$ then $\delta(R)=1$, and by the third item $R$ is an $A_2$ singularity (as $r(A_1)=2$). If $r(R)=3$ then $\delta(R)=2=r(R)-1$ and the result follows from Proposition~\ref{prop:axes}. Conversely, $A_2$ singularities have $\mu=2$ by the second item and the three coordinate axes have $\mu=r-1=2$ by Proposition~\ref{prop:axes}. The last statement follows since the three coordinate axes have embedding dimension $3$ and are not Gorenstein (Proposition~\ref{prop:axes}).
\end{proof}

\begin{remark}
\begin{enumerate}
    \item \textup{The first item of Proposition~\ref{cor:delta1} is an analogue of Theorem 2.48. in~\cite{greuel2007introduction} and the third item is analogue of Example 5.8 in~\cite{ooishi1987genera}.  }
    \item \textup{Note that Theorem 2 in~\cite{d2025upper} tells us that if $R$ is Cohen-Macaulay then $\mult(R) \leq (\edim(R)-1)\cdot\length_R\left(\frac{R}{\mathfrak{c}_R}\right)+1$ and that we have an equality if and only if $R$ is either an DVR or $\length_R\left(\frac{R}{\mathfrak{c}_R}\right)=1$. This, together with the first item of Remark~\ref{rem:conductor} and Corollary~\ref{cor:reg} gives us an alternative proof of the third item of Proposition~\ref{cor:delta1}.}
\end{enumerate}
\end{remark}

We now look at the relationship between the Milnor number of $R$ and the blow up of $R$ at the maximal ideal, known as the first neighborhood of $R$. This in fact allows us to uniquely define $\mu(R)$ based upon its numerical properties. 

\begin{definition}\label{def:blowup}
\begin{enumerate}
    \item The first neighborhood of $R$ is defined to be the (non-local) ring $R^{(1)}=\bigcup_{n \geq 0} \{a \in \overline{R} \colon a\mathfrak{m}^n \subset \mathfrak{m}^n\}$.
    \item We say that $x \in \mathfrak{m} \setminus \mathfrak{m}^2 \subset R$ is superficial if $x\mathfrak{m}^n=\mathfrak{m}^{n+1}$ for $n \gg 0$.
    \item We denote by $\rho(R)$ the $R-$length of the module $\frac{R^{(1)}}{R}$. 
\end{enumerate} 
\end{definition}

\begin{remark}\label{rem:reduction}
\begin{enumerate}
    \item \textup{Note that superficial elements in $R$ exist since $R$ is strictly Henselian, and therefore $\kappa$ is infinite (for a proof, see for example Proposition 8.5.7 in~\cite{huneke2006integral}). In addition, if $x_1, \dots, x_d$ is a minimal generating sequence for $\mathfrak{m}$ in $R$ and $x \in R$ is superficial, then $R^{(1)} = R[\frac{x_1}{x}, \dots, \frac{x_d}{x}]$. For more details, see~\cite{northcott1956general} and Chapter 1.5 in~\cite{campillo2006algebroid}. }
    \item \textup{Since $R$ is strictly Henselian we have that for every birational finite map $R \to S$, the number of maximal ideals in $S$ is bounded above by $r(R)$ (see~\cite[Tag 04GH]{stacks-project}). In particular, if $r(R)=1$ then $R^{(1)}$ is local. }
    \item \textup{The invariant $\rho(R)$ has many relations with invariants such as the multiplicity. By~\cite{northcott1959reduction} we have that if $n \gg 0$ then $\length_R\left(\frac{R}{\mathfrak{m}^n}\right) = n \cdot \mult(R) - \rho(R)$. Indeed, if $x$ is superficial then $\mathfrak{m}^n=x^nR^{(1)}$ for $n \gg 0$, and so $\length_R\left(\frac{R}{\mathfrak{m}^n}\right)=\length_R\left(\frac{R}{x^nR}\right)-\length_R\left(\frac{x^nR^{(1)}}{x^nR}\right)=n\cdot\mult(R)-\rho(R)$. In other words, $\rho(R)$ is the first Hilbert coefficient of $\mathfrak{m}$; moreover, $\length_R\left(\frac{R}{\mathfrak{m}^n}\right) \geq n\cdot \mult(R)-\rho(R)$ for every $n$. For example, for $R=k[[t^2,t^3]]$ we have $\mult(R)=2$, $R^{(1)}=k[[t]]$, $\rho(R)=1$, and $\length_R\left(\frac{R}{\mathfrak{m}^n}\right)=2n-1$ for $n\geq 1$. In addition, in~\cite{kirby1975reduction} it was shown that $\mult(R)-1 \leq \rho(R) \leq \frac{\mult(R) (\mult(R)-1)}{2}$. For more information on bounds on $\rho(R)$, see~\cite{elias1990characterization}. }
\end{enumerate}
\end{remark}

The following Lemma is an analogue of Lemma 1.5.9 in~\cite{campillo2006algebroid}:

\begin{lemma}\label{lem:reg_blowup}
    $R^{(1)}=R$ if and only if $R$ is a DVR. 
\end{lemma}
\begin{proof}
    If $R$ is a DVR then $R$ must be normal, and so $\overline{R}=R$. But since $R \subset R^{(1)} \subset \overline{R}$ we can conclude that $R=R^{(1)}$. Now, assume that $R=R^{(1)}$ and assume towards contradiction that $R$ is not a DVR. Let $x$ be a superficial element of $R$. Since $R$ is not a DVR then $\edim(R)>1$ and so $x$ does not generate $\mathfrak{m}$. Therefore, if $y \in \mathfrak{m} \setminus \langle x \rangle$ then $\frac{y}{x} \in R^{(1)}$ but $\frac{y}{x} \notin R$, which is a contradiction. 
\end{proof}

\begin{lemma}\label{lem:rho_delta}
    If $r(R)=1$ and the residue field of $\overline{R}$ is $\kappa$, then:
    \begin{enumerate}
        \item $\delta(R)=\rho(R)+\delta(R^{(1)})$.
        \item $\mu(R)-\mu(R^{(1)})  = 2\rho(R)$.
    \end{enumerate}
\end{lemma}

\begin{proof}
    By Remark~\ref{rem:reduction} we have that $R^{(1)}$ is local and so $\delta(R^{(1)})$ and $\mu(R^{(1)})$ are well defined; since the residue field of $\overline{R}$ is $\kappa$, so is the residue field of $R^{(1)}$, and lengths over $R$ and over $R^{(1)}$ of $R^{(1)}-$modules of finite length agree. Now, since the normalization of $R^{(1)}$ is $\overline{R}$ then we have a short exact sequence
    \begin{equation*}
        0 \to \frac{R^{(1)}}{R} \to \frac{\overline{R}}{R} \to \frac{\overline{R}}{R^{(1)}} \to 0,
    \end{equation*}
    \noindent and so the result follows by applying to it length over $R$ (and using $r(R^{(1)})=r(R)=1$ for the second item). 
\end{proof}

Since by Remark~\ref{rem:reduction}, if $r(R)=1$ then $R^{(1)}$ is local, then we can define higher neighborhoods of $R$ by inductively applying the first neighborhood of $R$. For more information on higher neighborhoods, see~\cite{northcott1955neighbourhoods, northcott1957notion}.  

\begin{definition}
    Assume that $r(R)=1$. Then the $i-$th neighborhood of $R$ is defined to be $R^{(i)}=\left(R^{(i-1)}\right)^{(1)}$, where $R = R^{(0)}$. 
\end{definition}

The following proposition is an analogue of Proposition 3.26 in Chapter I of~\cite{greuel2007introduction} and of Theorem 1.5.10 in~\cite{campillo2006algebroid}:

\begin{proposition}\label{rem:higher_NBHD}
    Assume that $r(R)=1$ and that $R$ is not regular. Then we have a sequence of inclusions $R \subsetneq R^{(1)} \subsetneq \cdots \subsetneq R^{(N)} = \overline{R}$ for some $N$ with
    \begin{equation*}
        \delta(R) = \sum_{i=1}^N \length_R \left(\frac{R^{(i)}}{R^{(i-1)}}\right).
    \end{equation*}
    In particular, we have that $\delta(R) \geq N$. 
\end{proposition}

\begin{proof}
    Since $\overline{R}$ is a Noetherian $R-$module and since by Remark~\ref{rem:reduction} we have that each $R^{(i)}$ is an $R-$submodule of $\overline{R}$, we must have that the chain of inclusions $R \subset R^{(1)} \subset \cdots $ is stationary. Therefore, there exists some $N$ such that $R^{(N)} = R^{(N+1)}$ and so by Lemma~\ref{lem:reg_blowup} we have that  $R^{(N)}$ is normal, and so $R^{(N)}=\overline{R^{(N)}}=\overline{R}$. The formula for $\delta(R)$ follows from the short exact sequences $0 \to \frac{R^{(i)}}{R^{(i-1)}} \to \frac{\overline{R}}{R^{(i-1)}} \to \frac{\overline{R}}{R^{(i)}} \to 0$. 
\end{proof}

The following remark shows that the Milnor number is determined by the properties we proved before. It is an analogue of Theorem 2.6 in~\cite{ploski1995milnor}. 

\begin{remark}\label{prop:uniqueness}
    \textup{Let $\tilde{\mu}$ be an invariant defined on the class of one-dimensional Noetherian strictly Henselian reduced Nagata local rings which are residually rational (see Definition~\ref{def:res_rat}) such that:
    \begin{enumerate}
        \item $\tilde{\mu}(R)=0$ if $R$ is a DVR, 
        \item If $r(R)=1$ then $\tilde{\mu}(R)-\tilde{\mu}(R^{(1)})  = 2\rho(R)$, 
        \item If $\mathfrak{p}_1, \dots, \mathfrak{p}_{r(R)}$ are the minimal primes of $R$ then \begin{equation*}\tilde{\mu}(R) -1 = \sum_{i=1}^{r(R)} \left(\tilde{\mu}\left(\frac{R}{\mathfrak{p}_i}\right)-1\right) + 2\sum_{j=2}^{r(R)} i(\mathfrak{p}_j, \mathfrak{p}_1\cap \cdots \cap \mathfrak{p}_{j-1}).\end{equation*}
    \end{enumerate}
    Then $\tilde{\mu}=\mu$. Indeed, $\mu$ satisfies these properties by Corollary~\ref{cor:reg}, Lemma~\ref{lem:rho_delta}, and Corollary~\ref{cor:lower_bound}. By the third property it is enough to show $\tilde{\mu}(R)=\mu(R)$ when $r(R)=1$, and then by Proposition~\ref{rem:higher_NBHD} we have a sequence of inclusions of local rings $R \subset R^{(1)} \subset \cdots \subset R^{(N)} = \overline{R}$, so $\tilde{\mu}(R)$ equals $\mu(R)$ by repeated application of the second property and the first property to $\overline{R}$. (Each $R^{(i)}$ is again in our class.)}
\end{remark}

We end this section by refining Corollary~\ref{cor:lower_bound} using the embedding dimension: the bounds $\mu(R)\geq r(R)-1$ and, for $\edim(R)\leq 2$, $\mu(R)\geq (r(R)-1)^2$ are the two extreme cases of a sharp bound depending on $\edim(R)$, whose equality case generalizes Propositions~\ref{prop:axes} and~\ref{prop:omp_edim2}. For integers $e\geq 1$ and $d\geq 0$ we write $C_e(d)=\binom{d+e-1}{e-1}$ for the number of monomials of degree $d$ in $e$ variables, and for $r\geq 1$ and $e\geq 2$ we set
\begin{equation*}
    \Phi(r,e)=\sum_{d\geq 0}\max\left(0,\, r-C_e(d)\right).
\end{equation*}
For example, $\Phi(r,2)=\sum_{d=0}^{r-1}(r-1-d)=\frac{r(r-1)}{2}$; $\Phi(r,e)=r-1$ if $e\geq r$ (as $C_e(0)=1$ and $C_e(d)\geq e$ for $d\geq 1$); and $\Phi(4,3)=4$, $\Phi(5,3)=6$, $\Phi(6,3)=8$. Since $\Phi(r+1,e)-\Phi(r,e)$ is the number of $d$ with $C_e(d)\leq r$, which is at least $1$ as $C_e(0)=1$, the function $\Phi(-,e)$ is strictly increasing.

\begin{definition}\label{def:tangent_directions}
    Assume that every branch $\frac{R}{\mathfrak{p}_i}$ of $R$ is regular, so that $V_i=\frac{R}{\mathfrak{p}_i}$ and $\kappa_i=\kappa$. Let $e=\edim(R)$ and let $x_1,\dots,x_e$ be a minimal set of generators of $\mathfrak{m}$. For every $i$ let $\xi_i=(\xi_{i,1},\dots,\xi_{i,e}) \in \kappa^e$, where $\xi_{i,l}\in\kappa$ is the residue of $\frac{(x_l)_i}{\pi_i}\in V_i$; then $\xi_i \neq 0$, since $R \to V_i$ is surjective and so some element of $\mathfrak{m}$ has order $1$ on the $i-$th branch. We call the point $[\xi_i] \in \mathbb{P}^{e-1}(\kappa)$ the tangent direction of the $i-$th branch. For $d \geq 0$ let $h_\xi(d)$ be the rank of the $\kappa-$linear evaluation map
    \begin{equation*}
        \varepsilon_d \colon \kappa[X_1,\dots,X_e]_d \to \kappa^{r(R)}, \qquad F \mapsto \left(F(\xi_1),\dots,F(\xi_{r(R)})\right),
    \end{equation*}
    defined on the forms of degree $d$. We say that the tangent directions of $R$ are in generic position if $h_\xi(d)=\min(r(R),C_e(d))$ for every $d \geq 0$.
\end{definition}

\begin{remark}\label{rem:tangent_directions}
\begin{enumerate}
    \item \textup{Rescaling $\pi_i$ rescales the $i-$th coordinate of $\varepsilon_d$, and changing the minimal generators $x_1,\dots,x_e$ changes all the $\xi_i$ by the same invertible linear transformation of $\kappa^e$; hence $h_\xi$, and the notion of generic position, do not depend on these choices. Since $h_\xi(d)\leq \min(r(R),C_e(d))$ always holds, generic position means that the tangent directions have the largest possible Hilbert function. In particular, tangent directions in generic position are pairwise distinct: if $[\xi_i]=[\xi_j]$ with $i\neq j$, then the $i-$th and $j-$th coordinates of $\varepsilon_d$ are proportional for every $d$, so $h_\xi(d)<r(R)$ for every $d$.}
    \item \textup{If $e=2$, then pairwise distinct tangent directions are in generic position: a nonzero binary form of degree $d<r(R)$ has at most $d$ zeros in $\mathbb{P}^1(\kappa)$, so $\varepsilon_d$ is injective for $d\leq r(R)-1$; and for $d\geq r(R)-1$ and every $i$, the product of linear forms $\ell_j$ vanishing at $[\xi_j]$ for $j \neq i$ with the $(d-r(R)+1)-$th power of a linear form vanishing at none of the $[\xi_j]$ (which exists as $\kappa$ is infinite) has image a nonzero multiple of the $i-$th unit vector, so $\varepsilon_d$ is surjective.}
    \item \textup{If $e\geq r(R)$, then the tangent directions are in generic position if and only if $\xi_1,\dots,\xi_{r(R)}$ are linearly independent: this is the condition $h_\xi(1)=r(R)$, and if it holds, choosing linear forms $L_i$ with $L_i(\xi_j)=0$ for $j\neq i$ and $L_i(\xi_i)=1$, the images of $L_i^d$ show that $\varepsilon_d$ is surjective for every $d\geq 1$.}
\end{enumerate}
\end{remark}

\begin{theorem}\label{thm:edim_bound}
    Assume that $R$ is not regular, and let $e=\edim(R)$ and $r=r(R)$. Then $\delta(R)\geq \Phi(r,e)$.  Moreover, equality holds if and only if every branch of $R$ is regular and the tangent directions of $R$ are in generic position (Definition~\ref{def:tangent_directions}). In this case the associated graded ring $\operatorname{gr}_{\mathfrak{m}}(R)=\bigoplus_{d\geq 0}\frac{\mathfrak{m}^d}{\mathfrak{m}^{d+1}}$ is isomorphic to $\frac{\kappa[X_1,\dots,X_e]}{I}$, where $I$ is the ideal generated by the forms vanishing at $\xi_1,\dots,\xi_r$; in particular, it is reduced.
\end{theorem}

\begin{proof}
    Let $H(d)=\length_R\left(\frac{\mathfrak{m}^d}{\mathfrak{m}^{d+1}}\right)$, let $x \in \mathfrak{m}$ be a superficial element (Remark~\ref{rem:reduction}), and let $A=\frac{R}{\langle x \rangle}$, with maximal ideal $\mathfrak{m}_A$. As in the proof of Lemma~\ref{lem:gor_minmult}, $x$ is a non zero divisor and $\length_R(A)=\mult(R)$, and $\edim(A)=e-1$ since $x \notin \mathfrak{m}^2$.\\

    First, for every $n \geq 0$ we have an exact sequence
    \begin{equation*}
        0 \to \frac{(\mathfrak{m}^{n+1}:x)}{\mathfrak{m}^n} \to \frac{R}{\mathfrak{m}^n}\xrightarrow{\cdot x}\frac{R}{\mathfrak{m}^{n+1}}\to \frac{A}{\mathfrak{m}_A^{n+1}}\to 0,
    \end{equation*}
    and so $H(n)=\length_R\left(\frac{A}{\mathfrak{m}_A^{n+1}}\right)-\length_R\left(\frac{(\mathfrak{m}^{n+1}:x)}{\mathfrak{m}^n}\right)$. As $\edim(A)=e-1$, each $\frac{\mathfrak{m}_A^k}{\mathfrak{m}_A^{k+1}}$ is spanned by $C_{e-1}(k)$ monomials in $e-1$ generators of $\mathfrak{m}_A$, so $\length_R\left(\frac{A}{\mathfrak{m}_A^{n+1}}\right)\leq \sum_{k=0}^nC_{e-1}(k)=C_e(n)$; and $\length_R\left(\frac{A}{\mathfrak{m}_A^{n+1}}\right)\leq \length_R(A)=\mult(R)$. Hence
    \begin{equation}\label{eq:hilbert_bound}
        H(n)\leq \min\left(\mult(R),C_e(n)\right)-\length_R\left(\frac{(\mathfrak{m}^{n+1}:x)}{\mathfrak{m}^n}\right),
    \end{equation}
    and in particular $H(n)\leq \min\left(\mult(R),C_e(n)\right)$.\\

    Second, by the third item of Remark~\ref{rem:reduction}, $\length_R\left(\frac{R}{\mathfrak{m}^n}\right)=\sum_{d<n}H(d)$ equals $n\cdot\mult(R)-\rho(R)$ for $n\gg 0$. Hence $H(d)=\mult(R)$ for $d \gg 0$ and $\rho(R)=\sum_{d\geq 0}(\mult(R)-H(d))$, and by~\eqref{eq:hilbert_bound} every summand is at least $\max(0,\mult(R)-C_e(d))$, so $\rho(R)\geq \Phi(\mult(R),e)$. The short exact sequence $0 \to \frac{R^{(1)}}{R} \to \frac{\overline{R}}{R} \to \frac{\overline{R}}{R^{(1)}} \to 0$ gives $\delta(R)=\rho(R)+\length_R\left(\frac{\overline{R}}{R^{(1)}}\right)$. Since $\mult(R)\geq r$ (Corollary~\ref{cor:mult_r}) and $\Phi(-,e)$ is increasing, we get 
    \begin{equation*}
        \delta(R)\geq \rho(R)\geq \Phi(\mult(R),e)\geq\Phi(r,e).
    \end{equation*}
    The second form of the bound follows from $\mu(R)=2\delta(R)-r+1$.\\

    Assume now that $\delta(R)=\Phi(r,e)$, so that all the inequalities above are equalities. Since $\Phi(-,e)$ is strictly increasing, $\mult(R)=r$, so by Lemma~\ref{lem:mult_sum_prim} every branch has multiplicity $1$, i.e. is regular; in particular $\kappa_i=\kappa$ for every $i$. Moreover, $R^{(1)}=\overline{R}$, and $H(n)=\min(r,C_e(n))$ for every $n$, so by~\eqref{eq:hilbert_bound} we have $(\mathfrak{m}^{n+1}:x)=\mathfrak{m}^n$ for every $n$; by induction on $k$ this gives $(\mathfrak{m}^{n+k}:x^k)=\mathfrak{m}^n$ for all $n,k\geq 0$. By the third item of Lemma~\ref{lem:semigroup} (whose proof only uses that $\kappa_i=\kappa$ for every $i$), $\sum_i\ord(x_i,V_i)=\length_R(A)=r$, and so $\ord(x_i,V_i)=1$ for every $i$. For $n \geq 0$ let $R(n)=\{g \in R \colon \ord(g_i,V_i)\geq n \text{ for every } i\}$ (as in Lemma~\ref{lem:lengths}). Clearly $\mathfrak{m}^n \subset R(n)$. Conversely, if $g \in R(n)$ then $\frac{g}{x^n} \in \overline{R}=R^{(1)}$. By Remark~\ref{rem:reduction}, $R^{(1)}=R[\frac{\mathfrak{m}}{x}]$, which is the increasing union of the sets $\frac{\mathfrak{m}^k}{x^k}$, so $\frac{g}{x^n}=\frac{a}{x^k}$ for some $k$ and some $a \in \mathfrak{m}^k$. Then $x^kg=x^na \in \mathfrak{m}^{n+k}$, and so $g \in (\mathfrak{m}^{n+k}:x^k)=\mathfrak{m}^n$. Thus $R(n)=\mathfrak{m}^n$ for every $n$.\\

    Now, for $n \geq 0$ let $\lambda_n \colon R(n) \to \kappa^r$ send $g$ to the tuple of residues of $\frac{g_i}{\pi_i^n}\in V_i$. It is $R-$linear (with $R$ acting on $\kappa^r$ through $\kappa$) with kernel $R(n+1)$. Since $(x^\alpha)_i=\pi_i^n\prod_l\left(\frac{(x_l)_i}{\pi_i}\right)^{\alpha_l}$ for a monomial $x^\alpha$ of degree $n$, we have $\lambda_n(x^\alpha)=\varepsilon_n(X^\alpha)$, and as $\mathfrak{m}^n$ is generated by these monomials, $\lambda_n(\mathfrak{m}^n)=\varepsilon_n(\kappa[X_1,\dots,X_e]_n)$. Since $R(n)=\mathfrak{m}^n$ and $R(n+1)=\mathfrak{m}^{n+1}$, the map $\lambda_n$ induces an injective map $\frac{\mathfrak{m}^n}{\mathfrak{m}^{n+1}} \to \kappa^r$ with image $\varepsilon_n(\kappa[X_1,\dots,X_e]_n)$, and so $\min(r,C_e(n))=H(n)=h_\xi(n)$ for every $n$, i.e. the tangent directions are in generic position. Moreover, $\lambda_{n+n'}(gg')=\lambda_n(g)\lambda_{n'}(g')$ (with the coordinatewise product on $\kappa^r$), so these maps define an injective homomorphism of graded rings from $\operatorname{gr}_{\mathfrak{m}}(R)$ to the reduced ring $\kappa^r[T]$, whose image is the image of the homomorphism $\kappa[X_1,\dots,X_e] \to \kappa^r[T]$ sending a form $F$ of degree $d$ to $\varepsilon_d(F)T^d$. The kernel of the latter is $I$, which proves the last statement of the theorem.\\

    Conversely, assume that every branch is regular and that the tangent directions are in generic position. Then $\kappa_i=\kappa$ for every $i$. Define $R(n)$ and $\lambda_n$ as above, and let $I_n=\lambda_n(R(n))$, so that $\length_R\left(\frac{R(n)}{R(n+1)}\right)=\dim_\kappa(I_n)$. Let $N \geq \max_i c_i$ (Notation~\ref{not:cond_vec}). Then $\prod_i\pi_i^NV_i \subset \mathfrak{c}_R \subset R$, so $R(N)=\prod_i\pi_i^NV_i$, and
    \begin{align*}
        \delta(R)&=\length_R\left(\frac{\overline{R}}{R(N)}\right)-\length_R\left(\frac{R}{R(N)}\right)\\
        &=Nr-\sum_{n<N}\dim_\kappa(I_n)=\sum_{n<N}\left(r-\dim_\kappa(I_n)\right).
    \end{align*}
    Since $\mathfrak{m}^n \subset R(n)$, we have $I_n \supset \lambda_n(\mathfrak{m}^n)=\varepsilon_n(\kappa[X_1,\dots,X_e]_n)$, so $\dim_\kappa(I_n)\geq h_\xi(n)=\min(r,C_e(n))$. Therefore $\delta(R)\leq \sum_{n\geq 0}(r-\min(r,C_e(n)))=\Phi(r,e)$, and equality holds by the first part of the proof. 
\end{proof}

\begin{remark}\label{rem:edim_bound}
\begin{enumerate}
    \item \textup{If $\edim(R)\geq r(R)$, then $\Phi(r(R),\edim(R))=r(R)-1$, and Theorem~\ref{thm:edim_bound} is the bound $\mu(R)\geq r(R)-1$ of Corollary~\ref{cor:lower_bound}. In the equality case $r(R)=\mult(R)\geq \edim(R)$ (Remark~\ref{rem:techincal_ring}), so $\edim(R)=r(R)$, and by the third item of Remark~\ref{rem:tangent_directions} the tangent directions are linearly independent; this agrees with Proposition~\ref{prop:axes}, where they are the coordinate points of $\mathbb{P}^{r(R)-1}$. If $e=\edim(R)$ and $r=r(R)$ satisfy $2\leq e<r$, then $\Phi(r,e)\geq (r-1)+(r-e)>r-1$, so the bound $\mu(R)\geq r-1$ is never attained.}
    \item \textup{If $\edim(R)=2$, then $\Phi(r(R),2)=\frac{r(R)(r(R)-1)}{2}$, and Theorem~\ref{thm:edim_bound} is the bound $\mu(R)\geq (r(R)-1)^2$ of Corollary~\ref{cor:lower_bound}. By the second item of Remark~\ref{rem:tangent_directions}, equality holds if and only if the branches are regular with pairwise distinct tangent directions, i.e. $R$ is an ordinary multiple point; in the notation of Proposition~\ref{prop:omp_edim2} the tangent direction of the branch $\langle x+u_iy\rangle$ is $[-\overline{u_i}:1]$, where $\overline{u_i}$ is the residue of $u_i$.}
    \item \textup{If $\edim(R)=3$ then $\mu(R)\geq 5$, $8$, and $11$ for $r(R)=4$, $5$, and $6$, respectively. The bound is attained for all $2 \leq e \leq r$. Indeed, since $\kappa$ is infinite there are $\xi_1,\dots,\xi_r \in \kappa^e$ whose classes are in generic position: for each $d$, the determinant of the matrix $(\xi_i^\alpha)$, for $\min(r,C_e(d))$ of the points and as many distinct monomials $\alpha$ of degree $d$, is a nonzero polynomial in the coordinates of the $\xi_i$, and only finitely many $d$ have to be considered, since $h_\xi(d)=r$ implies $h_\xi(d+1)=r$ (multiply by a linear form vanishing at no $[\xi_i]$). If $L_i \subset S=\kappa[[X_1,\dots,X_e]]$ is the prime ideal of the line spanned by $\xi_i$, then $R=\frac{S}{L_1\cap \cdots \cap L_r}$ has regular branches $\frac{S}{L_i}\cong \kappa[[t]]$ (via $X \mapsto \xi_it$) with tangent directions $[\xi_i]$, and $\edim(R)=e$, as $h_\xi(1)=e$ means that no linear form lies in $L_1\cap\cdots\cap L_r$. Thus $\delta(R)=\Phi(r,e)$. For instance, if $\chara(\kappa)\neq 2$, the four lines of Example~\ref{ex:hironaka_fails} have tangent directions $[0:1:\pm 1]$ and $[1:0:\pm 1]$, which are in generic position in $\mathbb{P}^2$ (no three of them are collinear, so they impose independent conditions on conics), and indeed $\delta(R)=4=\Phi(4,3)$ and $\mu(R)=5$.}
    \item \textup{If $\edim(R)>2$, then pairwise distinct tangent directions do not suffice for equality. Assume that $\chara(\kappa)\neq 2$, and let $R$ be the image of $\kappa[[X,Y,Z]]$ in $\kappa[[t]]^4$ under $X\mapsto (t,0,t,t)$, $Y\mapsto(0,t,t,2t)$, and $Z\mapsto(t^2,t^3,0,t^2)$, i.e. the union of the four regular branches $(t,0,t^2)$, $(0,t,t^3)$, $(t,t,0)$, and $(t,2t,t^2)$. With the notation of the proof of Theorem~\ref{thm:edim_bound}, $I_0=\kappa\cdot\mathds{1}$, $I_1$ is spanned by $\lambda_1(X)=(1,0,1,1)$ and $\lambda_1(Y)=(0,1,1,2)$, and $I_2$ contains $\lambda_2(X^2)=(1,0,1,1)$, $\lambda_2(XY)=(0,0,1,2)$, $\lambda_2(Y^2)=(0,1,1,4)$, and $\lambda_2(Z)=(1,0,0,1)$, which span $\kappa^4$ as $\chara(\kappa)\neq 2$. As $I_aI_b \subset I_{a+b}$ and $I_1$ contains an element with all coordinates nonzero, $I_n=\kappa^4$ for every $n\geq 2$, and so $\delta(R)=3+2=5$. If $\edim(R)$ were $2$, then $\delta(R)=\Phi(4,2)=6$ by the second item, since the tangent directions are distinct; thus $\edim(R)=3$, the tangent directions are $[1:0:0]$, $[0:1:0]$, $[1:1:0]$, and $[1:2:0]$, which are distinct but collinear, and $\delta(R)=5>4=\Phi(4,3)$.}
\end{enumerate}
\end{remark}

\section{The Semigroups $\Gamma(R)$ and $\nu(R)$}\label{sec:semigroup}

In this section we define two semigroup invariants of $R$, and we  show how they relate to $\mu(R)$. Again, we assume that $(R, \mathfrak{m}, \kappa)$ is a one-dimensional Noetherian strictly Henselian reduced Nagata local ring, and we use Notation~\ref{notation:ringstuff}. In addition, from now on we assume that $R$ is residually rational, in the following sense:

\begin{definition}\label{def:res_rat}
    We say that $R$ is residually rational if for every minimal prime $\mathfrak{p}_i$ of $R$ the map $\kappa \to \kappa_i$ is an isomorphism, where $\kappa_i$ is the residue field of the normalization $V_i$ of $\frac{R}{\mathfrak{p}_i}$. 
\end{definition}

\begin{example}\label{ex:insep}
\textup{This assumption holds, for example, if $\kappa$ is algebraically closed, but it does not follow from our standing assumptions, since a strictly Henselian ring only has a separably closed residue field. For example, let $\kappa$ be a separably closed field of characteristic $p>0$ which is not perfect, let $a \in \kappa \setminus \kappa^p$, and let $R=\frac{\kappa[[x,y]]}{\langle y^p-ax^p \rangle}$. Then $R$ is a complete local domain, and since $\left(\frac{y}{x}\right)^p=a$ we have $\overline{R}=\kappa(a^{1/p})[[x]]$, whose residue field is not $\kappa$. Here $\length_R\left(\frac{R}{\langle x \rangle}\right)=p$ while $\ord(x,\overline{R})=1$, and the semigroup $\Gamma(R)$ defined below is $p\mathbb{N}$. Without residual rationality, the length formula becomes $l(g,R)=\sum_i[\kappa_i:\kappa]\ord(g_i,V_i)$. However, the value semigroup alone need not determine $\delta(R)$
or finite Cohen--Macaulay type. The residual-rationality hypothesis
is therefore essential for the semigroup reconstruction and
classification results below.}
\end{example}

\begin{definition}
    For every $g \in \NZD(R)$ we define $l(g, R)=\length_R\left(\frac{R}{\langle g \rangle}\right) \in \mathbb{N}$. The set of all values $l(g,R)$ is denoted by $\Gamma(R)$. 
\end{definition}

\begin{remark}\label{rem:intersection}
\begin{enumerate}
    \item \textup{The motivation for $\Gamma(R)$ comes from the intersection number over algebraic curves. Specifically, given some field $k$, if $f \in k[[x,y]]$ is irreducible, then $\Gamma\left(\frac{k[[x,y]]}{\langle f\rangle}\right)$ equals to the collection of all intersection numbers $i(\langle f\rangle, \langle g\rangle)$ where $g \notin \langle f \rangle$, which is exactly the semigroup of values of $f$. For more information, see~\cite{barroso2012approach,hefez2003irreducible,castellanos2005semigroup} }
    \item \textup{Note that $\Gamma(R)$ is non trivial since for every unit $u \in R$ we have that $l(u,R)=0 \in \Gamma(R)$. In addition, by~\cite[Tag 00LD]{stacks-project},  $\NZD(R)$ is exactly the complement of the union $\mathfrak{p}_1 \cup \cdots \cup \mathfrak{p}_{r(R)}$, and so is non trivial. }
\end{enumerate}
\end{remark}

The next lemma tells us that $\Gamma(R)$ is a semigroup, in addition to some properties that we use throughout this section.  

\begin{lemma}\label{lem:semigroup}
Let $g \in \NZD(R)$. Then: 
    \begin{enumerate}
        \item If $g$ is superficial then $l(g,R)=\mult(R)$. 
        \item There exists a Zariski open set $U \subset \frac{\mathfrak{m}}{\mathfrak{m}^2}$ such that $l(g,R)=\mult(R)$ for every $g+\mathfrak{m}^2 \in U$. 
        \item $l(g,R)=\sum_{i=1}^{r(R)} \ord(g_i, V_i)$. In particular, if  $r(R)=1$ then $l(g,R)$ equals to $\ord(g, \overline{R})= \max\{N \colon g \in \langle \pi \rangle^N\}$, where $\pi$ is a uniformizer of $\overline{R}$.
        \item If $g \in \mathfrak{m}^n$ then $l(g,R) \geq n \cdot \mult(R)$.
        \item $\min(\Gamma(R)\setminus \{0\})=\mult(R)$.
        \item Given $g, h \in \NZD(R)$ we have that $l(gh, R)=l(g, R)+l(h,R)$.  
    \end{enumerate}
\end{lemma}

\begin{proof}
The first and second item follow from Theorem 8.6.6 and Proposition 11.2.2 in~\cite{huneke2006integral}, respectively (with the first item also being based upon~\cite{northcott1954reductions, trung2003constructive}). For the third item, note that $g_i \neq 0$ for every $i$, so $g$ is a non zero divisor on $\overline{R}$. Applying the snake lemma to multiplication by $g$ on the short exact sequence $0 \to R \to \overline{R} \to \frac{\overline{R}}{R} \to 0$, and using that the kernel and cokernel of multiplication by $g$ on the finite length module $\frac{\overline{R}}{R}$ have the same length, we get $l(g,R)=\length_R\left(\frac{\overline{R}}{g\overline{R}}\right)=\sum_i \length_R\left(\frac{V_i}{g_iV_i}\right)=\sum_i \ord(g_i,V_i)$, where the last equality uses that $R$ is residually rational.\\

For the fourth item, since $\overline{R}$ is a principal ideal ring we can write $\mathfrak{m}\overline{R}=(\pi_1^{m_1},\dots,\pi_{r(R)}^{m_{r(R)}})\overline{R}$, where $m_i=\min\{\ord(a_i,V_i) \colon a \in \mathfrak{m}\}$, and by Section 11.2 in~\cite{huneke2006integral} (as in Lemma~\ref{cor:edim_bound}) $\mult(R)=\length_R\left(\frac{\overline{R}}{\mathfrak{m}\overline{R}}\right)=\sum_i m_i$. If $g \in \mathfrak{m}^n$ then $\ord(g_i,V_i)\geq n m_i$ for every $i$, and the claim follows from the third item. The fifth item follows from the fourth item (with $n=1$) together with the first item, recalling that superficial elements exist (Remark~\ref{rem:reduction}).\\

For the sixth item, it is enough to show that $\frac{R}{\langle h \rangle} \cong \frac{\langle g \rangle}{\langle gh \rangle}$, as this isomorphism gives us the short exact sequence 

\begin{equation*}
    0 \to \frac{R}{\langle h \rangle} \to \frac{R}{\langle gh \rangle} \to \frac{R}{\langle g \rangle} \to 0.
\end{equation*}

\noindent In order to prove the isomorphism, we look at the map $R \to \frac{\langle g \rangle}{\langle gh \rangle}$ defined by $f \mapsto fg$. Assume that $fg \in \langle gh \rangle$. Then there exists some $a \in R$ such that $fg = gha$, and so $g(f-ha)=0$. Since $g \in \NZD(R)$, we must have that $f=ha$, and so $f \in \langle h \rangle$. This gives us that the kernel of the map is exactly $\langle h \rangle$, and the result follows. 
\end{proof}

\begin{remark}\label{rem:semigroup}
\begin{enumerate}
    \item \textup{In Section 2 of~\cite{svoray2025detecting}, an analogue of the Milnor and Tjurina numbers is defined over a mixed characteristic DVR as the multiplicity of a specific one dimensional ring. Their definitions, which is based upon Proposition 3.3. in that paper, uses an analogue of items 1 and 2 of Lemma~\ref{lem:semigroup}, as item 2 tells us that $l(g,R)$ equals to the multiplicity of $R$ for a "generic element in $\mathfrak{m} \setminus \mathfrak{m}^2$".}
    \item \textup{Items 4 and 5 in Lemma~\ref{lem:semigroup} give us analogous results to those of the intersection number of algebraic curves (as in Remark~\ref{rem:intersection}). For more information, see Section 3.1 of Chapter I in~\cite{greuel2007introduction}.  }
    \item     \textup{The third item in Lemma~\ref{lem:semigroup} can be viewed as an analogue of Proposition 1.1 in~\cite{ploski1995milnor} and of Proposition 3.2 in Chapter I of ~\cite{greuel2007introduction}. Specifically, if $r(R)=1$ we can think of the normalization map $R \to \overline{R}$ as a generalization of the idea of "parametrization of a curve", as it allows us to write the elements of $R$ as power series in the uniformizer of $\overline{R}$. For more on parametrization of curves, see Section 3.4 of Chapter I in~\cite{greuel2007introduction} and Chapter 1 of~\cite{campillo2006algebroid}.}
    \item \textup{As mentioned above, item 6 of Lemma~\ref{lem:semigroup} tells us that $\Gamma(R) \subset \mathbb{N}$ is a semigroup, i.e. contains zero and closed under addition. In fact, as we see in Proposition~\ref{prop:semigrp}, $\Gamma(R)$ is a numerical semigroup, i.e $\gcd(\Gamma(R))=1$, which is in fact equivalent to having that $N \in \Gamma(R)$ for every $N \gg 0$ (for a proof, see Proposition 1.1 in~\cite{zariski2006moduli}). Given a numerical semigroup $S \subset \mathbb{N}$, the smallest $c$ which has the property that  if $N \geq c$ then $N \in S$ is called the conductor of $S$, and as we see in Proposition~\ref{prop:mu_vec} and Corollary~\ref{cor:Gamma_cond}, the conductor of $R$ (as defined in Notation~\ref{notation:ringstuff}) plays an analogous role for the conductor of $\Gamma(R)$. If $r(R)=1$ and $R$ is Gorenstein, the conductor of $\Gamma(R)$ is $2\delta(R)$, and in particular it is even; in general it need not be even, e.g. $\Gamma(k[[t^3,t^4,t^5]])=\lceil 3,4,5\rfloor$ has conductor $3$. For more information on semigroups and their relation to curves and rings, see~\cite{assi2020numerical} or Section 1.1 in~\cite{barroso2012approach}.}
\end{enumerate}
\end{remark}

\begin{corollary}\label{Cor:DVR_Gamma}
    $R$ is a DVR if and only if $\Gamma(R)=\mathbb{N}$. 
\end{corollary}
\begin{proof}
    By Lemma~\ref{lem:semigroup} we have that $\Gamma(R)=\mathbb{N}$ if and only if $\mult(R)=1$, which is equivalent to $R$ being regular. 
\end{proof}

\begin{remark}
    \textup{Note that if $r(R)=1$ then by Proposition~\ref{rem:higher_NBHD} we have a sequence of inclusions $R \subsetneq R^{(1)} \subsetneq \cdots \subsetneq R^{(N)} = \overline{R}$, and so by Lemma~\ref{lem:semigroup} and Corollary~\ref{Cor:DVR_Gamma}  we have a sequence of inclusions $\Gamma(R) \subsetneq \Gamma(R^{(1)}) \subsetneq \cdots \subsetneq \Gamma(R^{(N)}) = \mathbb{N}$ (the inclusions are strict since the number of elements of $\mathbb{N}\setminus \Gamma$ is the delta invariant, see Remark~\ref{rem:mu-r(R)-gamma}, and the delta invariant decreases under blow up by Lemma~\ref{lem:rho_delta}). In particular, we can conclude that $\mult(R) \geq \mult(R^{(1)}) \geq \cdots \geq \mult(R^{(N)}) = 1$. }
\end{remark}

\begin{lemma}\label{lem:l_irred}
     Assume that $\mathfrak{p}_1, \dots, \mathfrak{p}_{r(R)}$ are the minimal primes of $R$. Then for every $g \in \NZD(R)$ we have that $l(g,R)=\sum_{i=1}^{r(R)} l(g_i, \frac{R}{\mathfrak{p}_i})$, where $g_i$ is the image of $g$ under the projection $R \to \frac{R}{\mathfrak{p}_i}$.
\end{lemma}

\begin{proof}
    Note that since $R \to \frac{R}{\mathfrak{p}_i}$ is a surjective map, then by~\cite[Tag 00IX]{stacks-project} we have that $l(g_i, \frac{R}{\mathfrak{p}_i}) = \length_R \left(\frac{R}{\mathfrak{p}_i + \langle g \rangle}\right)$. As in Lemma~\ref{lem:mult_sum_prim}, denote $\mathfrak{q}=\bigcap_{i=1}^{r(R)-1} \mathfrak{p}_i$ and recall that $\mathfrak{q}\cap \mathfrak{p}_{r(R)} = \{0\}$. Now, we have a short exact sequence 
    \begin{equation*}
        0 \to R \to \frac{R}{\mathfrak{q}} \oplus \frac{R}{\mathfrak{p}_{r(R)} } \to \frac{R}{\mathfrak{q}+\mathfrak{p}_{r(R)} } \to 0,
    \end{equation*}
     \noindent and so by applying to it $-\otimes_R \frac{R}{\langle g \rangle}$ (and noting that $g$ is a non zero divisor on $\frac{R}{\mathfrak{q}}$ and on $\frac{R}{\mathfrak{p}_{r(R)}}$) we have an exact sequence 
    \begin{equation*}
         0 \to  \Tor^R_1\left( \frac{R}{\mathfrak{q}+\mathfrak{p}_{r(R)} }, \frac{R}{\langle g \rangle} \right) \to  \frac{R}{\langle g \rangle} \to \frac{R}{\mathfrak{q} + \langle g \rangle} \oplus \frac{R}{\mathfrak{p}_{r(R)} + \langle g \rangle} \to \frac{R}{\mathfrak{q}+\mathfrak{p}_{r(R)} + \langle g \rangle} \to 0. 
    \end{equation*}
    \noindent Yet, we have that $\Tor^R_1\left( \frac{R}{\mathfrak{q}+\mathfrak{p}_{r(R)}}, \frac{R}{\langle g \rangle} \right)$ is exactly $( 0 \colon g) = \{ x \in \frac{R}{\mathfrak{q}+\mathfrak{p}_{r(R)}} \colon gx =0\}$. Since $\frac{R}{\mathfrak{q}+\mathfrak{p}_{r(R)}}$ is an Artinian local ring, by looking at the exact sequence
    \begin{equation*}
        0 \to (0 : g) \to \frac{R}{\mathfrak{q}+\mathfrak{p}_{r(R)}} \xrightarrow{x \mapsto gx}  \frac{R}{\mathfrak{q}+\mathfrak{p}_{r(R)}}  \to  \frac{R}{\mathfrak{q}+\mathfrak{p}_{r(R)}+ \langle g \rangle} \to 0,
    \end{equation*}
    \noindent we can conclude that the $R-$length of $\frac{R}{\mathfrak{q}+\mathfrak{p}_{r(R)}+ \langle g \rangle} $ is the same as the $R-$length of $\Tor^R_1\left( \frac{R}{\mathfrak{q}+\mathfrak{p}_{r(R)}}, \frac{R}{\langle g \rangle} \right)$. Therefore we have that $\length_R \left ( \frac{R}{\langle g \rangle} \right) = \length_R \left(\frac{R}{\mathfrak{q} + \langle g \rangle}\right) +  \length_R \left( \frac{R}{\mathfrak{p}_{r(R)}+ \langle g \rangle}\right)$, and the result follows by repeating the argument with respect to $\mathfrak{q}$. 
\end{proof}

\begin{remark}\label{cor:sum_length}
\begin{enumerate}
    \item \textup{By the third item of Lemma~\ref{lem:semigroup} (applied to $R$ and to each $\frac{R}{\mathfrak{p}_i}$), if $g \in \NZD(R)$ then $l\left(g_i, \frac{R}{\mathfrak{p}_i}\right)=\ord\left(g_i, V_i\right)$, and $l(g,R) = \sum_{i=1}^{r(R)} \ord\left(g_i, V_i\right)$.}
    \item \textup{We can use Lemma~\ref{lem:l_irred} to give an alternative proof to Lemma~\ref{lem:mult_sum_prim}. As in the second item of Lemma~\ref{lem:semigroup}, We can find sets $U, U_1, \dots, U_{r(R)} \subset \mathfrak{m}$ that are Zariski open sets in $\frac{\mathfrak{m}}{\mathfrak{m}^2}$ and $\frac{\mathfrak{m}}{\mathfrak{m}^2 + \mathfrak{p}_i}$ for every $i$ (respectively). Since they are Zariski open set, their intersection must be non-empty. Thus, we can find some $r \in R$ such that $l(r,R)=\mult(R)$ and $l\left(r_i,\frac{R}{\mathfrak{p}_i}\right)=\mult\left(\frac{R}{\mathfrak{p}_i}\right)$, and so the result follows.} 
\end{enumerate}
\end{remark}

We now turn to defining $\nu(R)$,  a version of $\Gamma(R)$ that is a subset of $\mathbb{N}^{r(R)}$, which can be thought of as a "component-wise" version of $\Gamma(R)$ based upon the minimal primes of $R$. The semigroup $\nu(R)$ has been studied in the past, both for a reduced algebraic curve over a field with multiple branches and for one dimensional local ring. For more information, see Section 3 of Chapter I in~\cite{greuel2007introduction} and~\cite{de1987semigroup,campillo1994gorenstein, d2025value, zariski2006moduli}. 

\begin{definition}\label{def:nu}
    $\nu(R)$ is the set of all tuples $\nu(g) = \left( l\left(g_i,\frac{R}{\mathfrak{p}_i}\right) \right)_{i=1}^{r(R)} \in \mathbb{N}^{r(R)}$ for every $g \in \NZD(R)$ (where $g_i$ is the image of $g$ under the map $R \to \frac{R}{\mathfrak{p}_i}$). 
\end{definition}

The following proposition summarizes some important properties that $\nu(R)$ has:

\begin{proposition}\label{prop:nu_beasic_prop}
    Let $\mathfrak{p}_1, \dots, \mathfrak{p}_{r(R)}$ be the minimal primes of $R$. Then we have that:
    \begin{enumerate}
        \item $(N_1, \dots, N_{r(R)}) \in \nu(R)$ if and only if there exist some units $u_i \in V_i$ for which we have that $(u_1\pi_1^{N_1}, \dots, u_{r(R)}\pi_{r(R)}^{N_{r(R)}}) \in R \subset \overline{R} = \prod_{i=1}^{r(R)} V_i$.
        \item The minimal tuple in $\nu(R) \setminus \{(0, \dots, 0)\}$ is $\left(\mult\left(\frac{R}{\mathfrak{p}_1}\right), \dots, \mult\left(\frac{R}{\mathfrak{p}_{r(R)}}\right)\right)$. 
        \item $\Gamma(R)$ is the image of $\nu(R)$ under the linear map $s \colon \mathbb{N}^{r(R)} \to \mathbb{N}$ given by $s(x_1,\dots, x_{r(R)})=x_1+\cdots+x_{r(R)}$.
        \item The image of $\nu(R)$ under the projection $\mathbb{N}^{r(R)} \to \mathbb{N}^{r(R)-1}$ defined by $(N_1, \dots, N_{r(R)}) \mapsto (N_1, \dots, N_{i-1}, N_{i+1}, \dots,  N_{r(R)})$ is exactly $\nu\left(\frac{R}{\mathfrak{q}_i} \right)$ where $\mathfrak{q}_i=\bigcap_{j \neq i} \mathfrak{p}_j$. 
        \item The image of $\nu(R)$ under the projection on the $i-$th coordinate is exactly $\Gamma\left(\frac{R}{\mathfrak{p}_i}\right)$. 
    \end{enumerate}
\end{proposition}

\begin{proof}
    The first three items follow from Remark~\ref{rem:techincal_ring}, Lemmas~\ref{lem:semigroup} and~\ref{lem:l_irred}, and Remark~\ref{cor:sum_length} (for the second item, note that $\mult\left(\frac{R}{\mathfrak{p}_i}\right)=m_i$ in the notation of the proof of Lemma~\ref{lem:semigroup}, and that a generic element of $\mathfrak{m}$ realizes all the $m_i$ simultaneously, as in Remark~\ref{cor:sum_length}). For the fourth item, if $g \in \NZD(R)$ then its image in $\frac{R}{\mathfrak{q}_i}$ is a non zero divisor with the same values on the branches $j \neq i$. Conversely, let $\overline{g} \in \NZD\left(\frac{R}{\mathfrak{q}_i}\right)$ and let $g \in R$ be a lift, so $g \notin \mathfrak{p}_j$ for $j \neq i$. Since the minimal primes are distinct, $\mathfrak{q}_i \not\subset \mathfrak{p}_i$, so we can choose $h \in \mathfrak{q}_i \setminus \mathfrak{p}_i$. If $g \in \mathfrak{p}_i$ we replace $g$ by $g+h$; then $g \in \NZD(R)$, and since $h \in \mathfrak{p}_j$ for $j\neq i$, the values of $g$ on the branches $j\neq i$ are unchanged. The fifth item follows from the fourth by induction (or directly by the same argument).
\end{proof}

\begin{remark}\label{rem:Gam_comp}
\begin{enumerate}
    \item \textup{Note that $\Gamma(R)=\Gamma(\hat{R})$ and $\nu(R) =\nu(\hat{R})$. Indeed, by Remark~\ref{rem:conductor} the normalization of $\hat{R}$ is $\prod_i\hat{V_i}$, with the same uniformizers and residue fields, so $\nu(g)$ is the same whether $g \in \NZD(R)$ is viewed in $R$ or in $\hat{R}$; conversely, a non zero divisor $\hat{g}$ of $\hat{R}$ can be approximated by some $g \in R$ with $\hat{g}-g \in \mathfrak{m}^N\hat{R}\subset\prod_i\pi_i^N\hat{V_i}$ for $N$ larger than all the coordinates of $\nu(\hat{g})$, and then $\nu(g)=\nu(\hat{g})$. The statement for $\Gamma$ follows from the third item of Proposition~\ref{prop:nu_beasic_prop}.   }
    \item \textup{From the fifth item of Proposition~\ref{prop:nu_beasic_prop} we can conclude that given $N_i \in \Gamma\left(\frac{R}{\mathfrak{p}_i}\right)$ for some $i$, we can find some $(N_1, \dots, N_{r(R)}) \in \nu(R)$ whose $i-$th coordinate is $N_i$, and so $N_1 + \cdots + N_{r(R)} \in \Gamma(R)$.}
    \item \textup{Note that the intersection of $\nu(R)$ with any of the coordinate hyperplanes must be the origin. This true since if $(N_1, \dots, N_{r(R)})\in \nu(R)$ such that $N_i=0$ for some $i$ then by Proposition~\ref{prop:nu_beasic_prop}, there is some $g$ such that $l(g_i, \frac{R}{\mathfrak{p}_i})=0$, and so $g_i$ is a unit. But since $g_i$ is the image of $g$ under the map $R \to \frac{R}{\mathfrak{p}_i}$, we can conclude that $g$ must be a unit and so $\nu(g)=(0,\dots,0)$. }
\end{enumerate}
\end{remark}

\begin{lemma}\label{lem:ord_ideal}
    Given $g,h \in \NZD(R)$, if $\nu(g)=\nu(h)$ then $h=gu$ for some unit $u \in \overline{R}$. 
\end{lemma}

\begin{proof}
    Write $g=(g_1, \dots, g_{r(R)})$ and $h=(h_1, \dots, h_{r(R)})$ as elements in $\overline{R}=\prod_i V_i$. By Remark~\ref{cor:sum_length} we have that $N_i= \ord(g_i, V_i) = \ord(h_i, V_i)$. Therefore, we can conclude that $g_i = u_i \pi_i^{N_i}$ and $h_i = v_i \pi_i^{N_i}$ for some units $u_i, v_i \in V_i$, and so $h=gu$ with $u=(v_iu_i^{-1})_i$. 
\end{proof}

\begin{remark}
    \textup{In Lemma~\ref{lem:ord_ideal}, we can not conclude that $u\in R$. That is, if two elements in $R$ generate the same ideal in $\overline{R}$ we can not conclude that they generate the same ideal in $R$. For example, let $R=\frac{k[[x,y]]}{\langle x^2 - y^3 \rangle}$ for some field $k$, so that $R \cong k[[t^2, t^3]] \subset k[[t]] = \overline{R}$ with $y \mapsto t^2$ and $x \mapsto t^3$. Then $y$ and $x+y \mapsto t^2(1+t)$ both have value $2$, but $\langle y \rangle \neq \langle x+y\rangle$ in $R$, as $1+t \notin R$. }
\end{remark}

Since $\overline{R}=\prod_i V_i$ is a principal ideal ring, the conductor $\mathfrak{c}_R$, viewed as an ideal of $\overline{R}$, is generated by an element of the form $(\pi_1^{c_1}, \dots, \pi_{r(R)}^{c_{r(R)}})$. 

\begin{notation}\label{not:cond_vec}
    We call $\vec{c}(R)=(c_1, \dots, c_{r(R)}) \in \mathbb{N}^{r(R)}$ the conductor vector of $R$. Note that $c_i$ is the smallest integer $n$ such that $\pi_i^nV_i \subset R$, where we view $V_i \subset \overline{R}$ as the $i-$th factor, since an $\overline{R}-$ideal $\prod_i \pi_i^{n_i}V_i$ is contained in $R$ if and only if each $\pi_i^{n_i}V_i$ is.
\end{notation}

\begin{lemma}\label{lem:cond_min}
    We have $\vec{c}(R)+\mathbb{N}^{r(R)} \subset \nu(R)$, and $\vec{c}(R)$ is the minimal element of $\mathbb{N}^{r(R)}$ with this property: if $\vec{\alpha} \in \mathbb{N}^{r(R)}$ satisfies $\vec{\alpha}+\mathbb{N}^{r(R)} \subset \nu(R)$ then $\vec{\alpha} \geq \vec{c}(R)$.
\end{lemma}

\begin{proof}
    For $\vec{\beta} \in \mathbb{N}^{r(R)}$, the element $(\pi_1^{c_1+\beta_1}, \dots, \pi_{r(R)}^{c_{r(R)}+\beta_{r(R)}})$ lies in $\mathfrak{c}_R \subset R$, is a non zero divisor, and has $\nu=\vec{c}(R)+\vec{\beta}$ by Remark~\ref{cor:sum_length}. Conversely, assume that $\vec{\alpha}+\mathbb{N}^{r(R)} \subset \nu(R)$, and for $\vec{\beta} \in \mathbb{N}^{r(R)}$ let $I_{\vec{\beta}}=\{z \in \overline{R} \colon \ord(z_i, V_i)\geq \beta_i \text{ for every } i\}=(\pi_1^{\beta_1},\dots, \pi_{r(R)}^{\beta_{r(R)}})\overline{R}$. We claim that for every $\vec{\beta} \geq \vec{\alpha}$ and every $z \in I_{\vec{\beta}}$ there is some $y \in R \cap I_{\vec{\beta}}$ with $z-y \in I_{\vec{\beta}+\mathds{1}}$. Given the claim, starting with $z \in I_{\vec{\alpha}}$ we find $y_1, \dots, y_N \in R$ with $z-(y_1+\cdots+y_N) \in I_{\vec{\alpha}+N\mathds{1}}$, which is contained in $\mathfrak{c}_R \subset R$ for $N \gg 0$. Thus $z \in R$, i.e. $I_{\vec{\alpha}}$ is an $\overline{R}-$ideal contained in $R$, so $I_{\vec{\alpha}} \subset \mathfrak{c}_R=I_{\vec{c}(R)}$ and $\vec{\alpha} \geq \vec{c}(R)$.\\

    To prove the claim, for every $i$ the tuple $\vec{\beta}+\sum_{j \neq i}e_j$ lies in $\vec{\alpha}+\mathbb{N}^{r(R)} \subset \nu(R)$ (where $e_j$ is the $j-$th unit vector), so there is some $w^{(i)} \in R$ with $w^{(i)}_i=u_i\pi_i^{\beta_i}$ for a unit $u_i \in V_i$ and $w^{(i)}_j \in \pi_j^{\beta_j+1}V_j$ for $j \neq i$. Since $R$ is residually rational, there is some $b_i \in R$ whose image in $\kappa_i=\kappa$ equals the residue of $\frac{z_i}{\pi_i^{\beta_i}u_i} \in V_i$. Let $y=\sum_i b_iw^{(i)} \in R \cap I_{\vec{\beta}}$. Then for every $i$ we have $z_i-y_i=(z_i-b_iu_i\pi_i^{\beta_i})-\sum_{j\neq i}b_jw^{(j)}_i \in \pi_i^{\beta_i+1}V_i$, as desired. 
\end{proof}

The semigroup $\nu(R)$ also determines the lengths of the ideals of $R$ defined by valuations (compare~\cite{de1987semigroup}):

\begin{lemma}\label{lem:lengths}
    For $\vec{\alpha} \in \mathbb{N}^{r(R)}$ let $R(\vec{\alpha})=\{g \in R \colon \ord(g_i,V_i) \geq \alpha_i \text{ for every } i\}$ (where $\ord(0,V_i)=\infty$), an ideal of $R$, and let $e_i$ be the $i-$th unit vector. Then $\length_R\left(\frac{R(\vec{\alpha})}{R(\vec{\alpha}+e_i)}\right) \leq 1$, with equality if and only if there is some $\vec{\beta} \in \nu(R)$ with $\beta_i=\alpha_i$ and $\beta_j \geq \alpha_j$ for every $j$. Consequently, $h_R(\vec{\alpha})=\length_R\left(\frac{R}{R(\vec{\alpha})}\right)$ is determined by $\nu(R)$ (by summing these lengths along any chain from $0$ to $\vec{\alpha}$ in which each step adds a unit vector), and so are $\delta(R)=\sum_i c_i - h_R(\vec{c}(R))$ and, by Remark~\ref{rem:conductor}, whether $R$ is Gorenstein.
\end{lemma}

\begin{proof}
    Sending $g \in R(\vec{\alpha})$ to the residue of $\frac{g_i}{\pi_i^{\alpha_i}}$ in $\kappa_i=\kappa$ gives an $R-$linear map $R(\vec{\alpha}) \to \kappa$ whose kernel is $R(\vec{\alpha}+e_i)$, which proves the first claim. The quotient is nonzero if and only if there is some $g \in R(\vec{\alpha})$ with $\ord(g_i,V_i)=\alpha_i$. If $\vec{\beta}$ is as in the statement, a non zero divisor $g$ with $\nu(g)=\vec{\beta}$ is such an element. Conversely, given such a $g$, let $z \in \overline{R}$ have $z_j=\pi_j^{N}$ if $g_j=0$ and $z_j=0$ otherwise, where $N \geq \max_j(c_j,\alpha_j)$. Then $z \in \mathfrak{c}_R \subset R$, $z_i=0$, and $g+z$ is a non zero divisor with $\vec{\beta}=\nu(g+z)$ as in the statement. Finally, $\mathfrak{c}_R=R(\vec{c}(R))$ and $\length_R\left(\frac{\overline{R}}{\mathfrak{c}_R}\right)=\sum_i c_i$ (as $R$ is residually rational), so $\delta(R)=\length_R\left(\frac{\overline{R}}{\mathfrak{c}_R}\right)-\length_R\left(\frac{R}{\mathfrak{c}_R}\right)=\sum_ic_i-h_R(\vec{c}(R))$.
\end{proof}

We use Lemma~\ref{lem:cond_min} to show that $\Gamma(R)$ is a numerical semigroup:

\begin{proposition}\label{prop:semigrp}
    $\gcd(\Gamma(R))=1$.
\end{proposition}

\begin{proof}
    By Lemma~\ref{lem:cond_min} we have $\vec{c}(R)+\mathbb{N}^{r(R)} \subset \nu(R)$, and so by the third item of Proposition~\ref{prop:nu_beasic_prop} every integer $N \geq c_1+\cdots+c_{r(R)}$ lies in $\Gamma(R)$. Thus $\Gamma(R)$ contains two consecutive integers, and $\gcd(\Gamma(R))=1$.
\end{proof}

As mentioned in Remark~\ref{rem:semigroup}, since $\Gamma(R)$ is a numerical semigroup, it has a conductor $c$. In  order to compute it, we first define the relative Milnor numbers of $R$, inspired by Section 3 of~\cite{ploski1995milnor}. 

\begin{definition}\label{def:RelMIlnor}
    If $\mathfrak{p}_1, \dots, \mathfrak{p}_{r(R)}$ are the minimal primes of $R$ and $\mathfrak{q}_i=\bigcap_{j\neq i}\mathfrak{p}_j$, we define the $i-$th relative Milnor number of $R$ to be $\mu_i(R)=\mu\left(\frac{R}{\mathfrak{p}_i}\right) + i(\mathfrak{p}_i, \mathfrak{q}_i)$. In addition, we define the vector of relative Milnor numbers of $R$ to be $\vec{\mu}(R)=(\mu_1(R), \dots, \mu_{r(R)}(R))$. 
\end{definition}

\begin{remark}
    \textup{Recall that $\mu\left(\frac{R}{\mathfrak{p}_i}\right)=2\delta\left(\frac{R}{\mathfrak{p}_i}\right)$. If $\edim(R)\leq 2$ then, as in the proof of the second item of Proposition~\ref{thm:branch}, $i(\mathfrak{p}_i, \mathfrak{q}_i)=\sum_{j\neq i} i(\mathfrak{p}_i, \mathfrak{p}_j)$, which recovers the classical definition. In general the two differ: for the complete intersection of Example~\ref{ex:hironaka_fails} we have $i(\mathfrak{p}_i,\mathfrak{q}_i)=2$, while $\sum_{j\neq i}i(\mathfrak{p}_i,\mathfrak{p}_j)=3$.}
\end{remark}

\begin{proposition}\label{prop:mu_vec}
    Assume that $R$ is Gorenstein. Then $\vec{c}(R)=\vec{\mu}(R)$, that is, the conductor $\mathfrak{c}_R$ is generated, as an ideal of $\overline{R}$, by $(\pi_1^{\mu_1(R)}, \dots, \pi_{r(R)}^{\mu_{r(R)}(R)})$. Consequently, $\vec{\mu}(R)$ is the minimal element in $\nu(R) \subset \mathbb{N}^{r(R)}$ such that $\vec{\mu}(R) + \mathbb{N}^{r(R)} \subset \nu(R)$. 
\end{proposition}

\begin{proof}
    The second statement follows from the first and Lemma~\ref{lem:cond_min}. Fix $i$, denote $\delta_i=\delta\left(\frac{R}{\mathfrak{p}_i}\right)$, and view $V_i \subset \overline{R}$ as the $i-$th factor. An element of $R$ lies in $V_i$ if and only if it lies in $\mathfrak{p}_j$ for every $j \neq i$, so $R \cap V_i = \mathfrak{q}_i$; in particular $\mathfrak{q}_i \cap \mathfrak{p}_i=0$, and so the projection identifies $\mathfrak{q}_i$ with the ideal $\frac{\mathfrak{q}_i+\mathfrak{p}_i}{\mathfrak{p}_i}$ of $\frac{R}{\mathfrak{p}_i}$. Since $\pi_i^{c_i}V_i \subset \mathfrak{c}_R \subset R$, we get a chain of $R-$modules
    \begin{equation*}
        \pi_i^{c_i}V_i \subset \mathfrak{q}_i \subset \frac{R}{\mathfrak{p}_i} \subset V_i,
    \end{equation*}
    with $\length_R\left(\frac{V_i}{\pi_i^{c_i}V_i}\right)=c_i$ (as $R$ is residually rational), $\length_R\left(\frac{R/\mathfrak{p}_i}{\mathfrak{q}_i}\right)=\length_R\left(\frac{R}{\mathfrak{p}_i+\mathfrak{q}_i}\right)$, which is $i(\mathfrak{p}_i, \mathfrak{q}_i)$, and $\length_R\left(\frac{V_i}{R/\mathfrak{p}_i}\right)=\delta_i$. We claim that $\length_R\left(\frac{\mathfrak{q}_i}{\pi_i^{c_i}V_i}\right)=\delta_i$, which gives $c_i=2\delta_i+i(\mathfrak{p}_i,\mathfrak{q}_i)=\mu_i(R)$.\\

    Denote $M^\vee=\Hom_R(M,R)$. First, since $R$ is Gorenstein of dimension one, for finitely generated maximal Cohen-Macaulay $R-$modules $M \subset N$ with $\frac{N}{M}$ of finite length we have $\length_R\left(\frac{M^\vee}{N^\vee}\right)=\length_R\left(\frac{N}{M}\right)$. Indeed, applying $\Hom_R(-,R)$ to $0 \to M \to N \to \frac{N}{M} \to 0$ gives an exact sequence $0 \to N^\vee \to M^\vee \to \Ext^1_R\left(\frac{N}{M},R\right) \to 0$, since $\Hom_R\left(\frac{N}{M},R\right)=0$ (as $\operatorname{depth}(R)=1$) and $\Ext^1_R(N,R)=0$ (as $N$ is maximal Cohen-Macaulay and $R$ is Gorenstein). Moreover, $\length_R(\Ext^1_R(L,R))=\length_R(L)$ for every $R-$module $L$ of finite length: this holds for $L=\kappa$ since $R$ is Gorenstein of dimension one, and $\Ext^1_R(-,R)$ is exact on modules of finite length, since $\Hom_R(L,R)=0$ and $\Ext^2_R(L,R)=0$ for such $L$.\\

    Now, $\frac{R}{\mathfrak{p}_i}$ and $V_i$ are maximal Cohen-Macaulay $R-$modules (they are finitely generated and every non zero divisor of $R$ acts injectively on them), and $\frac{V_i}{R/\mathfrak{p}_i}$ has length $\delta_i$. We have $\left(\frac{R}{\mathfrak{p}_i}\right)^\vee \cong \Ann_R(\mathfrak{p}_i)=\mathfrak{q}_i$ via $\varphi \mapsto \varphi(1)$ (as $R$ is reduced). If $\varphi \in V_i^\vee$ then $w=\varphi(1) \in R$ is killed by $\mathfrak{p}_i$, so $w \in \mathfrak{q}_i \subset V_i$; and for $v \in V_i$, choosing $s \in \NZD(R) \cap \mathfrak{c}_R$ (for example $s=(\pi_1^{c_1},\dots,\pi_{r(R)}^{c_{r(R)}})$), the element $a=sv$ lies in $R$, and so $s\varphi(v)=\varphi(a\cdot 1)=aw=s(vw)$; since $s$ is a non zero divisor on $\overline{R}$, we get $\varphi(v)=vw$. Conversely, every $w \in R \cap V_i$ with $wV_i \subset R$ defines such a $\varphi$. Thus $V_i^\vee \cong \{w \in R \cap V_i \colon wV_i \subset R\}$, which is the largest $V_i-$submodule of $V_i$ contained in $R$, namely $\pi_i^{c_i}V_i$ (Notation~\ref{not:cond_vec}). Under these identifications, the restriction map $V_i^\vee \to \left(\frac{R}{\mathfrak{p}_i}\right)^\vee$ is the inclusion $\pi_i^{c_i}V_i \subset \mathfrak{q}_i$, and so $\length_R\left(\frac{\mathfrak{q}_i}{\pi_i^{c_i}V_i}\right)=\length_R\left(\frac{V_i}{R/\mathfrak{p}_i}\right)=\delta_i$, proving the claim.
\end{proof}

\begin{lemma}\label{lem:vec_mu_sum}
    If $R$ is Gorenstein or $\edim(R)\leq 2$ then $\mu_1(R)+\cdots+\mu_{r(R)}(R)=\mu(R)+r(R)-1=2\delta(R)$. 
\end{lemma}

\begin{proof}
    If $R$ is Gorenstein then by Proposition~\ref{prop:mu_vec} and the first item of Remark~\ref{rem:conductor} we have $\sum_i \mu_i(R)=\sum_i c_i=\length_R\left(\frac{\overline{R}}{\mathfrak{c}_R}\right)=2\delta(R)$. If $\edim(R)\leq 2$ then $\sum_i i(\mathfrak{p}_i,\mathfrak{q}_i)=2\sum_{i<j} i(\mathfrak{p}_i, \mathfrak{p}_j)$, and the claim follows from the second item of Proposition~\ref{thm:branch}. 
\end{proof}

\begin{remark}
    \textup{Lemma~\ref{lem:vec_mu_sum} fails in general: for the three coordinate axes of Example~\ref{ex:hironaka_fails} we have $\mu_i(R)=i(\mathfrak{p}_i,\mathfrak{q}_i)=1$, so $\sum_i\mu_i(R)=3$, while $2\delta(R)=4$; here $\vec{c}(R)=(1,1,1)=\vec{\mu}(R)$ even though $R$ is not Gorenstein.}
\end{remark}

\begin{corollary}\label{cor:Gamma_cond}
    Assume $R$ is Gorenstein. Then the conductor of $\Gamma(R)$ is at most $2\delta(R) = \mu(R)+r(R)-1$, with equality if $r(R)=1$. 
\end{corollary}

\begin{proof}
    By Proposition~\ref{prop:mu_vec} and the third item of Proposition~\ref{prop:nu_beasic_prop}, every integer $N \geq \sum_i \mu_i(R)$ lies in $\Gamma(R)$, and $\sum_i \mu_i(R)=2\delta(R)$ by Lemma~\ref{lem:vec_mu_sum}. If $r(R)=1$ then $\nu(R)=\Gamma(R)$, and the minimality in Proposition~\ref{prop:mu_vec} gives equality. 
\end{proof}

\begin{remark}\label{rem:mu-r(R)-gamma}
\begin{enumerate}
    \item \textup{From Corollary~\ref{cor:Gamma_cond} we can in fact conclude that if $r(R)=1$ and $R$ is Gorenstein then $\mu(R)-1=2 \delta(R)-1 \notin \Gamma(R)$, as otherwise $2\delta(R)$ would not be the conductor. In addition, if $r(R)=1$, from Proposition 1.5 in~\cite{barroso2012approach} we can conclude that the number of integers in $\mathbb{N} \setminus \Gamma(R)$ is exactly $\delta(R)$.}
    \item \textup{Note that the conductor of $\Gamma(R)$ need not be $2\delta(R)$ if $r(R)>1$, as it is possible that there exists some tuple $(N_1, \dots, N_{r(R)}) \in \nu(R)$ with $N_1 + \cdots + N_{r(R)} = 2\delta(R)-1$. For a concrete example, see the semigroup of $E_7$ in Remark~\ref{rem:ADE_Gamma}.}
    \item \textup{From Theorem 4.8 in~\cite{campillo1994gorenstein} (which in turn is based upon~\cite{de1988gorenstein}) we can conclude another  property of $\vec{\mu}(R)$: Given some $ (\alpha_1, \dots, \alpha_r) \in \mathbb{Z}^r$, we set $\Delta_i(\alpha_1, \dots, \alpha_r)$ to be the set of all tuples $(\beta_1, \dots, \beta_r) \in \mathbb{N}^r$ such that $\beta_i = \alpha_i$ but $\beta_j >\alpha_j$ for every $j \neq i$. Then $R$ is Gorenstein if and only if $\nu(R)$ satisfies that $(\alpha_1, \dots, \alpha_{r(R)}) \in \nu(R)$ if and only if we have that  $\Delta_i\left(c_1-1-\alpha_1, \dots, c_{r(R)}-1-\alpha_{r(R)} \right) \cap \nu(R) = \emptyset$ for every $i$, where $\vec{c}(R)=(c_1,\dots,c_{r(R)})$. If $R$ is Gorenstein then $\vec{c}(R)=\vec{\mu}(R)$ by Proposition~\ref{prop:mu_vec}. This symmetry property is not used from now on. }
\end{enumerate}
\end{remark}

We can use the semigroup $\Gamma(R)$ to give a version of Theorem~\ref{thm:Morse}. 

\begin{corollary}\label{cor:A_1_Gamma}
    Assume $R$ is not a domain. Then $\Gamma(R) = \{0\} \cup (2 + \mathbb{N})$ if and only if $R$ is a double point.
\end{corollary}

\begin{proof}
    First, if $R$ is a double point then $\hat{R} = \frac{S}{\langle xy \rangle}$ where $(S, \mathfrak{n})$ is some complete regular $2-$dimensional ring with $\mathfrak{n}=\langle x,y \rangle$. Let $g \in \NZD(R)$ be a non unit, then there exists some $m,n \geq 1$ and some units $u,v \in S$ such that $g=ux^n+vy^m + \langle xy \rangle$. Therefore we have that the $R-$length of $\frac{S}{\langle xy, ux^n+vy^m \rangle}$ is $n+m$, and so by Remark~\ref{rem:Gam_comp} we can conclude that $\Gamma(R) = \{0, 2, 3, \dots \}$. \\

    Second, assume that $\Gamma(R) = \{0, 2,3, \dots\}$. Then $\mult(R)=2$ by Lemma~\ref{lem:semigroup}. Since $R$ is not a domain, $r(R) \geq 2$, and so by Corollary~\ref{cor:mult_r} and Lemma~\ref{lem:mult_sum_prim} we have $r(R)=2$ and $\mult\left(\frac{R}{\mathfrak{p}_1}\right)=\mult\left(\frac{R}{\mathfrak{p}_2}\right)=1$, i.e. $\frac{R}{\mathfrak{p}_1}=V_1$ and $\frac{R}{\mathfrak{p}_2}=V_2$ are DVRs. Let $A=\frac{R}{\mathfrak{p}_1+\mathfrak{p}_2}$, an Artinian local ring of length $i=i(\mathfrak{p}_1, \mathfrak{p}_2)$, which is a quotient of both $V_1$ and $V_2$, so that the images of $g_1$ and $g_2$ in $A$ coincide for every $g \in R$. Since $3 \in \Gamma(R)$, there is some $g \in \NZD(R)$ with $\nu(g)\in\{(1,2),(2,1)\}$, say $\nu(g)=(1,2)$. If $i \geq 2$, then the image of $g_1 \in \pi_1V_1 \setminus \pi_1^2V_1$ in $A$ lies in $\mathfrak{m}_A \setminus \mathfrak{m}_A^2$ (as the kernel $\pi_1^iV_1$ of $V_1 \to A$ is contained in $\pi_1^2V_1$), while the image of $g_2 \in \pi_2^2V_2$ lies in $\mathfrak{m}_A^2$, a contradiction. Thus $i(\mathfrak{p}_1,\mathfrak{p}_2)=1$, and so by Proposition~\ref{thm:branch} we get $\delta(R)=1$ and $\mu(R)=1$, and the result follows from Theorem~\ref{thm:Morse}.
\end{proof}

\begin{remark}\label{rem:A_12groups}
    \textup{If we remove the assumption in Corollary~\ref{cor:A_1_Gamma} that $R$ is not a domain, then we can conclude that $\Gamma(R)= \{0\} \cup (2 + \mathbb{N})$ if and only if $\delta(R)=1$, which by Proposition~\ref{cor:delta1} is equivalent (if we assume in addition that $\chara(\kappa) \neq 2$) to $R$ being either an $A_1$ or an $A_2$ singularity. This shows us that the semigroup $\Gamma(R)$ does not detect the number of minimal primes that $R$ has. }
\end{remark}

We end this section by showing that we can also use the semigroups to prove a version of Noether's theorem, as an analogue of Theorem 3.1 and Theorem 3.2 in~\cite{ploski1995milnor}:

\begin{proposition}\label{prop:noether}
    Let $g,h \in \NZD(R)$ and let $\mathfrak{p}_1, \dots, \mathfrak{p}_{r(R)}$ be the minimal primes of $R$. Then:
    \begin{enumerate}
        \item If $\nu(h) \geq \nu(g)+\vec{c}(R)$ then $h \in \langle g \rangle$. In particular, if $R$ is Gorenstein and $l(h, \frac{R}{\mathfrak{p}_i}) \geq l(g, \frac{R}{\mathfrak{p}_i}) + \mu_i(R)$ for every $i$, then $h \in \langle g \rangle$.
        \item If $r(R)=1$, $R$ is Gorenstein, and $l(h,R)=l(g,R) + \mu(R) - 1$, then $h \notin \langle g \rangle$. 
    \end{enumerate}
\end{proposition}

\begin{proof}
    For the first item, $\frac{h}{g}$ is an element of the total ring of fractions of $R$ whose $i-$th coordinate has order at least $c_i \geq 0$ for every $i$; thus $\frac{h}{g} \in \overline{R}$, and in fact $\frac{h}{g} \in (\pi_1^{c_1},\dots, \pi_{r(R)}^{c_{r(R)}})\overline{R}=\mathfrak{c}_R \subset R$. Hence $h \in \langle g \rangle$. The second statement of the first item follows from Proposition~\ref{prop:mu_vec}.\\

    For the second item, assume towards contradiction that $h=ga$ for some $a \in R$. Then $a \in \NZD(R)$, and from Lemma~\ref{lem:semigroup} we can conclude that $l(g,R)+l(a,R)=l(h,R)=l(g,R) + \mu(R) - 1$. This gives us that $l(a,R)=\mu(R) - 1=2\delta(R)-1$, which is impossible by Remark~\ref{rem:mu-r(R)-gamma}. 
\end{proof}

\begin{remark}
    \textup{The second item of Proposition~\ref{prop:noether} fails for $r(R)\geq 2$. For example, for an $E_7$ singularity (with $\chara(\kappa)=0$) we have that $\mu(R)=7$ (Remark~\ref{rem:ADE_Gamma}) and $l(y^2,R)=6=\mu(R)-1$. So, by  taking $g=1$ and $h=y^2$ we have that  $h\in\langle g\rangle$.}
\end{remark}

\section{One Dimensional Rings of Finite Cohen-Macaulay Type}\label{sec:FCMT1}

In this section and the next we show how we can use the tools and techniques we developed to relate one dimensional rings of finite Cohen-Macaulay type to the classical ADE curve singularities using their semigroups. In this section we recall the Drozd--Ro\u{\i}ter criterion and classify the residually rational domains of finite Cohen-Macaulay type by their semigroup $\Gamma(R)$; in Section~\ref{sec:gorenstein} we treat Gorenstein rings with several minimal primes. For more information on finite and countable Cohen-Macaulay type, see~\cite{svoray2025ade, yoshino1990maximal, leuschke2012cohen}. 

\begin{definition}\label{def:cm_type}
     Denote by $\mathcal{MCM}(R)$ the set of isomorphism classes $[M]$ of indecomposable maximal Cohen-Macaulay modules over $R$. If $|\mathcal{MCM}(R)|$ is finite we say that $R$ has finite Cohen-Macaulay type.  
\end{definition}

The following proposition summarizes the main results of Chapter $4$ of~\cite{leuschke2012cohen} (using our notations); the second item is the Drozd--Ro\u{\i}ter criterion~\cite{drozd1967commutative, green1978integral, ccimen1998one}. 

\begin{proposition}\label{prop:CMtype_onedim}
Assume that $R$ is a one dimensional Cohen-Macaulay local ring. Then: 
    \begin{enumerate}
        \item If $R$ is analytically unramified and does not have finite Cohen-Macaulay type, then $|\mathcal{MCM}(R)|\geq |\kappa|$.  
        \item If $R$ is analytically unramified, then $R$ has finite Cohen-Macaulay type if and only if $\mult(R) \leq 3$ and $\frac{\mathfrak{m}\overline{R}+R}{R}$ is a cyclic $R-$module. 
        \item If $R$ has finite Cohen-Macaulay type then $R$ is analytically unramified (and therefore reduced). 
        \item If $\mult(R)=2$ and $R$ is analytically unramified then $R$ has finite Cohen-Macaulay type. 
        \item $R$ has finite Cohen-Macaulay type if and only if the completion of $R$ does. 
    \end{enumerate}
\end{proposition}

\begin{remark}
\textup{From the first item of Proposition~\ref{prop:CMtype_onedim} we have that if $R$ is an analytically unramified one dimensional ring then either $R$ has finite Cohen-Macaulay type or $|\mathcal{MCM}(R)|\geq|\kappa|$. This in fact tells us that in the analytically unramified one dimensional case, having a sparse Cohen-Macaulay type (as in Definition 2.3 in~\cite{svoray2025ade}) is the same as having finite Cohen-Macaulay type. }
\end{remark}

The second condition in the second item of Proposition~\ref{prop:CMtype_onedim} can be read off the leading terms of the elements of $\mathfrak{m}$:

\begin{lemma}\label{lem:DR2_general}
    Let $(R,\mathfrak{m},\kappa)$ be a one dimensional reduced local ring whose normalization $\overline{R}$ is finite over $R$, and let $Q=\frac{\mathfrak{m}\overline{R}+R}{R}$. Then $\frac{\mathfrak{m}\overline{R}}{\mathfrak{m}^2\overline{R}}$ is a $\kappa-$vector space of dimension $\mult(R)$, and $\frac{Q}{\mathfrak{m}Q}\cong \frac{\mathfrak{m}\overline{R}}{\mathfrak{m}^2\overline{R}+\mathfrak{m}}$. Consequently, if $W$ denotes the image of $\mathfrak{m}$ in $\frac{\mathfrak{m}\overline{R}}{\mathfrak{m}^2\overline{R}}$, then $Q$ is a cyclic $R-$module if and only if $\dim_\kappa(W)\geq \mult(R)-1$.
\end{lemma}

\begin{proof}
    The ring $\overline{R}$ is a finite product of principal ideal domains, and $\mathfrak{m}$ contains a non zero divisor of $\overline{R}$, so $\mathfrak{m}\overline{R}=t\overline{R}$ for a non zero divisor $t$ of $\overline{R}$. Hence $\frac{\mathfrak{m}^n\overline{R}}{\mathfrak{m}^{n+1}\overline{R}} \cong \frac{\overline{R}}{\mathfrak{m}\overline{R}}$ for every $n$, so $\length_R\left(\frac{\overline{R}}{\mathfrak{m}^n\overline{R}}\right)=n\length_R\left(\frac{\overline{R}}{\mathfrak{m}\overline{R}}\right)$, and since $\frac{\overline{R}}{R}$ has finite length, $\mult(R)=\mult_R(\overline{R})=\length_R\left(\frac{\overline{R}}{\mathfrak{m}\overline{R}}\right)=\dim_\kappa\left(\frac{\mathfrak{m}\overline{R}}{\mathfrak{m}^2\overline{R}}\right)$ (see~\cite[Tag 0AZW]{stacks-project}). As in the proof of Lemma~\ref{lem:milnor_nonneg}, $\mathfrak{m} \subset \mathfrak{m}\overline{R} \cap R \subset J \cap R=\mathfrak{m}$, where $J$ is the Jacobson radical of $\overline{R}$. Since $\mathfrak{m}R \subset R$ we have $\mathfrak{m}Q=\frac{\mathfrak{m}^2\overline{R}+R}{R}$, and so
    \begin{equation*}
        \frac{Q}{\mathfrak{m}Q}\cong \frac{\mathfrak{m}\overline{R}+R}{\mathfrak{m}^2\overline{R}+R} \cong \frac{\mathfrak{m}\overline{R}}{\mathfrak{m}^2\overline{R}+(\mathfrak{m}\overline{R} \cap R)}=\frac{\mathfrak{m}\overline{R}}{\mathfrak{m}^2\overline{R}+\mathfrak{m}}.
    \end{equation*}
    Thus $\dim_\kappa\left(\frac{Q}{\mathfrak{m}Q}\right)=\mult(R)-\dim_\kappa(W)$, and the last statement follows from Nakayama's lemma.
\end{proof}

\begin{proposition}
Assume the standing hypotheses and residual rationality.
Let $\vec m$ be the minimal element of $\nu(R)\setminus\{0\}$,
and write
$h_R(\alpha)=\length_R(R/R(\alpha))$.
Then, for $Q=(\mathfrak m\overline R+R)/R$, we have that
\[
 \dim_\kappa(Q/\mathfrak mQ)=M+1-h_R(2\vec m).
\]
In particular, $R$ has finite Cohen--Macaulay type if and only if
$\mult(R) \leq3$ and $h_R(2\vec m)\geq \mult(R)$.
\end{proposition}
\begin{proof}
In the notation of Lemma~\ref{lem:DR2_general}, the kernel of
$\mathfrak m\to W$ is
$\mathfrak m\cap\mathfrak m^2\overline R=R(2\vec m)$.
Thus $\dim_\kappa W=h_R(2\vec m)-1$, and the identity follows
from that lemma. The criterion follows from
Proposition~\ref{prop:CMtype_onedim} and Nakayama's lemma.
\end{proof}

We now turn to the classification of one-dimensional domains of finite Cohen-Macaulay type based upon the semigroup $\Gamma(R)$ alone, without any assumption on the characteristic of the residue field and without assuming that $R$ is Gorenstein. As we see, this case is much more direct and does not require explicit computations of elements in $R$. From now until the end of this section, we assume that $(R, \mathfrak{m}, \kappa)$ is a one dimensional complete local domain whose normalization $V=\overline{R}$ is a DVR, finite over $R$, with uniformizer $\pi$ and valuation $v=\ord(-,V)$, and (except in Proposition~\ref{prop:ed_bound}) that $R$ is residually rational, i.e. the map $\kappa \to \frac{V}{\pi V}$ is an isomorphism (Definition~\ref{def:res_rat} with $r(R)=1$). We do not assume that $\kappa$ is separably closed. The proof of the third item of Lemma~\ref{lem:semigroup} only uses that $V$ is finite over $R$ and that $R$ is residually rational, so $\Gamma(R)=v(R\setminus\{0\})$. Write $\mathfrak{m}V=\pi^eV$. By Lemma~\ref{lem:DR2_general} and residual rationality,
\begin{equation*}
    \mult(R)=\dim_\kappa\left(\frac{V}{\mathfrak{m}V}\right)=e=\min(\Gamma(R)\setminus\{0\}).
\end{equation*}
In addition, filtering $\frac{V}{R}$ by values shows that $\delta(R)$ is the number of gaps of $\Gamma(R)$, i.e. of elements of $\mathbb{N}\setminus \Gamma(R)$. 

\begin{lemma}\label{lem:DR2_gaps}
    Let $Q=\frac{\mathfrak{m}V+R}{R}$. Then $\dim_\kappa\left(\frac{Q}{\mathfrak{m}Q}\right)$ equals the number of gaps of $\Gamma(R)$ in the interval $[e, 2e)$. In particular, $Q$ is a cyclic $R-$module if and only if $\Gamma(R)$ has at most one gap in $[e,2e)$.
\end{lemma}

\begin{proof}
    Here $\frac{\mathfrak{m}V}{\mathfrak{m}^2V}=\frac{\pi^eV}{\pi^{2e}V}$. Since $R$ is residually rational, the image $W$ of $\mathfrak{m}$ in it has dimension equal to the number of values in $[e,2e)$ attained by elements of $\mathfrak{m}$ (filter $W$ by values; each graded piece is at most one dimensional), i.e. $\dim_\kappa(W)=|\Gamma(R) \cap [e,2e)|$. By the proof of Lemma~\ref{lem:DR2_general}, $\dim_\kappa\left(\frac{Q}{\mathfrak{m}Q}\right)=e-|\Gamma(R) \cap [e,2e)|$, which is the number of gaps in $[e,2e)$. 
\end{proof}

\begin{remark}\label{rem:minimal_gaps}
    \textup{Consider the preorder on $\mathbb{N}$ given by $a \preceq b$ if $b-a \in \Gamma(R)$. The gaps of $\Gamma(R)$ in $[e,2e)$ are exactly the $\preceq-$minimal elements of the set of gaps that are at least $e$: if $n \in [e,2e)$ is a gap and $h \in \Gamma(R)\setminus\{0\}$, then $n-h<e$; and if $n\geq 2e$ is a gap then so is $n-e$, as otherwise $n=(n-e)+e \in \Gamma(R)$. Thus Lemma~\ref{lem:DR2_gaps} says that $Q$ is cyclic if and only if the gaps of $\Gamma(R)$ that are at least $e$ have at most one $\preceq-$minimal element. For example, if $e=3$ then $Q$ is cyclic if and only if $4 \in \Gamma(R)$ or $5\in\Gamma(R)$.}
\end{remark}

\begin{theorem}\label{thm:FCMT_DOm}
    Assume that $R$ is not regular. Then the following are equivalent:
    \begin{enumerate}
        \item $R$ has finite Cohen-Macaulay type,
        \item $\Gamma(R)$ is one of $\lceil 2,2n+1\rfloor$ (for some $n \geq 1$), $\lceil 3,4\rfloor$, $\lceil 3,5\rfloor$, $\lceil 3,4, 5\rfloor$, or $\lceil 3,5, 7\rfloor$,
        \item $\Gamma(R)$ contains the semigroup of an $A_{2n}$, $E_6$, or $E_8$ singularity, i.e. one of $\lceil 2,2n+1\rfloor$ (for some $n \geq 1$), $\lceil 3,4\rfloor$, or $\lceil 3,5\rfloor$.
    \end{enumerate}
    Moreover, for every field $\kappa$ and every semigroup $H$ in the second item, the complete semigroup ring $\kappa[[H]]=\kappa[[t^h \colon h \in H]] \subset \kappa[[t]]$ satisfies our assumptions and has $\Gamma(\kappa[[H]])=H$, so all of these semigroups occur.
\end{theorem}

\begin{proof}
Since $R$ is a complete domain, it is analytically unramified, so by the second item of Proposition~\ref{prop:CMtype_onedim}, $R$ has finite Cohen-Macaulay type if and only if $e=\mult(R) \leq 3$ and $Q=\frac{\mathfrak{m}V+R}{R}$ is cyclic, and by Lemma~\ref{lem:DR2_gaps} the latter means that $\Gamma(R)$ has at most one gap in $[e,2e)$. Since $R$ is not regular, $e \geq 2$.\\

Assume that $R$ has finite Cohen-Macaulay type. If $e=2$, then $\Gamma(R)$ contains every even integer, and since it contains all integers which are at least the conductor of $R$, it contains an odd integer; if $2n+1$ is the smallest one then $n \geq 1$ and every odd element of $\Gamma(R)$ is $2n+1$ plus an even integer, so $\Gamma(R)=\lceil 2,2n+1\rfloor$. If $e=3$, then by Lemma~\ref{lem:DR2_gaps} (see Remark~\ref{rem:minimal_gaps}) either $4$ or $5$ lies in $\Gamma(R)$. Let $a$ and $b$ be the minimal elements in $\Gamma(R)$ such that $a \equiv 1 \mod 3$ and $b \equiv 2 \mod 3$, respectively. Then $\Gamma(R)=3\mathbb{N} \cup (a + 3\mathbb{N}) \cup (b+ 3\mathbb{N})$, and so $\Gamma(R)= \lceil 3, a,b \rfloor$. If $a=4$ then $8 = 4+4 \in \Gamma(R)$, and since $8 \equiv 2 \mod3$, from the minimality of $b$ we have that either $b=5$ or $b=8$, in which cases we have that $\Gamma(R)$ is either $\lceil 3,4, 5\rfloor$ or $\lceil 3,4\rfloor$. If $a \neq 4$ then we must have that $b=5$ and so $10 = 5+5 \in \Gamma(R)$, and since $10 \equiv 1 \mod 3$, by the minimality of $a$ we have that $a = 7$ or $a=10$, and so $\Gamma(R)$ is either $\lceil 3,5, 7\rfloor$ or $\lceil 3,5\rfloor$. This proves that the first item implies the second, and the second item clearly implies the third.\\

Assume the third item holds. If $\Gamma(R)$ contains $\lceil 2,2n+1\rfloor$ then $2 \in \Gamma(R)$, so $e=2$, and the only possible gap in $[2,4)$ is $3$. If $\Gamma(R)$ contains $\lceil 3,4\rfloor$ or $\lceil 3,5\rfloor$ then $e \in \{2,3\}$; if $e=2$ we are in the previous case, and if $e=3$ then $4$ or $5$ lies in $\Gamma(R)$, so there is at most one gap in $[3,6)$. In all cases $R$ has finite Cohen-Macaulay type.\\

Finally, for a numerical semigroup $H$, the ring $\kappa[[H]]$ is a complete local domain (it is the image of a power series ring in finitely many variables, as $H$ is finitely generated), its normalization is $\kappa[[t]]$, which is finite over $\kappa[[H]]$ as $t^n \in \kappa[[H]]$ for $n \gg 0$, it is residually rational, and the values of its non zero elements are exactly the elements of $H$.
\end{proof}

\begin{remark}\label{rem:noncomplete_domains}
    \textup{Theorem~\ref{thm:FCMT_DOm} also applies to a one dimensional analytically irreducible residually rational local domain $R$ with finite normalization which is not complete: by the fifth item of Proposition~\ref{prop:CMtype_onedim}, $R$ has finite Cohen-Macaulay type if and only if $\hat{R}$ does, and $\Gamma(R)=\Gamma(\hat{R})$, since the normalization of $\hat{R}$ is the completion $\hat{V}$ of $V$ (with the same residue field) and an element $\hat{r} \in \hat{R}$ of value $n$ can be approximated by some $r \in R$ with $\hat{r}-r \in \mathfrak{m}^{n+1}\hat{R}\subset \pi^{n+1}\hat{V}$, so that $v(r)=n$ (compare Remark~\ref{rem:Gam_comp}). In particular, it applies to the domains in the class considered in Sections~\ref{sec:milnor} and~\ref{sec:semigroup}, which are analytically irreducible by Remark~\ref{rem:conductor}.}
\end{remark}

\begin{corollary}\label{cor:FCMT_Gor_domain}
    Assume that $R$ is not regular, Gorenstein, and has finite Cohen-Macaulay type. Then $\Gamma(R)$ is one of $\lceil 2,2n+1\rfloor$ ($n\geq 1$), $\lceil 3,4\rfloor$, or $\lceil 3,5\rfloor$, i.e. it is the semigroup of an $A_{2n}$, $E_6$, or $E_8$ singularity. Conversely, in every characteristic these semigroups are realized by the Gorenstein rings
    \begin{gather*}
        \kappa[[t^2,t^{2n+1}]] \cong \frac{\kappa[[X,Y]]}{\langle Y^2-X^{2n+1}\rangle}, \\
        \kappa[[t^3,t^4]] \cong \frac{\kappa[[X,Y]]}{\langle Y^3-X^4\rangle}, \qquad \kappa[[t^3,t^5]] \cong \frac{\kappa[[X,Y]]}{\langle Y^3-X^5\rangle}
    \end{gather*}
    (given by $X\mapsto t^2, Y \mapsto t^{2n+1}$, respectively $X\mapsto t^3$ and $Y\mapsto t^4$ or $t^5$), which have finite Cohen-Macaulay type. The rings $\kappa[[t^3,t^4,t^5]]$ and $\kappa[[t^3,t^5,t^7]]$ have finite Cohen-Macaulay type and are not Gorenstein.
\end{corollary}

\begin{proof}
    By Kunz's theorem~\cite{kunz1970value} (equivalently, by the first item of Remark~\ref{rem:conductor}, as $\vec{c}(R)$ is the conductor of $\Gamma(R)$ by Lemma~\ref{lem:cond_min} and $\delta(R)$ is the number of gaps), $R$ is Gorenstein if and only if the conductor $c$ of $\Gamma(R)$ equals twice the number $g$ of gaps of $\Gamma(R)$, i.e. $\Gamma(R)$ is symmetric. For the semigroups of Theorem~\ref{thm:FCMT_DOm} we have:
    \begin{center}
    \begin{tabular}{|l|l|l|l|}
    \hline
    $\Gamma(R)$ & gaps & $c$ & $2g$ \\ \hline \hline
    $\lceil 2,2n+1\rfloor$ & $1,3,\dots,2n-1$ & $2n$ & $2n$ \\ \hline
    $\lceil 3,4\rfloor$ & $1,2,5$ & $6$ & $6$ \\ \hline
    $\lceil 3,5\rfloor$ & $1,2,4,7$ & $8$ & $8$ \\ \hline
    $\lceil 3,4,5\rfloor$ & $1,2$ & $3$ & $4$ \\ \hline
    $\lceil 3,5,7\rfloor$ & $1,2,4$ & $5$ & $6$ \\ \hline
    \end{tabular}
    \end{center}
    This proves the first statement and that the last two rings are not Gorenstein; they have finite Cohen-Macaulay type by Theorem~\ref{thm:FCMT_DOm}. For the presentations, the kernel of the map $\kappa[[X,Y]] \to \kappa[[t]]$ is a height one prime of the UFD $\kappa[[X,Y]]$, hence principal, and it contains the given polynomial $F$. As a polynomial in $Y$ over $\kappa[[X]]$, $F$ has degree $2$ or $3$ and no root in $\kappa[[X]]$ (by comparing orders in $X$, as $2n+1$ is odd and $4$ and $5$ are not divisible by $3$), so $F$ is irreducible, and it generates the kernel. Thus these rings are hypersurfaces, hence Gorenstein, and they have finite Cohen-Macaulay type by Theorem~\ref{thm:FCMT_DOm}. No extraction of roots is needed, so this holds in every characteristic.
\end{proof}

We end this section by discussing what changes when $R$ is not residually rational, in which case the semigroup of values loses most of its information (compare Example~\ref{ex:insep}).

\begin{proposition}\label{prop:ed_bound}
    Let $(R, \mathfrak{m}, \kappa)$ be a one dimensional complete local domain whose normalization $V$ is a DVR finite over $R$, with uniformizer $\pi$ and residue field $K$. Write $\mathfrak{m}V=\pi^eV$ and $d=[K:\kappa]$. Then $\mult(R)=\dim_\kappa\left(\frac{V}{\mathfrak{m}V}\right)=ed$. Consequently, if $R$ has finite Cohen-Macaulay type then $ed \leq 3$, and so either $R$ is residually rational (and Theorem~\ref{thm:FCMT_DOm} applies), or $e=1$ and $d \in \{2,3\}$; in the latter case $v(R\setminus\{0\})=\mathbb{N}$ and $\Gamma(R)=d\mathbb{N}$. If moreover $\kappa$ is separably closed (for example, if $R$ is strictly Henselian), then $K/\kappa$ is purely inseparable, and so in the latter case either $\chara(\kappa)=d=2$ or $\chara(\kappa)=d=3$.
\end{proposition}

\begin{proof}
    Since $d$ is finite (as $V$ is finite over $R$), we have $\dim_\kappa\left(\frac{V}{\mathfrak{m}V}\right)=\dim_\kappa\left(\frac{V}{\pi^{e}V}\right)=ed$, which is $\mult(R)$ by Lemma~\ref{lem:DR2_general}. By the second item of Proposition~\ref{prop:CMtype_onedim}, finite Cohen-Macaulay type implies $ed \leq 3$, so if $d>1$ then $e=1$ and $d\in\{2,3\}$. If $e=1$ then some element of $R$ has value $1$, so $v(R\setminus\{0\})=\mathbb{N}$, and as in the proof of the third item of Lemma~\ref{lem:semigroup} we have $l(g,R)=\length_R\left(\frac{V}{gV}\right)=d\cdot v(g)$ for every $0 \neq g \in R$, so $\Gamma(R)=d\mathbb{N}$. Finally, if $\kappa$ is separably closed then $K/\kappa$ is purely inseparable, so $d$ is a power of $\chara(\kappa)$.
\end{proof}

\begin{example}\label{ex:insep_FCMT}
    \textup{In Example~\ref{ex:insep}, the element $x$ has value $1$ and $d=p$, so $\mult(R)=p$; thus $R$ does not have finite Cohen-Macaulay type if $p\geq 5$. If $p=2$ then $R$ has multiplicity $2$, so it has finite Cohen-Macaulay type by the fourth item of Proposition~\ref{prop:CMtype_onedim}; it is a Gorenstein ring (a hypersurface) of finite Cohen-Macaulay type with $\Gamma(R)=2\mathbb{N}$. If $p=3$ then $R$ satisfies both Drozd--Ro\u{\i}ter conditions, so it has finite Cohen-Macaulay type by the second item of Proposition~\ref{prop:CMtype_onedim}: writing $b=a^{1/3}$ and $V=\kappa(b)[[x]]$, we have $R=\kappa[[x]]\oplus bx\kappa[[x]]\oplus b^2x^2\kappa[[x]]$ and $\mathfrak{m}V=xV$, so $\frac{\mathfrak{m}V+R}{R}\cong\frac{b^2x\kappa[[x]]}{b^2x^2\kappa[[x]]}\cong \kappa$. More generally, for every finite extension $K/\kappa$ of degree $d \leq 3$ the ring $\kappa+tK[[t]]$ satisfies $e=1$ and $\mathfrak{m}V=tK[[t]] \subset R$, so it satisfies both Drozd--Ro\u{\i}ter conditions and has finite Cohen-Macaulay type. A finer classification of these rings requires additional residue field data and is not pursued here.}
\end{example}

\section{Gorenstein Rings of Finite Cohen-Macaulay Type}\label{sec:gorenstein}

Throughout this section we assume that $(R, \mathfrak{m}, \kappa)$ is a one dimensional Noetherian strictly Henselian reduced Nagata local ring which is residually rational (see Definition~\ref{def:res_rat}), and we use Notation~\ref{notation:ringstuff}. We make no assumption on the characteristic of $\kappa$. Since finite Cohen-Macaulay type, being Gorenstein, residual rationality, and the semigroups $\nu(R)$ and $\Gamma(R)$ do not change when passing to the completion (see Proposition~\ref{prop:CMtype_onedim}, Remark~\ref{rem:conductor}, and Remark~\ref{rem:Gam_comp}), we may assume that $R$ is complete whenever convenient. If $R$ is complete and $\edim(R)=2$, then as in the proof of Proposition~\ref{thm:branch} we write $R=\frac{S}{\langle f \rangle}$ with $(S,\mathfrak{n})$ a $2-$dimensional complete regular local ring and $f=f_1\cdots f_{r(R)}$, so that $\mathfrak{p}_i=f_iR$; such an $R$ is Gorenstein. 

\begin{definition}\label{def:equidominated}
    We say that $R_1$ and $R_2$ are equisingular if $r(R_1)=r(R_2)$ and $\nu(R_2)$ is obtained from $\nu(R_1)$ by a permutation of the coordinates (i.e. by reordering the minimal primes). 
\end{definition}

\begin{remark}
\begin{enumerate}
    \item \textup{The name "equisingular" comes from a similar notion of "equisingularity of algebraic curves", in which we say that two irreducible power series $f,g \in \mathbb{C}[[x,y]]$ over the complex numbers are equisingular if $\Gamma\left(\frac{\mathbb{C}[[x,y]]}{\langle f \rangle}\right) = \Gamma\left(\frac{\mathbb{C}[[x,y]]}{\langle g \rangle}\right)$. This notion over algebraic curves allows us to understand the topological type of the curves, as studied in~\cite{zariski2006moduli, zariski1965studies}, and to study deformations of plane curve singularities, as in Section 2 of Chapter 2 in~\cite{greuel2007introduction}. A notion of equisingularity exists for higher dimension and was developed by Zariski to study resolution of singularities (for more information, see~\cite{zariski1971some, parusinski2021algebro, lipman2000equisingularity}).  }
    \item \textup{The reason we look at $\nu(R)$ and not at $\Gamma(R)$ in the definition of equisingularity is because $\Gamma(R)$ does not distinguish between rings based upon the number of minimal primes (for example, Remark~\ref{rem:A_12groups} tells us that $A_1$ and $A_2$ singularities have the same $\Gamma(R)$ semigroup). In addition,  Proposition~\ref{prop:nu_beasic_prop} tells us that elements in $\nu(R)$ correspond with elements in $R$. }
\end{enumerate}
\end{remark}

\begin{lemma}\label{lem:equising_invariants}
    If $R_1$ and $R_2$ are equisingular, then 
\begin{enumerate}
    \item $\mult(R_1) =  \mult(R_2)$, 
    \item $\delta(R_1) = \delta(R_2)$, 
    \item $r(R_1)=r(R_2)$, 
    \item $\mu(R_1) = \mu(R_2)$,  
    \item $\Gamma(R_1)=\Gamma(R_2)$.
\end{enumerate}
\end{lemma}

\begin{proof}
    The multiplicity is the sum of the coordinates of the minimal element of $\nu(R)\setminus\{0\}$ by Proposition~\ref{prop:nu_beasic_prop} and Lemma~\ref{lem:mult_sum_prim}, $r(R)$ is the number of coordinates, and $\Gamma(R)$ is the image of $\nu(R)$ under the sum map by Proposition~\ref{prop:nu_beasic_prop}. By Lemma~\ref{lem:lengths}, $\delta(R)$ is determined by $\nu(R)$. Finally, $\mu$ is determined by $\delta$ and $r$. Note that, also by Lemma~\ref{lem:lengths}, $R_1$ is Gorenstein if and only if $R_2$ is.
\end{proof}

We now define the semigroups of the ADE singularities combinatorially; they do not depend on the characteristic.

\begin{definition}\label{def:ADE_semigroups}
    The semigroups of the ADE singularities are the following subsets of $\mathbb{N}$, $\mathbb{N}^2$, and $\mathbb{N}^3$:
    \begin{enumerate}
        \item $\Sigma(A_{2n})= \lceil 2, 2n+1 \rfloor \subset \mathbb{N}$ for $n \geq 1$,
        \item $\Sigma(A_{2n+1}) = \lceil (1,1)\rfloor \cup ((n+1, n+1) + \mathbb{N}^2)$ for $n \geq 0$,
        \item $\Sigma(D_{2n}) = \{(0,0,0)\} \cup \{(a,k,k) \colon a\geq 1,\ 1 \leq k \leq n-1\} \cup \{(1, n-1, l), (1,l,n-1) \colon l \geq n\} \cup  ((2,n,n) + \mathbb{N}^3)$ for $n \geq 2$,
        \item $\Sigma(D_{2n+1}) = \{(0,0)\} \cup \{(a,2s) \colon a \geq 1,\ 1 \leq s \leq n-1 \} \cup \{(1, 2n-1)\} \cup ((2, 2n)+\mathbb{N}^2)$ for $n \geq 2$,
        \item $\Sigma(E_6) = \lceil 3,4 \rfloor$, $\Sigma(E_7) = \lceil (1,2) \rfloor \cup \{(k,3) \colon k \geq 2\} \cup ((3,5) + \mathbb{N}^2)$, and $\Sigma(E_8) = \lceil 3,5 \rfloor$.
    \end{enumerate}
\end{definition}

\begin{remark}\label{rem:ADE_Gamma}
    \textup{By Lemma~\ref{lem:cond_min}, the conductor vectors corresponding to these sets are $(n+1,n+1)$ for $A_{2n+1}$, $(2,n,n)$ for $D_{2n}$, $(2,2n)$ for $D_{2n+1}$, and $(3,5)$ for $E_7$. Their images under the sum map $s$ (see Proposition~\ref{prop:nu_beasic_prop}) are $\lceil 2,2n+3\rfloor$ for $A_{2n+1}$, $\{0\}\cup(3+\mathbb{N})$ for $D_{2n}$ and $D_{2n+1}$, and $\{0,3\} \cup (5+\mathbb{N})$ for $E_7$. In particular, the conductor of $s(\Sigma(E_7))$ is $5$, while $2\delta=8$ for a ring with $\nu=\Sigma(E_7)$ which is Gorenstein (Lemma~\ref{lem:vec_mu_sum}). The sets are drawn in Figures~\ref{fig:A2n} through \ref{fig:E7} in Appendix~\ref{app:figures}.}
\end{remark}

We next recall the ADE normal forms. Note that the $A_k$ singularities below are exactly as we defined in Definition~\ref{def:A_k}. In addition, we also list the values of $\mu(R)$ and $r(R)$, which are easy to compute (in a similar fashion to the computation done in the second item of Proposition~\ref{cor:delta1}). 

\begin{definition}\label{def:ADE_curves}
    Assume that $\chara(\kappa)\neq 2$. We say that $R$ is an ADE type singularity if $\hat{R} \cong \frac{S}{\langle f \rangle}$ where $(S, \mathfrak{n}, \kappa)$ is a $2-$dimensional complete regular local ring and $\mathfrak{n} = \langle x,y \rangle$ such that $f$ equals to one of the following below:
\begin{center}
\hspace*{-0.8cm}
    \begin{tabular}{ | l | l | l | l | l |}
    \hline
    $ADE$  & Formula & $\mu(R)$ & $r(R)$ &  Note  \\ \hline \hline
    $A_{2n}$ & $x^2  + y^{2n+1}$ & $2n$ & $1$ & $1 \leq n$  \\ \hline
    $A_{2n+1}$ & $x^2  + y^{2n+2}$ & $2n+1$ & $2$ & $0 \leq n$  \\ \hline
    $D_{2n}$ & $x\left(y^2  + x^{2n-2}\right)$ & $2n$ & $3$ & $2 \leq n$  \\ \hline
    $D_{2n+1}$ & $x\left(y^2  + x^{2n-1}\right)$ & $2n+1$ & $2$ & $2 \leq n$  \\ \hline
    $E_6$ & $x^3 + y^{4}$ & $6$ & $1$ &    \\ \hline
    $E_7$ & $x\left(x^2 +  y^{3}\right)$ & $7$ & $2$ &  \\ \hline
    $E_8$ & $x^3 + y^{5}$ & $8$ & $1$ &   \\ \hline
    \end{tabular}

    \vspace{5mm}
\end{center}

\end{definition}

\begin{remark}\label{rem:char_nu}
\begin{enumerate}
    \item \textup{In~\cite{svoray2025ade} it was shown that if $R$ is a one dimensional ring of finite Cohen-Macaulay type with $\edim(R)=2$ and $\chara(\kappa) \neq 2,3,5$, then $R$ is an ADE singularity.  If $\chara(\kappa)$ is $3$ or $5$, then there are additional singularities of type $E$ (see~\cite{svoray2025ade,greuel1988simple}), namely $x^3+y^4+x^2y^2$ ($E_6^1$), $x(x^2+y^3+xy^2)$ ($E_7^1$), $x^3+y^5+x^2y^3$ ($E_8^1$), and $x^3+y^5+x^2y^2$ ($E_8^2$) in characteristic $3$, and $x^3+y^5+xy^4$ ($E_8^1$) in characteristic $5$; by Lemma~\ref{lem:ADE_nu_grp} they have the same semigroups $\nu$ as $E_6$, $E_7$, and $E_8$.}
    \item \textup{The reason we assume that $\chara(\kappa) \neq 2$ in Definition~\ref{def:ADE_curves} is that in characteristic $2$ some of the normal forms behave differently: if $R$ is equicharacteristic then $x^2+y^{2n+2}=(x+y^{n+1})^2$ and $x(y^2+x^{2n-2})=x(y+x^{n-1})^2$ are not reduced, so they do not define rings in our class. The sets of Definition~\ref{def:ADE_semigroups}, and Theorem~\ref{thm:ADE_curve} below, do not have this problem; models for these sets in every characteristic are given in Lemma~\ref{lem:ADE_nu_grp}.}
\end{enumerate}
\end{remark}

We start with rings with two regular branches, for which no assumption besides residual rationality is needed:

\begin{lemma}\label{lem:two_smooth}
    Assume that $r(R)=2$ and that $\frac{R}{\mathfrak{p}_1}$ and $\frac{R}{\mathfrak{p}_2}$ are DVRs, and let $m=i(\mathfrak{p}_1,\mathfrak{p}_2)$. Then for every $g \in R$ with $\nu(g)=(a,b) \in \mathbb{N}_\infty^2$ we have $\min(a,m)=\min(b,m)$. Moreover, $\nu(R)=\Sigma(A_{2m-1})$, $\vec{c}(R)=(m,m)$, $\delta(R)=m$, and $R$ is Gorenstein.
\end{lemma}

\begin{proof}
    Here $V_j=\frac{R}{\mathfrak{p}_j}$, and $B=\frac{R}{\mathfrak{p}_1+\mathfrak{p}_2}$ is a common quotient of $V_1$ and $V_2$ of length $m$, so $B \cong \frac{V_j}{\pi_j^mV_j}$ for $j=1,2$. For $g \in R$, the images of $g_1$ and $g_2$ in $B$ coincide, and the image of $g_j$ lies in $\mathfrak{m}_B^{t}\setminus \mathfrak{m}_B^{t+1}$ if $t=\ord(g_j,V_j)<m$, and is $0$ otherwise. This proves the first statement, and so every element of $\nu(R)$ is of the form $(k,k)$ or lies in $(m,m)+\mathbb{N}^2$. Conversely, an element with $\nu=(1,1)$ exists by Proposition~\ref{prop:nu_beasic_prop}, so $(k,k) \in \nu(R)$ for every $k$. The image of $\mathfrak{p}_2$ in $V_1$ is an ideal of colength $m$, i.e. $\pi_1^mV_1$, so $\pi_1^{l}V_1 \times \{0\} \subset R$ for every $l \geq m$, and similarly $\{0\}\times \pi_2^lV_2 \subset R$; thus $(m,m)+\mathbb{N}^2 \subset \nu(R)$ and $\pi_1^{m}V_1 \times \pi_2^mV_2 \subset R$, while $\pi_1^{m-1}V_1 \times \{0\} \not\subset R$. Hence $\nu(R)=\Sigma(A_{2m-1})$ and $\vec{c}(R)=(m,m)$. By Proposition~\ref{thm:branch}, $\delta(R)=i(\mathfrak{p}_1,\mathfrak{p}_2)=m$, so $\length_R\left(\frac{\overline{R}}{\mathfrak{c}_R}\right)=2m=2\delta(R)$, and $R$ is Gorenstein by Remark~\ref{rem:conductor}. 
\end{proof}

The next proposition computes $\nu(R)$ from the data of the branches in the remaining cases of multiplicity $3$.

\begin{proposition}\label{prop:nu_data}
    Assume that $R$ is complete with $\edim(R)=2$, and write $R=\frac{S}{\langle f\rangle}$ and $\mathfrak{p}_i=f_iR$ as above.
    \begin{enumerate}
        \item Assume that $r(R)=2$, that $\frac{R}{\mathfrak{p}_1}$ is a DVR, and that $C=\frac{R}{\mathfrak{p}_2}$ has multiplicity $2$, so that $\Gamma(C)=\lceil 2,2k+1\rfloor$ for some $k\geq 1$. Let $i=i(\mathfrak{p}_1,\mathfrak{p}_2)$. If $i=2$ then $\nu(R)=\Sigma(D_{2k+3})$. If $i=3$ then $k=1$ and $\nu(R)=\Sigma(E_7)$.
        \item Assume that $r(R)=3$, that $\frac{R}{\mathfrak{p}_j}$ is a DVR for every $j$, and that $i(\mathfrak{p}_1,\mathfrak{p}_2)=i(\mathfrak{p}_1,\mathfrak{p}_3)=1$. Let $m=i(\mathfrak{p}_2,\mathfrak{p}_3)$. Then $\nu(R)=\Sigma(D_{2m+2})$.
    \end{enumerate}
\end{proposition}

\begin{proof}
    In both cases $R$ is Gorenstein, so its conductor vector is $\vec{\mu}(R)$ by Proposition~\ref{prop:mu_vec}, and the intersection multiplicities are additive by the second item of Proposition~\ref{thm:branch}. Let $x \in R$ be the image of $f_1$, so that $\mathfrak{p}_1=xR$, and choose $y \in S$ such that $\mathfrak{n}=\langle f_1, y \rangle$; then $V_1=\frac{R}{\mathfrak{p}_1}$ is a DVR whose maximal ideal is generated by the image of $y$, and for every $j \neq 1$ we have $\ord(x_j, V_j)=i(\mathfrak{p}_1,\mathfrak{p}_j)$ by the third item of Lemma~\ref{lem:semigroup}. We write $v_j=\ord(-,V_j)$, extended to $R$ via $R \to V_j$ (with $v_j(0)=\infty$). Since $R \to V_1$ is surjective and maps $\mathfrak{m}$ onto $\pi_1V_1$, an element $g \in R$ with $v_1(g)\geq 2$ lies in $\mathfrak{m}^2+xR$.\\

    For the first item, the ring $C$ is a complete residually rational domain of multiplicity $2$, hence of finite Cohen-Macaulay type by the fourth item of Proposition~\ref{prop:CMtype_onedim}, so $\Gamma(C)=\lceil 2,2k+1\rfloor$ with $k=\delta(C)\geq 1$ by Theorem~\ref{thm:FCMT_DOm}. We have $v_2(x)=i$, and $A_0=\frac{R}{\mathfrak{p}_1+\mathfrak{p}_2}$ is isomorphic both to $\frac{C}{xC}$ and to $\frac{V_1}{\pi_1^iV_1}$, the image of $\mathfrak{p}_2$ in $V_1$ being an ideal of colength $i$. By Proposition~\ref{prop:mu_vec}, $\vec{c}(R)=(i,2k+i)$. Every $g \in \mathfrak{m}$ can be written as $g=\alpha y+\beta x$ with $\alpha, \beta \in R$, and then $v_1(g)=v_1(\alpha)+1$.\\

    Assume that $i=2$. Then $\mathfrak{m}_C^2 \subset xC$, since $\frac{C}{xC}$ has length $2$. We first arrange $v_2(y)=2k+1$. If $v_2(y)=2j$ is even, we replace $y$ by $y-\lambda x^{j}$ for a suitable $\lambda \in R$ (using residual rationality and $v_2(x^j)=2j$), which increases $v_2(y)$ and does not change $v_1(y)$ or the ideal $\langle x,y \rangle$. We cannot reach $v_2(y) \geq 2k+2=c_2$: otherwise $(0,y_2) \in \mathfrak{c}_R \subset R$, so $(y_1,0) \in R \cap V_1=\mathfrak{q}_1=\mathfrak{p}_2$ would be an element of $\mathfrak{p}_2$ with $v_1=1$, contradicting that the image of $\mathfrak{p}_2$ in $V_1$ is $\pi_1^2V_1$. Thus after finitely many steps $v_2(y)$ is odd and smaller than $2k+2$, so $v_2(y)=2k+1$, as every odd element of $\Gamma(C)$ is at least $2k+1$. Now let $g \in \NZD(R)$ be a non unit with $\nu(g)=(a,b)$. If $a=1$ then $\alpha$ is a unit, $v_2(\alpha y)=2k+1$, and $v_2(\beta x) \in 2+\Gamma(C)$ (or is $\infty$), which does not contain $2k+1$ as $2k-1 \notin \Gamma(C)$; so $b=\min(2k+1, v_2(\beta x)) \in \{2,4,\dots,2k\}\cup\{2k+1\}$. If $a \geq 2$ then $g \in \mathfrak{m}^2+xR$, so its image in $C$ lies in $\mathfrak{m}_C^2+xC=xC$, and $b \in 2+\Gamma(C)=\{2,4,\dots,2k\} \cup (2k+2+\mathbb{N})$. Conversely, $\nu(x^s+y^a)=(a,2s)$ for $a \geq 1$ and $1 \leq s \leq k$ (as $2s<a(2k+1)$), $\nu(y)=(1,2k+1)$, and $(2,2k+2)+\mathbb{N}^2 \subset \nu(R)$ by Lemma~\ref{lem:cond_min}. Thus $\nu(R)=\Sigma(D_{2k+3})$.\\

    Assume that $i=3$. Then $3 = v_2(x) \in \Gamma(C)$, so $k=1$ and $\Gamma(C)=\lceil 2,3\rfloor$, and $v_2(y)=2$ (otherwise all elements of $\mathfrak{m}_C=xC+yC$ would have value at least $3$). Since $A_0\cong \frac{V_1}{\pi_1^3V_1}$ and $y$ maps to a uniformizer of $V_1$, the image of $y^3$ in $C$ lies in $xC$. Let $g \in \NZD(R)$ be a non unit with $\nu(g)=(a,b)$. If $a=1$ then $\alpha$ is a unit, $v_2(\alpha y)=2$ and $v_2(\beta x) \geq 3$, so $b=2$. If $a=2$ then $v_1(\alpha)=1$, so $\alpha=\alpha'y+\alpha''x$ with $\alpha'$ a unit, and $g=\alpha'y^2+(\alpha''y+\beta)x$; here $v_2(\alpha'y^2)=4$ and $v_2((\alpha''y+\beta)x) \in 3+\Gamma(C)=\{3\}\cup(5+\mathbb{N})$ (or is $\infty$), so $b \in \{3,4\}$. If $a \geq 3$ then $g \in \mathfrak{m}^3+xR \subset xR+y^3R$, so the image of $g$ in $C$ lies in $xC$ and $b \in 3+\Gamma(C)=\{3\}\cup(5+\mathbb{N})$. Conversely, $\nu(y^k)=(k,2k)$, $\nu(x+y^k)=(k,3)$ for $k\geq 2$, and $\vec{c}(R)=(3,5)$, so $(3,5)+\mathbb{N}^2 \subset \nu(R)$. Thus $\nu(R)=\Sigma(E_7)$.\\

    For the second item, let $A=\frac{R}{\mathfrak{q}_1}$, where $\mathfrak{q}_1=\mathfrak{p}_2\cap\mathfrak{p}_3$. By Lemma~\ref{lem:two_smooth}, $\nu(A)=\Sigma(A_{2m-1})$ and the conductor of $A$ is $\mathfrak{c}_A=\pi_2^mV_2\times \pi_3^mV_3$. By additivity, $i(\mathfrak{p}_1,\mathfrak{q}_1)=2$ and $i(\mathfrak{p}_2,\mathfrak{q}_2)=i(\mathfrak{p}_3,\mathfrak{q}_3)=m+1$, so $\vec{c}(R)=(2,m+1,m+1)$; the image of $\mathfrak{q}_1$ in $V_1$ is $\pi_1^2V_1$, and the image of $\mathfrak{q}_j$ in $V_j$ is $\pi_j^{m+1}V_j$ for $j=2,3$. Moreover, $x$ is a non zero divisor of $A$ with $\nu_A(x)=(1,1)$, and $\frac{A}{xA}\cong\frac{R}{\mathfrak{p}_1+\mathfrak{q}_1}$ has length $2$, so $\mathfrak{m}_A^2 \subset xA$. Let $g \in R$ with $v_1(g)=a<\infty$, and let $b=v_2(g)$ and $c=v_3(g)$.
    \begin{itemize}
        \item If $a \geq 2$ then $g \in \mathfrak{m}^2+xR$, so the image of $g$ in $A$ lies in $\mathfrak{m}_A^2+xA=xA$. If moreover $g \in \NZD(R)$, then $(b,c) \in (1,1)+\nu(A)$, so $b=c$ or $b,c \geq m+1$.
        \item If $a=1$ then $\min(b,c)\leq m$. Indeed, otherwise the image $\overline{g}$ of $g$ in $\overline{A}$ satisfies $\frac{\overline{g}}{x} \in \pi_2^mV_2 \times \pi_3^mV_3=\mathfrak{c}_A \subset A$, so $\overline{g} \in xA$, and $g \in xR+\mathfrak{q}_1$; then the image of $g$ in $V_1$ lies in the image of $\mathfrak{q}_1$, which is $\pi_1^2V_1$, a contradiction. Moreover, if $b=\infty$ then $c \geq m$, as the image of $\mathfrak{p}_2$ in $V_3$ is $\pi_3^mV_3$, so $c=m$; similarly if $c=\infty$ then $b=m$. If $b,c$ are finite then $(b,c) \in \nu(A)$, so $b=c\leq m$ or $\min(b,c)=m$.
    \end{itemize}
    This shows that $\nu(R) \subset \Sigma(D_{2m+2})$. For the converse, note first that for every $l \geq m+1$ there are elements of $\mathfrak{q}_2$ and of $\mathfrak{q}_3$ with values $(\infty,l,\infty)$ and $(\infty,\infty,l)$, respectively. Let $y \in R$ with $v_1(y)=1$. If $(v_2(y),v_3(y))=(j,j)$ with $j<m$, we replace $y$ by $y-\lambda x^j$ for a suitable $\lambda \in R$, which does not change $v_1(y)$ and increases $v_2(y)$; by the previous paragraph, $v_3(y)$ increases as well. After finitely many steps $\min(v_2(y),v_3(y))=m$, and after adding an element of $\mathfrak{q}_2$ or $\mathfrak{q}_3$ of large value if necessary, we may assume that $v_2(y)$ and $v_3(y)$ are finite. Then:
    \begin{itemize}
        \item $\nu(x^j+y)=(1,j,j)$ for $1 \leq j \leq m-1$, and adding $\lambda x^m$ to $y$ for a suitable unit $\lambda$ gives an element $y'$ with $\nu(y')=(1,m,m)$;
        \item replacing $y'$ by $y'-\lambda x^m$ for a suitable $\lambda \in R$, we get an element with value $(1,b',m)$ with $b' \geq m+1$ (possibly $b'=\infty$; the third coordinate is still $m$ by the previous paragraph), and adding or subtracting elements of $\mathfrak{q}_2$ we get elements with value $(1,l,m)$ for every $l \geq m+1$; symmetrically we get $(1,m,l)$ for every $l \geq m+1$;
        \item $\nu(x^j+y'^a)=(a,j,j)$ for $a\geq 2$ and $1 \leq j \leq m$, as $am>j$;
        \item $(2,m+1,m+1)+\mathbb{N}^3 \subset \nu(R)$ by Lemma~\ref{lem:cond_min}.
    \end{itemize}
    Thus $\nu(R)=\Sigma(D_{2m+2})$.
\end{proof}

We can now prove the main result of this section, which holds in every characteristic.

\begin{theorem}\label{thm:ADE_curve}
    Assume that $R$ is Gorenstein and not regular. Then $R$ has finite Cohen-Macaulay type if and only if $\nu(R)$ is, up to a permutation of the coordinates, one of the sets $\Sigma(X)$ of Definition~\ref{def:ADE_semigroups}. More precisely, if $R$ has finite Cohen-Macaulay type then $\edim(R)=2$ and exactly one of the following holds:
    \begin{enumerate}
        \item $\mult(R)=2$, $r(R)=1$, and $\nu(R)=\Sigma(A_{2n})$ with $n=\delta(R)$;
        \item $\mult(R)=2$, $r(R)=2$, both branches are regular, and $\nu(R)=\Sigma(A_{2m-1})$ with $m=i(\mathfrak{p}_1,\mathfrak{p}_2)=\delta(R)$;
        \item $\mult(R)=3$, $r(R)=1$, and $\nu(R)$ is $\Sigma(E_6)$ or $\Sigma(E_8)$;
        \item $\mult(R)=3$, $r(R)=2$, and, ordering the branches so that $\frac{R}{\mathfrak{p}_1}$ is regular, $i(\mathfrak{p}_1,\mathfrak{p}_2) \in \{2,3\}$ and $\nu(R)$ is $\Sigma(D_{2n+1})$ or $\Sigma(E_7)$, as in the first item of Proposition~\ref{prop:nu_data};
        \item $\mult(R)=3$, $r(R)=3$, the branches are regular, and after renumbering them $i(\mathfrak{p}_1,\mathfrak{p}_2)=i(\mathfrak{p}_1,\mathfrak{p}_3)=1$ and $\nu(R)=\Sigma(D_{2m+2})$ with $m=i(\mathfrak{p}_2,\mathfrak{p}_3)$.
    \end{enumerate}
    In particular, if $\chara(\kappa) \neq 2$ then $R$ is equisingular to an ADE type singularity by Lemma~\ref{lem:ADE_nu_grp}.
\end{theorem}

\begin{proof}
    We may assume that $R$ is complete. Assume first that $R$ has finite Cohen-Macaulay type. By the second item of Proposition~\ref{prop:CMtype_onedim}, $\mult(R)\leq 3$ and $Q=\frac{\mathfrak{m}\overline{R}+R}{R}$ is cyclic. By Lemma~\ref{lem:gor_minmult} we have $\edim(R)\leq 2$, and so $\edim(R)=2$ as $R$ is not regular; write $R=\frac{S}{\langle f \rangle}$ and $\mathfrak{p}_i=f_iR$. By Lemma~\ref{lem:mult_sum_prim}, $\mult(R)=\sum_i \mult\left(\frac{R}{\mathfrak{p}_i}\right)$, which leaves the following cases. If $\mult(R)=2$ and $r(R)=1$, then $R$ is a residually rational complete domain of multiplicity $2$, and $\Gamma(R)=\lceil 2, 2n+1 \rfloor$ with $n=\delta(R)$ by Theorem~\ref{thm:FCMT_DOm}. If $\mult(R)=2$ and $r(R)=2$, both branches are regular, and Lemma~\ref{lem:two_smooth} applies. If $\mult(R)=3$ and $r(R)=1$, Corollary~\ref{cor:FCMT_Gor_domain} applies.\\

    If $\mult(R)=3$ and $r(R)=2$, we order the branches so that $\frac{R}{\mathfrak{p}_1}$ is regular and $\frac{R}{\mathfrak{p}_2}$ has multiplicity $2$. Let $x$ and $y$ be as in the proof of Proposition~\ref{prop:nu_data}, and let $i=i(\mathfrak{p}_1,\mathfrak{p}_2)=v_2(x)\geq 2$. Here $\mathfrak{m}\overline{R}=\pi_1V_1 \times \pi_2^2V_2$ and $\mathfrak{m}^2\overline{R}=\pi_1^2V_1 \times \pi_2^4V_2$, and the image $W$ of $\mathfrak{m}=\langle x,y \rangle$ in the $3-$dimensional space $\frac{\mathfrak{m}\overline{R}}{\mathfrak{m}^2\overline{R}}$ is spanned by the images of $x$ and $y$. The image of $y$ has a nonzero first coordinate, while the image of $x$ has first coordinate $0$, and it is nonzero if and only if $v_2(x)\leq 3$. By Lemma~\ref{lem:DR2_general}, $\dim_\kappa(W)\geq 2$, so $i \in \{2,3\}$, and the first item of Proposition~\ref{prop:nu_data} applies.\\

    If $\mult(R)=3$ and $r(R)=3$, all branches are regular, and $\frac{\mathfrak{m}\overline{R}}{\mathfrak{m}^2\overline{R}}=\prod_j \frac{\pi_jV_j}{\pi_j^2V_j}\cong \kappa^3$. Let $W \subset \kappa^3$ be the image of $\mathfrak{m}$, and let $\varphi_1,\varphi_2,\varphi_3$ be the coordinate functions on $W$; they are nonzero since $(1,1,1)\in \nu(R)$. We claim that $\varphi_j$ and $\varphi_k$ are linearly dependent if and only if $i(\mathfrak{p}_j,\mathfrak{p}_k)\geq 2$. Indeed, $B=\frac{R}{\mathfrak{p}_j+\mathfrak{p}_k}$ is a common quotient of $V_j$ and $V_k$ of length $i(\mathfrak{p}_j,\mathfrak{p}_k)$. If this length is $1$, then $\mathfrak{p}_j+\mathfrak{p}_k=\mathfrak{m}$, so the image of $\mathfrak{p}_k$ in $V_j$ is $\pi_jV_j$, and an element of $\mathfrak{p}_k$ with value $1$ on the $j-$th branch shows that $\varphi_j \neq 0=\varphi_k$ at some point of $W$; by symmetry $\varphi_j$ and $\varphi_k$ are independent. If the length is at least $2$, then $\frac{\mathfrak{m}_B}{\mathfrak{m}_B^2}$ is one dimensional and spanned by the image of $\pi_j$ and by that of $\pi_k$, say $\overline{\pi_k}=\lambda\overline{\pi_j}$; since the images of $g_j$ and $g_k$ in $B$ coincide for $g \in \mathfrak{m}$, we get $\varphi_j=\lambda\varphi_k$. Now $\dim_\kappa(W)\leq \edim(R)=2$, and $\dim_\kappa(W) \geq 2$ by Lemma~\ref{lem:DR2_general}, so $\dim_\kappa(W)=2$ and the three $\varphi_j$ are not all pairwise dependent. Since linear dependence is an equivalence relation on nonzero functionals, at most one pair is dependent, so after renumbering $i(\mathfrak{p}_1,\mathfrak{p}_2)=i(\mathfrak{p}_1,\mathfrak{p}_3)=1$, and the second item of Proposition~\ref{prop:nu_data} applies.\\

    Conversely, assume that $\nu(R)=\Sigma(X)$ up to a permutation of the coordinates. Then $\mult(R)\leq 3$, as the multiplicity is the sum of the coordinates of the minimal element of $\nu(R)\setminus\{0\}$ (Proposition~\ref{prop:nu_beasic_prop} and Lemma~\ref{lem:mult_sum_prim}), so $\edim(R)=2$ by Lemma~\ref{lem:gor_minmult}. If $\mult(R)=2$, then $R$ has finite Cohen-Macaulay type by the fourth item of Proposition~\ref{prop:CMtype_onedim}, and if $r(R)=1$, then $R$ has finite Cohen-Macaulay type by Theorem~\ref{thm:FCMT_DOm}. If $r(R)=2$ and $\mult(R)=3$, then $X \in \{D_{2n+1}, E_7\}$, the regular branch corresponds to the coordinate equal to $1$ in the minimal element, and by Lemma~\ref{lem:cond_min} and Proposition~\ref{prop:mu_vec} we get $i(\mathfrak{p}_1,\mathfrak{p}_2)=c_1 \in \{2,3\}$; by the computation above, $\dim_\kappa(W) \geq 2$. If $r(R)=3$, then $X=D_{2n}$ and, numbering the branches as in Definition~\ref{def:ADE_semigroups}, $c_1=2=i(\mathfrak{p}_1,\mathfrak{q}_1)=i(\mathfrak{p}_1,\mathfrak{p}_2)+i(\mathfrak{p}_1,\mathfrak{p}_3)$, so $i(\mathfrak{p}_1,\mathfrak{p}_2)=i(\mathfrak{p}_1,\mathfrak{p}_3)=1$ and $\varphi_1,\varphi_2$ are independent, i.e. $\dim_\kappa(W)\geq 2$. In both cases $Q$ is cyclic by Lemma~\ref{lem:DR2_general}, and $R$ has finite Cohen-Macaulay type by the second item of Proposition~\ref{prop:CMtype_onedim}.
\end{proof}

\begin{remark}
\begin{enumerate}
    \item \textup{If $\chara(\kappa) \neq 2$, the ``only if'' part of Theorem~\ref{thm:ADE_curve} can also be deduced from Herzog's theorem that a Gorenstein local ring of finite Cohen-Macaulay type is a hypersurface~\cite{herzog1978ringe}, together with the classification of~\cite{greuel1984einfache, svoray2025ade} and Lemma~\ref{lem:ADE_nu_grp} (which also covers the additional normal forms in characteristics $3$ and $5$). In dimension one, Lemma~\ref{lem:gor_minmult} replaces Herzog's theorem. The proof above does not use any classification or normal forms, only the Drozd--Ro\u{\i}ter conditions, and so it applies in every characteristic, including equal and mixed characteristic $2$.}
    \item \textup{As in The Gorenstein assumption cannot be dropped, see Example~\ref{ex:nonGor}.}
    \item \textup{The assumption that $R$ is residually rational cannot be dropped either: if $\kappa$ is a separably closed non perfect field of characteristic $2$ and $a \in \kappa\setminus\kappa^2$, then $R=\frac{\kappa[[x,y]]}{\langle y^2-ax^2 \rangle}$ is a Gorenstein domain of finite Cohen-Macaulay type with $\nu(R)=\Gamma(R)=2\mathbb{N}$ (Example~\ref{ex:insep_FCMT}), which is not the semigroup of an ADE singularity.}
\end{enumerate}
\end{remark}

The following lemma identifies the semigroups of the ADE normal forms, and gives models for the sets of Definition~\ref{def:ADE_semigroups} in every characteristic.

\begin{lemma}\label{lem:ADE_nu_grp}
    \begin{enumerate}
        \item Assume that $\chara(\kappa)\neq 2$. If $R$ is an ADE type singularity of type $X$ (see Definition~\ref{def:ADE_curves}), then $\nu(R)=\Sigma(X)$, with the minimal primes ordered as $\langle x+iy^{n+1}\rangle, \langle x-iy^{n+1}\rangle$ for $A_{2n+1}$, as $\langle x \rangle, \langle y+ix^{n-1} \rangle, \langle y-ix^{n-1} \rangle$ for $D_{2n}$, and with the regular branch $\langle x \rangle$ first for $D_{2n+1}$ and $E_7$ (where $i \in R$ satisfies $i^2=-1$). The same holds for the additional singularities of type $E$ in characteristics $3$ and $5$ listed in Remark~\ref{rem:char_nu}.
        \item In every characteristic, if $\hat{R}\cong\frac{S}{\langle f \rangle}$ with $(S,\mathfrak{n})$ a $2-$dimensional complete regular local ring with $\mathfrak{n}=\langle x,y\rangle$, then $\nu(R)$ is $\Sigma(A_{2n})$, $\Sigma(A_{2n+1})$, $\Sigma(D_{2n})$, $\Sigma(D_{2n+1})$, $\Sigma(E_6)$, $\Sigma(E_7)$, or $\Sigma(E_8)$ if $f$ is $x^2+y^{2n+1}$, $y(y+x^{n+1})$, $xy(y+x^{n-1})$, $x(y^2+x^{2n-1})$, $x^3+y^4$, $x(x^2+y^3)$, or $x^3+y^5$, respectively. These rings are Gorenstein and have finite Cohen-Macaulay type by Theorem~\ref{thm:ADE_curve}.
    \end{enumerate}
\end{lemma}

\begin{proof}
    We first note that the polynomials $x^2+y^{2n+1}$, $y^2+x^{2n-1}$, $x^2+y^3$, $x^2+y^3+xy^2$, $x^3+y^4$, $x^3+y^5$, and the additional forms $x^3+y^4+x^2y^2$, $x^3+y^5+x^2y^3$, $x^3+y^5+x^2y^2$, $x^3+y^5+xy^4$ are irreducible in $S$. Indeed, each of them has order at most $3$ and initial form a power of a linear form $\ell \in \{x,y\}$, so if it were reducible it would have a factor $g\in \mathfrak{n}\setminus\mathfrak{n}^2$ whose initial form is a multiple of $\ell$. In the DVR $\frac{S}{\langle g \rangle}$, whose maximal ideal is generated by the image of the other variable, the image of $\ell$ then has order at least $2$, and one checks directly that the image of the polynomial has finite order: the orders of the two pure powers are an even and an odd number in the first three cases (as in the proof of Proposition~\ref{cor:delta1}), and in the remaining cases the pure power of the other variable has order $3$, $4$, or $5$, while all other terms have larger order. This contradicts that the polynomial lies in $\langle g\rangle$.\\

    Now each case follows from the results above, by computing intersection multiplicities as lengths of quotients of $S$. For $A_{2n}$ and $E_6$, $E_8$ (and the additional forms of type $E_6$, $E_8$) the ring is a domain of multiplicity $2$, respectively $3$. In the first case $l(y,R)=2$ and $l(x,R)=2n+1$, so $R$ is residually rational (the lengths of a domain are multiples of its residue degree), $\delta(R)=n$ by the proof of the second item of Proposition~\ref{cor:delta1} (which does not use $\chara(\kappa) \neq 2$ when $k$ is even), and so $\Gamma(R)=\lceil 2,2n+1\rfloor$ by Theorem~\ref{thm:FCMT_DOm}. In the other cases $l(y,R)=3$ and $l(x,R) \in \{4,5\}$, so again $R$ is residually rational, and since $R$ is Gorenstein, $\Gamma(R)$ is symmetric (proof of Corollary~\ref{cor:FCMT_Gor_domain}); the only symmetric numerical semigroups of multiplicity $3$ containing $4$, respectively $5$, are $\lceil 3,4\rfloor$ and $\lceil 3,5\rfloor$ (proof of Theorem~\ref{thm:FCMT_DOm}). For $A_{2n+1}$, the two branches are regular with intersection multiplicity $n+1$: we have $\length_S\left(\frac{S}{\langle x+iy^{n+1}, x-iy^{n+1}\rangle}\right)=\length_S\left(\frac{S}{\langle x, y^{n+1}\rangle}\right)=n+1$ for the normal form (as $2i$ is a unit), and $\length_S\left(\frac{S}{\langle y, y+x^{n+1}\rangle}\right)=n+1$ for the model; so Lemma~\ref{lem:two_smooth} applies. For $D_{2n}$, the three branches are regular, the first one meets the other two with intersection multiplicity $1$, and the last two meet with intersection multiplicity $n-1$, so the second item of Proposition~\ref{prop:nu_data} applies with $m=n-1$. For $D_{2n+1}$, the branch $C=\frac{S}{\langle y^2+x^{2n-1} \rangle}$ satisfies $l(x,C)=2$ and $l(y,C)=2n-1$, so as in the case $A_{2n}$ it is residually rational with $\Gamma(C)=\lceil 2,2n-1\rfloor$ (so $k=n-1$), and it meets $\langle x \rangle$ with intersection multiplicity $\length_S\left(\frac{S}{\langle x,y^2\rangle}\right)=2$. For $E_7$ (and $E_7^1$), the branch $C=\frac{S}{\langle x^2+y^3 (+xy^2)\rangle}$ satisfies $l(y,C)=2$ and $l(x,C)=3$, so it is residually rational with $\Gamma(C)=\lceil 2,3\rfloor$, and it meets $\langle x \rangle$ with intersection multiplicity $\length_S\left(\frac{S}{\langle x,y^3 \rangle}\right)=3$. In the last two cases the first item of Proposition~\ref{prop:nu_data} applies. 
\end{proof}

We end with the coordinate axes and with the non-Gorenstein case.

\begin{remark}\label{rem:omp_nu}
    \textup{The ring $R$ satisfies $\delta(R)=r(R)-1$ (see Proposition~\ref{prop:axes}) if and only if $\nu(R) = \{0\} \cup (\mathds{1} + \mathbb{N}^{r(R)})$. Indeed, if $\delta(R)=r(R)-1$ then $\mathfrak{m}$ is the Jacobson radical of $\overline{R}$, so every non unit of $R$ has all its values positive and every such tuple of values occurs. Conversely, if $\nu(R) = \{0\} \cup (\mathds{1} + \mathbb{N}^{r(R)})$ then $\vec{c}(R) \leq \mathds{1}$ by Lemma~\ref{lem:cond_min}, so the Jacobson radical of $\overline{R}$ is contained in $R$, and the claim follows from Lemma~\ref{lem:milnor_nonneg}. Note that the semigroup of an ordinary multiple point is different: for example, $\Sigma(D_4)$ is described in Definition~\ref{def:ADE_semigroups}.}
\end{remark}

Without the Gorenstein assumption, Theorem~\ref{thm:ADE_curve} fails (the semigroups of all non-Gorenstein rings of finite Cohen-Macaulay type are determined in Section~\ref{sec:nonGorenstein}):

\begin{example}\label{ex:nonGor}
\textup{Let $k$ be any field. Each of the following rings $R$ has multiplicity $3$ and satisfies $\mathfrak{m}\overline{R} \subset R$, so $\frac{\mathfrak{m}\overline{R}+R}{R}=0$ and $R$ has finite Cohen-Macaulay type by the second item of Proposition~\ref{prop:CMtype_onedim}; but $\nu(R)$ is not the semigroup of an ADE singularity.
\begin{enumerate}
    \item $R=k[[t^3,t^4,t^5]]$, which contains $k[[t^3,t^4]]$. Here $\Gamma(R)=\lceil 3,4,5\rfloor$, while the only ADE semigroups of domains of multiplicity $3$ are $\lceil 3,4\rfloor$ and $\lceil 3,5\rfloor$.
    \item The three coordinate axes $R=\frac{k[[x,y,z]]}{\langle xy,yz,xz\rangle}$. Here $\nu(R)=\{0\}\cup (\mathds{1}+\mathbb{N}^3)$ by Remark~\ref{rem:omp_nu}, which contains all permutations of $(1,1,2)$ and $(1,2,2)$; but $\Sigma(D_4)$ does not contain $(1,2,2)$, and for $n\geq 3$, $\Sigma(D_{2n})$ does not contain $(1,1,2)$ or $(1,2,1)$.
    \item $R=\{(a,b) \in k[[s]] \times k[[t^2,t^3]] \colon a(0)=b(0)\}$. Here $\nu(R)=\{0\}\cup ((1,2)+\mathbb{N}^2)$, which contains $(2,3)$ and $(1,3)$; but $(2,3) \notin \Sigma(D_5)$, $(1,3) \notin \Sigma(D_{2n+1})$ for $n\geq 3$, and $(1,3)\notin \Sigma(E_7)$.
\end{enumerate}}
\end{example}

\begin{remark}\label{rem:ADE_semigroup}
    \textup{In general, by the classification of Greuel and Kn\"{o}rrer~\cite{greuel1984einfache} (for rings containing an algebraically closed field of characteristic zero; see~\cite{wiegand1991curve, wiegand1994one} for extensions), a one dimensional ring $R$ of finite Cohen-Macaulay type birationally dominates an ADE singularity $A$, i.e. $A \subset R \subset \overline{A}$. Since $A$ and $R$ have the same total ring of fractions, non zero divisors of $A$ are non zero divisors of $R$, and so $\nu(A) \subset \nu(R)$ (after ordering the minimal primes compatibly). Thus, without the Gorenstein assumption, the correct analogue of Theorem~\ref{thm:ADE_curve} is the containment $\nu(R)\supseteq \nu(A)$ for some ADE type singularity $A$; Example~\ref{ex:nonGor} shows that equality fails in general. In Corollary~\ref{cor:containment} we show that this containment characterizes finite Cohen-Macaulay type, in every characteristic.}
\end{remark}

\section{Non-Gorenstein Rings of Finite Cohen-Macaulay Type}\label{sec:nonGorenstein}

In this section we complete the classification by determining the semigroups of the rings of finite Cohen-Macaulay type which are not Gorenstein. As in Section~\ref{sec:gorenstein}, throughout this section we assume that $(R, \mathfrak{m}, \kappa)$ is a one dimensional Noetherian strictly Henselian reduced Nagata local ring which is residually rational, and we use Notation~\ref{notation:ringstuff}; no assumption is made on the characteristic of $\kappa$. The key point is that such rings have minimal multiplicity, and so they are determined by their first neighborhood.\\

Given a subring $B$ with $R \subset B \subset \overline{R}$, we write $\nu(B) \subset \mathbb{N}^{r(R)}$ for the set of the tuples $\nu(b)=(\ord(b_i,V_i))_i$, where $b \in B$ is a non zero divisor of $\overline{R}$. If $B$ is a product of local rings $B_1 \times \cdots \times B_s$, then $\nu(B)=\nu(B_1)\times \cdots \times \nu(B_s)$, after grouping the coordinates according to the factors.

\begin{lemma}\label{lem:minmult}
    Assume that $\mult(R)=\edim(R)$, let $x \in \mathfrak{m}$ be a superficial element, and let $\vec{m}=(\mult(\frac{R}{\mathfrak{p}_1}), \dots, \mult(\frac{R}{\mathfrak{p}_{r(R)}}))$ be the minimal element of $\nu(R)\setminus\{0\}$ (see Proposition~\ref{prop:nu_beasic_prop}). Then:
    \begin{enumerate}
        \item $x$ is a non zero divisor with $\nu(x)=\vec{m}$, and $\mathfrak{m}^2=x\mathfrak{m}$.
        \item $B=\{z \in \overline{R} \colon xz \in \mathfrak{m}\}$ is a subring of $\overline{R}$ containing $R$, which is finite over $R$, and $\mathfrak{m}=xB$; in fact $B=R^{(1)}$. Moreover, $B$ is a finite product of local rings, each of which satisfies the standing assumptions of this section.
        \item $\nu(R)=\{0\} \cup (\vec{m}+\nu(B))$.
        \item $\frac{\overline{R}}{x\overline{R}}$ is a $\kappa-$vector space of dimension $\mult(R)$, and $\frac{\mathfrak{m}\overline{R}+R}{R}$ is a cyclic $R-$module if and only if the image of $B$ in $\frac{\overline{R}}{x\overline{R}}$ has dimension at least $\mult(R)-1$.
    \end{enumerate}
\end{lemma}

\begin{proof}
    For the first item, $x$ is a non zero divisor as in the proof of Lemma~\ref{lem:gor_minmult}, and by the first and third items of Lemma~\ref{lem:semigroup} we have $\sum_i \ord(x_i,V_i)=l(x,R)=\mult(R)=\sum_i m_i$, while $\ord(x_i,V_i)\geq m_i$ for every $i$; hence $\nu(x)=\vec{m}$. As in the proof of Lemma~\ref{lem:gor_minmult}, $\mult(R)=\edim(R)$ implies $\mathfrak{m}^2 \subset xR$. If $z=xr \in \mathfrak{m}^2$ with $r \in R$, then $r \in \mathfrak{m}$, as otherwise $x \in \mathfrak{m}^2$; thus $\mathfrak{m}^2 \subset x\mathfrak{m} \subset \mathfrak{m}^2$.\\

    For the second item, every $g \in \mathfrak{m}$ satisfies $\frac{g}{x}\in \overline{R}$ (its coordinates have order $\ord(g_i,V_i)-m_i \geq 0$), so $\mathfrak{m}=xB$, and $B \cong \mathfrak{m}$ is a finitely generated $R-$module containing $R$ (as $x \in \mathfrak{m}$). If $z,z' \in B$ then $x(zz')=\frac{(xz)(xz')}{x} \in \frac{\mathfrak{m}^2}{x}=\mathfrak{m}$, so $B$ is a ring. Since $B=\frac{\mathfrak{m}}{x}$, we have $B=R[\frac{x_1}{x},\dots,\frac{x_d}{x}]=R^{(1)}$ for generators $x_1,\dots,x_d$ of $\mathfrak{m}$ (Remark~\ref{rem:reduction}). As $R$ is Henselian, $B$ is a finite product of local rings, each of which is finite over $R$ (\cite[Tag 04GH]{stacks-project}); each of them is a one dimensional reduced Henselian Nagata local ring (\cite[Tag 032T]{stacks-project}) whose residue field lies between $\kappa$ and some $\kappa_i=\kappa$, so it satisfies our standing assumptions.\\

    For the third item, units of $R$ have $\nu=0$. A non zero divisor $g \in \mathfrak{m}=xB$ can be written as $g=xb$ with $b=\frac{g}{x} \in B$ a non zero divisor of $\overline{R}$, and conversely every such $xb$ is a non zero divisor of $R$ lying in $\mathfrak{m}$; and $\nu(xb)=\vec{m}+\nu(b)$.\\

    For the fourth item, since $1 \in B$ we have $\mathfrak{m}\overline{R}=x\overline{R}$ and $\mathfrak{m}^2\overline{R}=x^2\overline{R}$, so by Lemma~\ref{lem:DR2_general} $\frac{\overline{R}}{x\overline{R}}\cong \frac{\mathfrak{m}\overline{R}}{\mathfrak{m}^2\overline{R}}$ is a $\kappa-$vector space of dimension $\mult(R)$, and, as $x$ is a non zero divisor on $\overline{R}$, 
    \begin{equation*}
        \frac{\mathfrak{m}\overline{R}}{\mathfrak{m}^2\overline{R}+\mathfrak{m}}=\frac{x\overline{R}}{x^2\overline{R}+xB}\cong \frac{\overline{R}}{x\overline{R}+B}.
    \end{equation*}
    The claim follows from Lemma~\ref{lem:DR2_general}.
\end{proof}

\begin{lemma}\label{lem:nonGor_minmult}
    If $R$ has finite Cohen-Macaulay type and is not Gorenstein, then $\mult(R)=\edim(R)=3$, and the minimal element of $\nu(R)\setminus\{0\}$ is, after renumbering the minimal primes, one of $(3)$, $(1,2)$, or $(1,1,1)$. Conversely, if $\mult(R)=\edim(R)=3$, then $R$ is not Gorenstein.
\end{lemma}

\begin{proof}
    By the second item of Proposition~\ref{prop:CMtype_onedim} we have $\mult(R)\leq 3$. If $\edim(R)\leq 2$, then by Lemma~\ref{lem:edim_CST} and since $\hat{R}$ is reduced of dimension one, $\hat{R}$ is a hypersurface (as in the proof of Proposition~\ref{thm:branch}), hence Gorenstein, and so $R$ is Gorenstein. Since $\edim(R)\leq \mult(R)$ (Remark~\ref{rem:techincal_ring}), we get $\mult(R)=\edim(R)=3$. The minimal element of $\nu(R)\setminus\{0\}$ has positive coordinates whose sum is $\mult(R)=3$ by Proposition~\ref{prop:nu_beasic_prop} and Lemma~\ref{lem:mult_sum_prim}. The converse is Lemma~\ref{lem:gor_minmult}.
\end{proof}

We now define the semigroups that occur. For $\Sigma' \subset \mathbb{N}^s$ with $0 \in \Sigma'$ we write $(\mathbb{N}\setminus\{0\})\times(\Sigma'\setminus\{0\}) \subset \mathbb{N}^{1+s}$ for the set of tuples $(a,\vec{\gamma})$ with $a \geq 1$ and $0 \neq \vec{\gamma} \in \Sigma'$.

\begin{definition}\label{def:nonGor_semigroups}
    For $k \geq 1$ let $\Sigma(L\vee A_k)=\{0\} \cup \left((\mathbb{N}\setminus\{0\})\times(\Sigma(A_k)\setminus\{0\})\right)$, which lies in $\mathbb{N}^2$ if $k$ is even and in $\mathbb{N}^3$ if $k$ is odd, and let 
    \begin{equation*}
        \Theta=\{(0,0),(1,2)\}\cup((2,3)+\mathbb{N}^2) \subset \mathbb{N}^2.
    \end{equation*}
    The non-Gorenstein semigroups of finite Cohen-Macaulay type are $\lceil 3,4,5\rfloor$, $\lceil 3,5,7\rfloor$, $\Sigma(L\vee A_k)$ for $k\geq 1$, and $\Theta$.
\end{definition}

Explicitly, $\Sigma(L\vee A_{2k})=\{0\}\cup\{(a,b) \colon a\geq 1,\ b \in \lceil 2,2k+1\rfloor \setminus\{0\}\}$ and $\Sigma(L\vee A_{2m+1})=\{0\}\cup\{(a,j,j)\colon a,j\geq 1\} \cup \{(a,b,c) \colon a \geq 1,\ b,c\geq m+1\}$. In particular $\Sigma(L \vee A_1)=\{0\}\cup(\mathds{1}+\mathbb{N}^3)$ is the semigroup of the three coordinate axes (Remark~\ref{rem:omp_nu}). The sets are drawn in Figures~\ref{fig:nonGor_domains} through~\ref{fig:LA2m1} in Appendix~\ref{app:figures}. The name $L \vee A_k$ is explained by the following lemma, as these are the semigroups of the rings obtained by gluing a line to an $A_k$ singularity at its singular point:

\begin{lemma}\label{lem:fiber_product}
    Assume that, after renumbering the minimal primes, $\frac{R}{\mathfrak{p}_1}=V_1$ is a DVR and $R=\{(a,c) \in V_1 \times A \colon \overline{a}=\overline{c}\}$, where $A=\frac{R}{\mathfrak{q}_1}$ and $\overline{a}, \overline{c}\in\kappa$ denote residue classes. Then $\nu(R)=\{0\}\cup\left((\mathbb{N}\setminus\{0\})\times(\nu(A)\setminus\{0\})\right)$.
\end{lemma}

\begin{proof}
    The maximal ideal of $R$ is $\pi_1V_1 \times \mathfrak{m}_A$, so the non zero divisors of $R$ which are not units are exactly the pairs $(a,c)$ with $0 \neq a \in \pi_1V_1$ and $c \in \mathfrak{m}_A$ a non zero divisor of $A$, and their values are $(\ord(a,V_1),\nu(c))$.
\end{proof}

\begin{theorem}\label{thm:nonGor}
    Assume that $R$ is not Gorenstein. Then $R$ has finite Cohen-Macaulay type if and only if $\nu(R)$ is, up to a permutation of the coordinates, one of the non-Gorenstein semigroups of finite Cohen-Macaulay type of Definition~\ref{def:nonGor_semigroups}. More precisely, if $R$ has finite Cohen-Macaulay type then $\mult(R)=\edim(R)=3$, and, with $B=R^{(1)}$ as in Lemma~\ref{lem:minmult}, exactly one of the following holds:
    \begin{enumerate}
        \item $r(R)=1$, and either $B=\overline{R}$ and $\nu(R)=\lceil 3,4,5\rfloor$, or $\Gamma(B)=\lceil 2,3 \rfloor$ and $\nu(R)=\lceil 3,5,7\rfloor$;
        \item $r(R)=2$, $B$ is not local, and, ordering the minimal primes so that $\frac{R}{\mathfrak{p}_1}$ is regular, $R=\{(a,c) \in V_1\times C \colon \overline{a}=\overline{c}\}$ where $C=\frac{R}{\mathfrak{p}_2}$ has multiplicity $2$; then $\Gamma(C)=\Sigma(A_{2k})$ for some $k \geq 1$ and $\nu(R)=\Sigma(L\vee A_{2k})$;
        \item $r(R)=2$, $B$ is local, and $B$ is a double point; then $\nu(R)=\Theta$;
        \item $r(R)=3$ and, after renumbering the minimal primes, $R=\{(a,c) \in V_1\times A \colon \overline{a}=\overline{c}\}$ where $A=\frac{R}{\mathfrak{q}_1}$ has two regular branches with intersection multiplicity $m+1\geq 1$; then $\nu(R)=\Sigma(L\vee A_{2m+1})$. Here $m=0$ if and only if $\delta(R)=r(R)-1$, i.e. $R$ is the three coordinate axes (Proposition~\ref{prop:axes}).
    \end{enumerate}
    Conversely, for every separably closed field $\kappa$ each of these semigroups is the semigroup $\nu$ of a ring of finite Cohen-Macaulay type which is not Gorenstein: namely of $\kappa[[t^3,t^4,t^5]]$ and $\kappa[[t^3,t^5,t^7]]$, of the rings
    \begin{gather*}
        \{(a,b) \in \kappa[[s]] \times \kappa[[t^2,t^{2k+1}]] \colon a(0)=b(0)\}, \\
        \{(a,b,c) \in \kappa[[s]] \times \kappa[[t]]^2 \colon a(0)=b(0)=c(0),\ b \equiv c \bmod t^{m+1}\}
    \end{gather*}
    for $\Sigma(L\vee A_{2k})$ and $\Sigma(L\vee A_{2m+1})$, respectively, and of the subring 
    \begin{multline*}
        \kappa[[(s,t^2),(s^2,0),(0,t^3)]]\\=\left\{\left(\sum_j a_js^j,\sum_j d_jt^j\right) \in \kappa[[s]]\times\kappa[[t]] \colon a_0=d_0,\ d_1=0,\ a_1=d_2\right\}
    \end{multline*}
    of $\kappa[[s]]\times\kappa[[t]]$ for $\Theta$.
\end{theorem}

\begin{proof}
    Assume that $R$ has finite Cohen-Macaulay type. By Lemma~\ref{lem:nonGor_minmult}, $\mult(R)=\edim(R)=3$, so Lemma~\ref{lem:minmult} applies; let $x$, $\vec{m}$, and $B$ be as there. By the fourth item of Lemma~\ref{lem:minmult} and the second item of Proposition~\ref{prop:CMtype_onedim}, the image $\widetilde{B}$ of $B$ in the $3-$dimensional $\kappa-$vector space $\frac{\overline{R}}{x\overline{R}}=\prod_i\frac{V_i}{\pi_i^{m_i}V_i}$ has dimension at least $2$. Note that $x \in R \subset B$, and that a local factor of $B$ whose image in some $V_i$ contains an element of value $1$ has $V_i$ as the corresponding branch (a residually rational local domain whose semigroup contains $1$ is normal). We consider the three possibilities for $\vec{m}$ given by Lemma~\ref{lem:nonGor_minmult}.\\

    If $\vec{m}=(3)$, then $B$ is a local domain with $3=\ord(x) \in \Gamma(B)$. Filtering by values, $\dim_\kappa(\widetilde{B})=|\Gamma(B)\cap\{0,1,2\}|$, so $\Gamma(B)$ contains $1$ or $2$, i.e. $\mult(B)\leq 2$. If $\mult(B)=1$ then $B=\overline{R}$ and $\nu(R)=\{0\}\cup(3+\mathbb{N})=\lceil 3,4,5\rfloor$. If $\mult(B)=2$, then $\Gamma(B)=\lceil 2,2j+1\rfloor$ for some $j\geq 1$ (proof of Theorem~\ref{thm:FCMT_DOm}), and as $3 \in \Gamma(B)$ we get $\Gamma(B)=\lceil 2,3\rfloor$ and $\nu(R)=\{0,3\}\cup(5+\mathbb{N})=\lceil 3,5,7\rfloor$ by Lemma~\ref{lem:minmult}. This agrees with Theorem~\ref{thm:FCMT_DOm} and Corollary~\ref{cor:FCMT_Gor_domain}.\\

    If $\vec{m}=(1,2)$, then $\frac{R}{\mathfrak{p}_1}=V_1$ is regular, $C=\frac{R}{\mathfrak{p}_2}$ has multiplicity $2$, and $\Gamma(C)=\lceil 2,2k+1\rfloor$ for some $k\geq 1$ (proof of Theorem~\ref{thm:FCMT_DOm}). Assume first that $B$ is not local, so $B=B_1\times B_2$ with $B_j$ local with one branch. As $\ord(x_1,V_1)=1$ we have $B_1=V_1$, and so $\mathfrak{m}=xB=\pi_1V_1\times x_2B_2$. Since $x_2B_2$ is the image of $\mathfrak{m}$ in $V_2$, i.e. the maximal ideal $\mathfrak{m}_C$ of $C$, and $\frac{R}{\mathfrak{m}}=\kappa$, every pair $(a,c)\in V_1\times C$ with $\overline{a}=\overline{c}$ lies in $R$ (subtract an element of $R$ with the same residue), so $R=\{(a,c) \in V_1\times C \colon \overline{a}=\overline{c}\}$, and $\nu(R)=\Sigma(L\vee A_{2k})$ by Lemma~\ref{lem:fiber_product}. Assume now that $B$ is local. Then $\widetilde{B}$ is a local subalgebra of $\frac{V_1}{\pi_1V_1}\times\frac{V_2}{\pi_2^2V_2}\cong \kappa \times \frac{\kappa[\varepsilon]}{\langle \varepsilon^2\rangle}$, so its maximal ideal consists of nilpotent elements and lies in $0 \times \kappa\varepsilon$. As $\dim_\kappa(\widetilde{B})\geq 2$, there is some $b \in B$ with image $(0,u\varepsilon)$, $u \neq 0$, i.e. $\ord(b_1,V_1)\geq 1$ and $\ord(b_2,V_2)=1$. Thus the images of $B$ in $V_1$ and in $V_2$ contain elements of value $1$ (namely $x_1$ and $b_2$), so $B$ is a local ring with two regular branches, and by Lemma~\ref{lem:two_smooth}, $\nu(B)=\Sigma(A_{2n-1})$ where $n$ is the intersection multiplicity of the two branches of $B$. Since $x \in B$ is a non zero divisor with $\nu(x)=(1,2)$, Lemma~\ref{lem:two_smooth} gives $\min(1,n)=\min(2,n)$, so $n=1$: by Proposition~\ref{prop:axes} $B$ is a double point, $\nu(B)=\{0\}\cup((1,1)+\mathbb{N}^2)$, and $\nu(R)=\{0\}\cup((1,2)+\nu(B))=\Theta$ by Lemma~\ref{lem:minmult}.\\

    If $\vec{m}=(1,1,1)$, then all branches are regular, $x\overline{R}$ is the Jacobson radical of $\overline{R}$, and $\widetilde{B}=\frac{B}{B \cap x\overline{R}}$ is the product of the residue fields of the local factors of $B$ (as $\overline{R}$ is integral over $B$, the Jacobson radical of $B$ is $B \cap x\overline{R}$), each of which is $\kappa$. Thus $B$ has $s \geq 2$ local factors. If $s=3$, then each factor contains an element of value $1$ (a coordinate of $x$), so $B=\overline{R}$ and $\mathfrak{m}=x\overline{R}$ is the Jacobson radical of $\overline{R}$; by Lemma~\ref{lem:milnor_nonneg} and Proposition~\ref{prop:axes}, $\delta(R)=r(R)-1$, $R$ is the three coordinate axes, $\frac{R}{\mathfrak{q}_1}$ is a double point (i.e. its two branches meet with intersection multiplicity $1$), and $\nu(R)=\{0\}\cup(\mathds{1}+\mathbb{N}^3)=\Sigma(L\vee A_1)$; also $R=\{(a,c) \in V_1\times \frac{R}{\mathfrak{q}_1} \colon \overline{a}=\overline{c}\}$, as $\mathfrak{m}=\pi_1V_1 \times \mathfrak{m}_{R/\mathfrak{q}_1}$. If $s=2$, we renumber the minimal primes so that $B=B_1\times B_{23}$, where $B_1$ has the first branch and $B_{23}$ the other two. As before $B_1=V_1$, and the images of $B_{23}$ in $V_2$ and $V_3$ contain elements of value $1$ (coordinates of $x$), so $B_{23}$ is a local ring with two regular branches, and by Lemma~\ref{lem:two_smooth}, $\nu(B_{23})=\Sigma(A_{2m-1})$ for some $m\geq 1$. As in the case $\vec{m}=(1,2)$, $\mathfrak{m}=\pi_1V_1 \times \mathfrak{m}_A$ where $A=\frac{R}{\mathfrak{q}_1}$ and $\mathfrak{m}_A=(x_2,x_3)B_{23}$, so $R=\{(a,c) \in V_1\times A \colon \overline{a}=\overline{c}\}$. The non zero divisors of $A$ in $\mathfrak{m}_A$ are the elements $(x_2,x_3)b$ with $b \in B_{23}$ a non zero divisor, so 
    \begin{equation*}
        \nu(A)=\{0\}\cup((1,1)+\Sigma(A_{2m-1}))=\lceil (1,1)\rfloor\cup((m+1,m+1)+\mathbb{N}^2)=\Sigma(A_{2m+1}).
    \end{equation*}
    The ring $A$ has two regular branches, so by Lemma~\ref{lem:two_smooth} (comparing conductor vectors) they meet with intersection multiplicity $m+1$, and $\nu(R)=\Sigma(L\vee A_{2m+1})$ by Lemma~\ref{lem:fiber_product}. In this case $\delta(R) \neq r(R)-1$, as $B \neq \overline{R}$.\\

    Conversely, assume that $\nu(R)$ is one of the sets of Definition~\ref{def:nonGor_semigroups}. If $r(R)=1$, then $R$ has finite Cohen-Macaulay type and is not Gorenstein by Theorem~\ref{thm:FCMT_DOm}, Remark~\ref{rem:noncomplete_domains}, and Corollary~\ref{cor:FCMT_Gor_domain}. Otherwise, the minimal element $\vec{m}$ of $\nu(R)\setminus\{0\}$ is $(1,2)$ or $(1,1,1)$, so $\mult(R)=3$. Let $h_R$ be as in Lemma~\ref{lem:lengths}. By Lemma~\ref{lem:DR2_general}, the second Drozd--Ro\u{\i}ter condition holds if and only if the image $W$ of $\mathfrak{m}$ in $\frac{\mathfrak{m}\overline{R}}{\mathfrak{m}^2\overline{R}}$ has dimension at least $2$. Since $\mathfrak{m}^2\overline{R}=\{z \in \overline{R} \colon \nu(z) \geq 2\vec{m}\}$, we have $\mathfrak{m}\cap\mathfrak{m}^2\overline{R}=R(2\vec{m})$, and so $\dim_\kappa(W)=h_R(2\vec{m})-1$, which we compute using Lemma~\ref{lem:lengths}. By Remark~\ref{rem:conductor} and Lemma~\ref{lem:lengths}, $R$ is Gorenstein if and only if $\sum_ic_i=2\delta(R)=2\sum_ic_i-2h_R(\vec{c}(R))$, where $\vec{c}(R)$ is determined by $\nu(R)$ by Lemma~\ref{lem:cond_min}.\\

    Assume that $\nu(R)=\Sigma(L\vee A_k)$, and write $\Sigma'=\Sigma(A_k)$, a subset of $\mathbb{N}^s$ with $s \in \{1,2\}$. By Lemma~\ref{lem:ADE_nu_grp} there is a ring $A$ (satisfying our standing assumptions) with $\nu(A)=\Sigma'$; it has multiplicity $2$, so it has finite Cohen-Macaulay type and is Gorenstein. We compare the criterion of Lemma~\ref{lem:lengths} for $R$ and for $A$. At a point $(0,\vec{\gamma})$, a step in one of the last $s$ coordinates has length $1$ for $R$ if and only if it has length $1$ for $A$ at $\vec{\gamma}$: if $\vec{\gamma}=0$ both hold (take $\vec{\beta}=0$), and if $\vec{\gamma}\neq 0$ the required elements of $\nu(R)$ are exactly the elements $(a,\vec{\beta}')$ with $a\geq 1$ and $\vec{\beta}' \in \Sigma'$ as required for $A$. A step in the first coordinate from $(a,\vec{\gamma})$ has length $1$ if $a \geq 1$ (as $\Sigma'$ contains elements which are at least $\vec{\gamma}$), and from $(0,\vec{\gamma})$ it has length $1$ if and only if $\vec{\gamma}=0$. Hence, for $a \geq 1$ and $\vec{\gamma}\neq 0$,
    \begin{equation*}
        h_R(a,\vec{\gamma})=h_A(\vec{\gamma})+a-1.
    \end{equation*}
    Writing $\vec{m}=(1,\vec{m}')$, we get $\dim_\kappa(W)=h_R(2,2\vec{m}')-1=h_A(2\vec{m}')$, which is $1+\dim_\kappa(W_A)\geq 2$, where $W_A$ is the image of $\mathfrak{m}_A$ in $\frac{\mathfrak{m}_A\overline{A}}{\mathfrak{m}_A^2\overline{A}}$, as $A$ satisfies the second Drozd--Ro\u{\i}ter condition (Lemma~\ref{lem:DR2_general}). Moreover, $\vec{c}(R)=(1,\vec{c}(A))$, since $\vec{\alpha}+\mathbb{N}^{1+s}\subset \nu(R)$ if and only if $\alpha_1\geq 1$ and $\vec{\alpha}'+\mathbb{N}^{s}\subset \Sigma'\setminus\{0\}$. Thus $\delta(R)=1+\sum_i c_i(A)-h_A(\vec{c}(A))=1+\delta(A)$, and $\sum_i c_i(R)=1+2\delta(A)=2\delta(R)-1$, as $A$ is Gorenstein; so $R$ is not Gorenstein.\\

    Assume that $\nu(R)=\Theta$, so $\vec{m}=(1,2)$ and $\vec{c}(R)=(2,3)$. Using Lemma~\ref{lem:lengths} along the chain $(0,0),(0,1),\dots,(0,4),(1,4),(2,4)$, the steps have lengths $1,0,1,1$ (realized by $(0,0)$, none, $(1,2)$, and $(2,3)$) and then $0,0$ (no element of $\Theta$ has first coordinate $0$ or $1$ and second coordinate at least $4$). Thus $h_R(2,4)=3$ and $\dim_\kappa(W)=2$. Similarly $h_R(2,3)=2$, so $\delta(R)=5-2=3$ and $\sum_ic_i=5 \neq 2\delta(R)$; thus $R$ is not Gorenstein.\\

    In both cases $\mult(R)=3$ and $\dim_\kappa(W)\geq 2$, so $R$ has finite Cohen-Macaulay type by Lemma~\ref{lem:DR2_general} and the second item of Proposition~\ref{prop:CMtype_onedim}. Finally, the rings listed in the statement are complete Noetherian local rings with residue field $\kappa$ (they are images of power series rings over $\kappa$ in finitely many variables), they are reduced and residually rational, and so they satisfy our standing assumptions. The first two have $\Gamma=\lceil 3,4,5\rfloor$ and $\lceil 3,5,7\rfloor$. The next two are of the form of Lemma~\ref{lem:fiber_product}, with $A$ equal to $\kappa[[t^2,t^{2k+1}]]$, respectively to $\{(b,c) \colon b\equiv c \mod t^{m+1}\}$, which has two regular branches with intersection multiplicity $m+1$; so their semigroups are $\Sigma(L\vee A_{2k})$ and $\Sigma(L\vee A_{2m+1})$ by Lemmas~\ref{lem:fiber_product} and~\ref{lem:two_smooth}. For the last ring, the generators $u=(s,t^2)$, $(s^2,0)$, $(0,t^3)$ give $u^2-(s^2,0)=(0,t^4)$, and so the subring they generate contains $s^2\kappa[[s]]\times t^3\kappa[[t]]$ and equals $\kappa+\kappa u+(s^2\kappa[[s]]\times t^3\kappa[[t]])$, which is the ring described in the statement. Its non units which are non zero divisors are $\lambda u+z$ with $z \in s^2\kappa[[s]]\times t^3\kappa[[t]]$, which have value $(1,2)$ if $\lambda \neq 0$, while the values of the elements $z$ fill $(2,3)+\mathbb{N}^2$; so its semigroup is $\Theta$. By the converse just proved, all of these rings have finite Cohen-Macaulay type and are not Gorenstein.
\end{proof}

\begin{remark}
\begin{enumerate}
    \item \textup{By Lemma~\ref{lem:nonGor_minmult} and the proof of Lemma~\ref{lem:gor_minmult}, every non-Gorenstein ring of finite Cohen-Macaulay type has minimal multiplicity $3$, and so its Cohen-Macaulay type (the dimension of the socle of $\frac{R}{\langle x\rangle}$ for a superficial element $x$) is $\mult(R)-1=2$. Correspondingly, in all cases of Theorem~\ref{thm:nonGor} we have $\sum_i c_i=2\delta(R)-1$.}
    \item \textup{In characteristic zero, the rings in Theorem~\ref{thm:nonGor} are, by~\cite{greuel1984einfache}, the non-Gorenstein rings birationally dominating an ADE singularity; for example, the rings of the second and fourth items are overrings of $D_{2k+3}$ and of $D_{2m+4}$, and the ring of the third item is an overring of $E_7$ (compare Corollary~\ref{cor:containment}). Theorem~\ref{thm:nonGor} does not use this classification, and holds in every characteristic.}
\end{enumerate}
\end{remark}

Combining Theorem~\ref{thm:ADE_curve} and Theorem~\ref{thm:nonGor}, we obtain a complete description of the semigroups of the rings of finite Cohen-Macaulay type.

\begin{corollary}\label{cor:full_classification}
    The ring $R$ has finite Cohen-Macaulay type if and only if $\nu(R)$ is, up to a permutation of the coordinates, one of 
    \begin{gather*}
        \Sigma(A_k)\ (k\geq 0),\ \Sigma(D_k)\ (k\geq 4),\ \Sigma(E_6),\ \Sigma(E_7),\ \Sigma(E_8),\\ \lceil 3,4,5\rfloor,\ \lceil 3,5,7\rfloor,\ \Sigma(L\vee A_k)\ (k\geq 1),\ \Theta,
    \end{gather*}
    where $\Sigma(A_0)=\mathbb{N}$ corresponds to the case where $R$ is a DVR. In this case $R$ is Gorenstein if and only if $\nu(R)$ is one of the sets $\Sigma(A_k)$, $\Sigma(D_k)$, $\Sigma(E_6)$, $\Sigma(E_7)$, $\Sigma(E_8)$. In particular, whether $R$ has finite Cohen-Macaulay type, and whether it is Gorenstein, depends only on $\nu(R)$.
\end{corollary}

\begin{proof}
    If $R$ has finite Cohen-Macaulay type then $\nu(R)$ is in the list by Theorem~\ref{thm:ADE_curve} if $R$ is Gorenstein and not regular, by Theorem~\ref{thm:nonGor} if $R$ is not Gorenstein, and $\nu(R)=\mathbb{N}$ if $R$ is regular. Conversely, if $\nu(R)$ is one of the non-Gorenstein semigroups of Definition~\ref{def:nonGor_semigroups}, then $R$ has finite Cohen-Macaulay type and is not Gorenstein by Theorem~\ref{thm:nonGor}. If $\nu(R)=\Sigma(X)$ for an ADE singularity $X$, then $R$ is Gorenstein: by Lemma~\ref{lem:lengths}, whether $R$ is Gorenstein depends only on $\nu(R)$, and the rings of the second item of Lemma~\ref{lem:ADE_nu_grp} are Gorenstein with semigroup $\Sigma(X)$. Hence $R$ has finite Cohen-Macaulay type by Theorem~\ref{thm:ADE_curve} (or, if $X=A_0$, as $R$ is a DVR).
\end{proof}

Finally, we obtain a semigroup version of the theorem of Greuel and Kn\"orrer~\cite{greuel1984einfache} that rings of finite Cohen-Macaulay type are exactly those dominating an ADE singularity, which holds in every characteristic. For domains, it is the equivalence of the first and third items of Theorem~\ref{thm:FCMT_DOm}.

\begin{corollary}\label{cor:containment}
    The ring $R$ has finite Cohen-Macaulay type if and only if $\nu(R)$ contains, after a permutation of the coordinates, the semigroup $\Sigma(X)$ of an ADE singularity $X$.
\end{corollary}

\begin{proof}
    Assume that $R$ has finite Cohen-Macaulay type. By Corollary~\ref{cor:full_classification}, either $\nu(R)$ is some $\Sigma(X)$, or (after a permutation) it is one of the non-Gorenstein semigroups, and we have $\lceil 3,4,5\rfloor \supset \Sigma(E_6)$, $\lceil 3,5,7\rfloor \supset \Sigma(E_8)$, $\Theta \supset \Sigma(E_7)$, $\Sigma(L\vee A_{2k})\supset \Sigma(D_{2k+3})$, and $\Sigma(L\vee A_{2m+1})\supset\Sigma(D_{2m+4})$, directly from Definitions~\ref{def:ADE_semigroups} and~\ref{def:nonGor_semigroups}. For example, the last inclusion holds since $(1,m+1,l)$ and $(1,l,m+1)$ lie in $\Sigma(L\vee A_{2m+1})$ for every $l\geq m+2$, as $(m+1,l) \in \Sigma(A_{2m+1})$.\\

    Conversely, assume that $\Sigma(X)\subset \nu(R)$ after renumbering the minimal primes, and let $\vec{m}$ and $\vec{m}_X$ be the minimal elements of $\nu(R)\setminus\{0\}$ and of $\Sigma(X)\setminus\{0\}$. Since $\vec{m}$ is the componentwise minimum of $\nu(R)\setminus\{0\}$ (Proposition~\ref{prop:nu_beasic_prop}), we have $\vec{m}\leq \vec{m}_X$, so $\mult(R)\leq \mult(X)\leq 3$. If $\mult(R)\leq 2$, then $R$ has finite Cohen-Macaulay type by the fourth item of Proposition~\ref{prop:CMtype_onedim} (or as $R$ is a DVR). Otherwise $\mult(R)=\mult(X)=3$ and $\vec{m}=\vec{m}_X$. Let $A$ be a ring as in the second item of Lemma~\ref{lem:ADE_nu_grp} with $\nu(A)=\Sigma(X)$; it has finite Cohen-Macaulay type. Since $\Sigma(X) \subset \nu(R)$, every step of length $1$ for $A$ in Lemma~\ref{lem:lengths} also has length $1$ for $R$, and so, computing along the same chain, $h_R(2\vec{m}) \geq h_A(2\vec{m})$. As in the proof of Theorem~\ref{thm:nonGor}, the image $W$ of $\mathfrak{m}$ in $\frac{\mathfrak{m}\overline{R}}{\mathfrak{m}^2\overline{R}}$ has dimension $h_R(2\vec{m})-1$, and similarly for $A$, which has dimension at least $2$ by Lemma~\ref{lem:DR2_general}. Hence $\dim_\kappa(W)\geq 2$, and $R$ has finite Cohen-Macaulay type by Lemma~\ref{lem:DR2_general} and the second item of Proposition~\ref{prop:CMtype_onedim}.
\end{proof}

\appendix
\section{Figures of the ADE Semigroups}\label{app:figures}

In this appendix we draw the semigroups of the ADE singularities of Definition~\ref{def:ADE_semigroups} (see Remark~\ref{rem:ADE_Gamma}). In each picture, the marked points (in red) are the elements of the semigroup.

\begin{figure}[ht]
\centering

\begin{tikzpicture}[scale=0.9]

    \draw (0,0) -- (11,0);

    \foreach \x in {0,1,2,3,4,5} {
        \draw (\x,0.1) -- (\x,-0.1)
              node[below=3pt] {\scriptsize $\x$};
    }

    \draw (7,0.1) -- (7,-0.1)
          node[below=3pt] {\scriptsize $2n$};

    \draw (8,0.1) -- (8,-0.1)
          node[below=3pt] {\scriptsize $2n+1$};

    \foreach \x in {8,9,10} {
        \draw (\x,0.1) -- (\x,-0.1);
    }

    \fill[red] (2,0) circle (2pt);
    \fill[red] (4,0) circle (2pt);
    \fill[red] (7,0) circle (2pt); 
    \fill[red] (0,0) circle (2pt);

    \node at (6,0.35) {\scriptsize $\cdots$};

    \fill[red] (8,0) circle (2pt);
    \fill[red] (9,0) circle (2pt);
    \fill[red] (10,0) circle (2pt);

    \draw[black,->] (10.3,0) -- (11,0);

\end{tikzpicture}

\caption{The semigroup $\Sigma(A_{2n})$.}\label{fig:A2n}
    \end{figure}

    \begin{figure}[ht]
    \centering
    
    \begin{tikzpicture}[scale=0.8]
    
        \draw[->] (0,0) -- (9,0) node[right] {};
        \draw[->] (0,0) -- (0,9) node[above] {};
    
        \foreach \i in {1,2} {  
            \draw (\i,0.1) -- (\i,-0.1) node[below=2pt] {\scriptsize $\i$};
            \draw (0.1,\i) -- (-0.1,\i) node[left=2pt] {\scriptsize $\i$};
        }
    
        \draw (3,0.1)  (3,-0.1) node[below=2pt] {\scriptsize $\cdots$};
        \draw (0,3)  (-0.1,3) node[left=2pt] {\scriptsize $\vdots$};

        \draw (5,0.05) -- (5,-0.05) node[below=2pt] {\scriptsize $n+1$};
        \draw (0.1,5) -- (-0.1,5) node[left=2pt] {\scriptsize $n+1$};
    
        \draw (4,0.1) -- (4,-0.1) node[below=2pt] {\scriptsize $n$};
        \draw (0.1,4) -- (-0.1,4) node[left=2pt] {\scriptsize $n$};
    
        \draw (6,0.1) (6,-0.1) node[below=2pt] {\scriptsize $\cdots$};
        \draw (0.1,6)(-0.1,6) node[left=2pt] {\scriptsize $\vdots$};
    
        \foreach \i in {0,1,2,4} {  
            \fill[red] (\i,\i) circle (2pt);
        }
    
        \node at (3,3) {\scriptsize $\iddots$};
    
        \foreach \i in {5,6,7,8} {
            \foreach \j in {5,6,7,8} {
                \fill[red] (\i,\j) circle (2pt);
            }
        }
    
    \end{tikzpicture}
    
    \caption{The semigroup $\Sigma(A_{2n+1})$.}\label{fig:A2n1}
    \end{figure}

\begin{figure}[ht]
\centering
\begin{tikzpicture}[x={(0.74cm,0.26cm)}, y={(1cm,0cm)}, z={(0cm,1cm)}, scale=0.7]
    \draw[->] (0,0,0) -- (5.4,0,0) node[above right=-2pt] {\scriptsize $a$};
    \draw[->] (0,0,0) -- (0,10.4,0) node[right] {\scriptsize $b$};
    \draw[->] (0,0,0) -- (0,0,10.4) node[above] {\scriptsize $c$};
    \foreach \i/\l in {1/1,2/2,4/n-1,5/n} {
        \draw (0,\i,0.12) -- (0,\i,-0.12) node[below=1pt] {\scriptsize $\l$};
        \draw (0,0.12,\i) -- (0,-0.12,\i) node[left=1pt] {\scriptsize $\l$};
    }
    \draw[gray!50, densely dashed] (1,10,0) -- (1,0,0) -- (1,0,10);
    \node[gray, right] at (1,10,0) {\scriptsize $a=1$};
    \fill[red!7] (2,9.6,5) -- (4.8,9.6,5) -- (4.8,9.6,9.6) -- (2,9.6,9.6) -- cycle;  
    \fill[red!7] (2,5,9.6) -- (4.8,5,9.6) -- (4.8,9.6,9.6) -- (2,9.6,9.6) -- cycle;  
    \fill[red!16] (2,5,5) -- (2,9.6,5) -- (2,9.6,9.6) -- (2,5,9.6) -- cycle;        
    \draw[red!40, dashed] (2,5,5) -- (4.8,5,5);                                       
    \draw[red!40, dashed] (4.8,5,5) -- (4.8,9.6,5);
    \draw[red!40, dashed] (4.8,5,5) -- (4.8,5,9.6);
    \draw[red!40] (2,9.6,5) -- (4.8,9.6,5) (2,5,9.6) -- (4.8,5,9.6) (2,9.6,9.6) -- (4.8,9.6,9.6);
    \draw[red!70] (2,9.6,5) -- (2,5,5) -- (2,5,9.6);
    \foreach \t in {6,7,8,9} { \fill[red] (2,\t,5) circle (2.4pt); \fill[red] (2,5,\t) circle (2.4pt); }
    \fill[red] (2,5,5) circle (3pt);
    \node[above right=1pt, fill=white, inner sep=1pt, rounded corners=1pt] at (2,5.05,5.05) {\scriptsize $(2,n,n)$};
    \fill[red] (0,0,0) circle (3pt);
    \foreach \k in {1,2,4} {
        \draw[red!50, densely dotted] (1,\k,\k) -- (4.8,\k,\k);
        \foreach \a in {1,2,3,4} { \fill[red] (\a,\k,\k) circle (2.4pt); }
    }
    \foreach \t in {2.75,3,3.25} { \fill (1,\t,\t) circle (0.6pt); }
    \foreach \l in {5,6,7,8,9} { \fill[red] (1,4,\l) circle (2.4pt); \fill[red] (1,\l,4) circle (2.4pt); }
    \draw[red!50, densely dotted] (1,4,9) -- (1,4,10);
    \draw[red!50, densely dotted] (1,9,4) -- (1,10,4);
\end{tikzpicture}
\caption{The semigroup $\Sigma(D_{2n})$ in three dimensions. The points $(a,k,k)$ with $a\geq 1$ and $1\leq k\leq n-1$ lie on lines parallel to the $a-$axis (drawn for $k=1,2,n-1$), the rays $(1,n-1,l)$ and $(1,l,n-1)$, $l\geq n$, lie in the plane $a=1$, and the shaded box is the octant $(2,n,n)+\mathbb{N}^3$ (truncated), of which only the points on the two edges of its face $a=2$ through the corner $(2,n,n)$ are marked. Compare the slices in Figure~\ref{fig:D2n}.}\label{fig:D2n_3d}
\end{figure}

\begin{figure}[ht]
\centering
\begin{tikzpicture}[scale=0.6]
    \draw[->] (0,0) -- (9,0) node[right] {\scriptsize $b$};
    \draw[->] (0,0) -- (0,9) node[above] {\scriptsize $c$};
    \foreach \i in {1,2} {
        \draw (\i,0.1) -- (\i,-0.1) node[below=2pt] {\scriptsize $\i$};
        \draw (0.1,\i) -- (-0.1,\i) node[left=2pt] {\scriptsize $\i$};
    }
    \draw (4,0.1) -- (4,-0.1) node[below=2pt] {\scriptsize $n-1$};
    \draw (0.1,4) -- (-0.1,4) node[left=2pt] {\scriptsize $n-1$};
    \draw (5,0.1) -- (5,-0.1);
    \draw (0.1,5) -- (-0.1,5);
    \node at (3,3) {\scriptsize $\iddots$};
    \foreach \i in {1,2,4} { \fill[red] (\i,\i) circle (3pt); }
    \foreach \j in {5,6,7,8} { \fill[red] (4,\j) circle (3pt); \fill[red] (\j,4) circle (3pt); }
    \node at (4.5,-1.5) {\scriptsize slice $a=1$};
\end{tikzpicture}
\hspace{1cm}
\begin{tikzpicture}[scale=0.6]
    \draw[->] (0,0) -- (9,0) node[right] {\scriptsize $b$};
    \draw[->] (0,0) -- (0,9) node[above] {\scriptsize $c$};
    \foreach \i in {1,2} {
        \draw (\i,0.1) -- (\i,-0.1) node[below=2pt] {\scriptsize $\i$};
        \draw (0.1,\i) -- (-0.1,\i) node[left=2pt] {\scriptsize $\i$};
    }
    \draw (4,0.1) -- (4,-0.1) node[below=2pt] {\scriptsize $n-1$};
    \draw (5,0.1) -- (5,-0.1) node[below=8pt] {\scriptsize $n$};
    \draw (0.1,4) -- (-0.1,4) node[left=2pt] {\scriptsize $n-1$};
    \draw (0.1,5) -- (-0.1,5) node[left=2pt] {\scriptsize $n$};
    \node at (3,3) {\scriptsize $\iddots$};
    \foreach \i in {1,2,4} { \fill[red] (\i,\i) circle (3pt); }
    \foreach \i in {5,6,7,8} { \foreach \j in {5,6,7,8} { \fill[red] (\i,\j) circle (3pt); } }
    \node at (4.5,-1.5) {\scriptsize slices $a\geq 2$};
\end{tikzpicture}
\caption{The semigroup $\Sigma(D_{2n})$, by slices $\{(a,b,c)\colon a \text{ fixed}\}$ (the slice $a=0$ is the origin).}\label{fig:D2n}
\end{figure}

    \begin{figure}[ht]
\centering

\begin{tikzpicture}[scale=0.8]

    \draw[->] (0,0) -- (11,0) node[right] {};
    \draw[->] (0,0) -- (0,11) node[above] {};

    \foreach \i in {1,2,3,4} {  
        \draw (0.1,\i) -- (-0.1,\i) node[left=2pt] {\scriptsize $\i$};
    }

    \draw (1,0.1) -- (1,-0.1) node[below=2pt] {\scriptsize $1$};
    \draw (2,0.1) -- (2,-0.1) node[below=2pt] {\scriptsize $2$};
    
    \draw (5,0.1)  (3,-0.1) node[below=2pt] {\scriptsize $\cdots$};
    \draw (0,5)  (-0.1,3) node[left=2pt] {\scriptsize $\vdots$};

    \draw (0.1,5.5)  (-0.1,5.5) node[left=2pt] {\scriptsize $\vdots$};
    \draw (0.1,7) -- (-0.1,7) node[left=2pt] {\scriptsize $2n-2$};
    \draw (0.1,8) -- (-0.1,8) node[left=2pt] {\scriptsize $2n-1$};
    \draw (0.1,9) -- (-0.1,9) node[left=2pt] {\scriptsize $2n$};
    \draw (0.1,10)  (-0.1,10) node[left=2pt] {\scriptsize $\vdots$};

    \foreach \i in {7,8}{
        \draw (0.1,\i) -- (-0.1,\i) node[left=2pt] { };
        }

    \foreach \i in {2,4, 7} { 
        \foreach \j in {1,2,3,4,5,6,7,8,9,10}{
             \fill[red] (\j,\i) circle (2pt);
            }
    }
    \fill[red] (0,0) circle (2pt);
    \fill[red] (1,8) circle (2pt);

    \foreach \i in {2,3,4,5,6,7,8,9,10} {
        \foreach \j in {9,10} {
            \fill[red] (\i,\j) circle (2pt);
        }
    }

\end{tikzpicture}

\caption{The semigroup $\Sigma(D_{2n+1})$.}\label{fig:D2n1}
\end{figure}

\begin{figure}[ht]
\centering

\begin{tikzpicture}[scale=0.8]

    \draw[->] (0,0) -- (9,0) node[right] {};
    \draw[->] (0,0) -- (0,10) node[above] {};

    \foreach \i in {1,2,3,4,5,6,7,8} {  
        \draw (0.1,\i) -- (-0.1,\i) node[left=2pt] {\scriptsize $\i$};
    }

    \foreach \i in {1,2,3} {  
        \draw (\i,0.1) -- (\i,-0.1) node[below=2pt] {\scriptsize $\i$};
    }

    \draw (0.1,9) (-0.1,9) node[left=2pt] {\scriptsize $\vdots$};

    \draw (4,0.1)  (4,-0.1) node[below=2pt] {\scriptsize $\cdots$};

    \foreach \i in {2,3,4,5,6,7,8} {
        \fill[red] (\i,3) circle (2pt);
    } 

    \foreach \i in {0,1,2,4} {  
        \fill[red] (\i,2*\i) circle (2pt);
    }

    \foreach \i in {3,4,5,6,7,8} {
        \foreach \j in {5,6,7,8,9} {
            \fill[red] (\i,\j) circle (2pt);
        }
    }

\end{tikzpicture}

\caption{The semigroup $\Sigma(E_7)$.}\label{fig:E7}
\end{figure}

\begin{figure}[ht]
\centering

\begin{tikzpicture}[scale=0.9]
    \draw[->] (0,0) -- (11.5,0);

    \foreach \x in {0,1,...,11} {
        \draw (\x,0.1) -- (\x,-0.1) node[below=3pt] {\scriptsize $\x$};
    }

    \foreach \x in {0,3,4,6,7,8,9,10,11} {
        \fill[red] (\x,0) circle (2pt);
    }

    \node at (5.75,-1) {\scriptsize    };
\end{tikzpicture}
\hfill
\begin{tikzpicture}[scale=0.9]
    \draw[->] (0,0) -- (11.5,0);

    \foreach \x in {0,1,...,11} {
        \draw (\x,0.1) -- (\x,-0.1) node[below=3pt] {\scriptsize $\x$};
    }

    \foreach \x in {0,3,5,6,8,9,10,11} {
        \fill[red] (\x,0) circle (2pt);
    }

    \node at (5.75,-1) {\scriptsize  };
\end{tikzpicture}

\caption{The semigroups $\Sigma(E_6)$ (top) and $\Sigma(E_8)$ (bottom).}\label{fig:E68}
\end{figure}

%

\begin{figure}[ht]
\centering
\begin{tikzpicture}[scale=0.9]
    \draw[->] (0,0) -- (11.5,0);
    \foreach \x in {0,1,...,11} { \draw (\x,0.1) -- (\x,-0.1) node[below=3pt] {\scriptsize $\x$}; }
    \foreach \x in {0,3,4,5,6,7,8,9,10,11} { \fill[red] (\x,0) circle (2pt); }
\end{tikzpicture}

\bigskip
\begin{tikzpicture}[scale=0.9]
    \draw[->] (0,0) -- (11.5,0);
    \foreach \x in {0,1,...,11} { \draw (\x,0.1) -- (\x,-0.1) node[below=3pt] {\scriptsize $\x$}; }
    \foreach \x in {0,3,5,6,7,8,9,10,11} { \fill[red] (\x,0) circle (2pt); }
\end{tikzpicture}

\caption{The semigroups $\lceil 3,4,5\rfloor$ (top) and $\lceil 3,5,7\rfloor$ (bottom).}\label{fig:nonGor_domains}
\end{figure}

\begin{figure}[ht]
\centering
\begin{tikzpicture}[scale=0.8]
    \draw[->] (0,0) -- (9,0) node[right] {\scriptsize $a$};
    \draw[->] (0,0) -- (0,10) node[above] {\scriptsize $b$};
    \foreach \i in {1,2,3} { \draw (\i,0.1) -- (\i,-0.1) node[below=2pt] {\scriptsize $\i$}; }
    \draw (4,0.1) (4,-0.1) node[below=2pt] {\scriptsize $\cdots$};
    \foreach \i in {1,2,3,4} { \draw (0.1,\i) -- (-0.1,\i) node[left=2pt] {\scriptsize $\i$}; }
    \draw (0,5) (-0.1,5) node[left=2pt] {\scriptsize $\vdots$};
    \draw (0.1,6) -- (-0.1,6) node[left=2pt] {\scriptsize $2k$};
    \draw (0.1,7) -- (-0.1,7) node[left=2pt] {\scriptsize $2k+1$};
    \draw (0.1,8) -- (-0.1,8);
    \draw (0.1,9) -- (-0.1,9);
    \fill[red] (0,0) circle (2pt);
    \foreach \i in {1,...,8} {
        \foreach \j in {2,4,6,7,8,9} { \fill[red] (\i,\j) circle (2pt); }
        \node at (\i,5.05) {\scriptsize $\vdots$};
    }
\end{tikzpicture}
\caption{The semigroup $\Sigma(L\vee A_{2k})$.}\label{fig:LA2k}
\end{figure}

\begin{figure}[ht]
\centering
\begin{tikzpicture}[scale=0.8]
    \draw[->] (0,0) -- (9,0) node[right] {\scriptsize $a$};
    \draw[->] (0,0) -- (0,10) node[above] {\scriptsize $b$};
    \foreach \i in {1,...,8} { \draw (\i,0.1) -- (\i,-0.1) node[below=2pt] {\scriptsize $\i$}; }
    \foreach \i in {1,...,9} { \draw (0.1,\i) -- (-0.1,\i) node[left=2pt] {\scriptsize $\i$}; }
    \fill[red] (0,0) circle (2pt);
    \fill[red] (1,2) circle (2pt);
    \foreach \i in {2,...,8} { \foreach \j in {3,...,9} { \fill[red] (\i,\j) circle (2pt); } }
\end{tikzpicture}
\caption{The semigroup $\Theta$.}\label{fig:Theta}
\end{figure}

\begin{figure}[ht]
\centering
\begin{tikzpicture}[x={(0.74cm,0.26cm)}, y={(1cm,0cm)}, z={(0cm,1cm)}, scale=0.7]
    \draw[->] (0,0,0) -- (4.6,0,0) node[above right=-2pt] {\scriptsize $a$};
    \draw[->] (0,0,0) -- (0,10.4,0) node[right] {\scriptsize $b$};
    \draw[->] (0,0,0) -- (0,0,10.4) node[above] {\scriptsize $c$};
    \foreach \i/\l in {1/1,2/2,4/m,5/m+1} {
        \draw (0,\i,0.12) -- (0,\i,-0.12) node[below=1pt] {\scriptsize $\l$};
        \draw (0,0.12,\i) -- (0,-0.12,\i) node[left=1pt] {\scriptsize $\l$};
    }
    \draw[gray!50, densely dashed] (1,10,0) -- (1,0,0) -- (1,0,10);
    \node[gray, right] at (1,10,0) {\scriptsize $a=1$};
    \fill[red!7] (1,9.6,5) -- (3.8,9.6,5) -- (3.8,9.6,9.6) -- (1,9.6,9.6) -- cycle;
    \fill[red!7] (1,5,9.6) -- (3.8,5,9.6) -- (3.8,9.6,9.6) -- (1,9.6,9.6) -- cycle;
    \fill[red!16] (1,5,5) -- (1,9.6,5) -- (1,9.6,9.6) -- (1,5,9.6) -- cycle;
    \draw[red!40, dashed] (1,5,5) -- (3.8,5,5);
    \draw[red!40, dashed] (3.8,5,5) -- (3.8,9.6,5);
    \draw[red!40, dashed] (3.8,5,5) -- (3.8,5,9.6);
    \draw[red!40] (1,9.6,5) -- (3.8,9.6,5) (1,5,9.6) -- (3.8,5,9.6) (1,9.6,9.6) -- (3.8,9.6,9.6);
    \draw[red!70] (1,9.6,5) -- (1,5,5) -- (1,5,9.6);
    \foreach \t in {6,7,8,9} { \fill[red] (1,\t,5) circle (2.4pt); \fill[red] (1,5,\t) circle (2.4pt); }
    \fill[red] (1,5,5) circle (3pt);
    \node[above right=1pt, fill=white, inner sep=1pt, rounded corners=1pt] at (1,5.05,5.05) {\scriptsize $(1,m+1,m+1)$};
    \fill[red] (0,0,0) circle (3pt);
    \foreach \k in {1,2,4} {
        \draw[red!50, densely dotted] (1,\k,\k) -- (3.8,\k,\k);
        \foreach \a in {1,2,3} { \fill[red] (\a,\k,\k) circle (2.4pt); }
    }
    \foreach \t in {2.75,3,3.25} { \fill (1,\t,\t) circle (0.6pt); }
\end{tikzpicture}
\caption{The semigroup $\Sigma(L\vee A_{2m+1})$ (drawn for $m\geq 3$ as for $m=0$ only the octant remains and $\Sigma(L\vee A_1)=\{0\}\cup(\mathds{1}+\mathbb{N}^3)$ is the semigroup of the three coordinate axes).}\label{fig:LA2m1}
\end{figure}

\clearpage

\bibliographystyle{alpha}
\bibliography{bib}

\end{document}